\documentclass{amsart}

\usepackage{amsthm, amssymb, amsmath, stmaryrd, mathrsfs}

\input xy
\xyoption{all}

\newcommand{\quash}[1]{}  
\newcommand{\leftexp}[2]{{\vphantom{#2}}^{#1}{#2}}

\newcommand{\const}[1]{\overline{\QQ}_{\ell,#1}}
\newcommand{\twtimes}[1]{\stackrel{#1}{\times}}

\newcommand{\homo}[2]{\mathbf{H}_{#1}({#2})}   
\newcommand{\homog}[2]{\textup{H}_{#1}({#2})}  
\newcommand{\coho}[2]{\mathbf{H}^{#1}({#2})}    
\newcommand{\cohog}[2]{\textup{H}^{#1}({#2})}     
\newcommand{\hBM}[2]{\textup{H}^{\textup{BM}}_{#1}({#2})}  
\newcommand{\jiao}[1]{\langle{#1}\rangle}

\DeclareMathOperator{\SL}{SL}
\DeclareMathOperator{\PGL}{PGL}

\DeclareMathOperator{\Sp}{Sp}
\DeclareMathOperator{\Hom}{Hom}
\DeclareMathOperator{\Ext}{Ext}
\DeclareMathOperator{\End}{End}
\DeclareMathOperator{\Aut}{Aut}
\DeclareMathOperator{\Corr}{Corr}

\DeclareMathOperator{\Frob}{Frob}
\DeclareMathOperator{\Spec}{Spec}

\DeclareMathOperator{\Res}{Res}

\DeclareMathOperator{\id}{id}
\DeclareMathOperator{\Tr}{Tr}
\DeclareMathOperator{\ev}{ev}
\DeclareMathOperator{\Irr}{Irr}
\DeclareMathOperator{\Pic}{Pic}

\DeclareMathOperator{\Supp}{Supp}

\DeclareMathOperator{\codim}{codim}

\DeclareMathOperator{\Ch}{Ch}

\def\bI{\mathbf{I}}

\def\bR{\mathbf{R}}

\def\AA{\mathbb{A}}
\def\BB{\mathbb{B}}
\def\CC{\mathbb{C}}
\def\DD{\mathbb{D}}

\def\FF{\mathbb{F}}
\def\GG{\mathbb{G}}

\def\PP{\mathbb{P}}
\def\QQ{\mathbb{Q}}

\def\ZZ{\mathbb{Z}}

\def\frg{\mathfrak{g}}
\def\frt{\mathfrak{t}}
\def\frh{\mathfrak{h}}

\def\frc{\mathfrak{c}}

\def\frr{\mathfrak{r}}

\def\frR{\mathfrak{R}}

\def\calA{\mathcal{A}}
\def\calB{\mathcal{B}}
\def\calC{\mathcal{C}}
\def\calE{\mathcal{E}}
\def\calF{\mathcal{F}}
\def\calG{\mathcal{G}}
\def\calH{\mathcal{H}}
\def\calI{\mathcal{I}}

\def\calL{\mathcal{L}}
\def\calM{\mathcal{M}}
\def\calN{\mathcal{N}}
\def\calO{\mathcal{O}}
\def\calP{\mathcal{P}}
\def\calQ{\mathcal{Q}}
\def\calR{\mathcal{R}}

\def\tilh{\widetilde{h}}
\def\tilm{\widetilde{m}}
\def\tilp{\widetilde{p}}
\def\tils{\widetilde{s}}
\def\tilw{\widetilde{w}}
\def\tilW{\widetilde{W}}
\def\tila{\widetilde{a}}
\def\tilv{\widetilde{v}}
\def\tilx{\widetilde{x}}
\def\tilj{\widetilde{j}}
\def\tiliota{\widetilde{\iota}}

\def\tilxi{\widetilde{\xi}}

\def\tilf{\widetilde{f}}
\def\tilep{\widetilde{\epsilon}}
\def\tiltau{\widetilde{\tau}}
\def\tilmu{\widetilde{\mu}}
\def\tilnu{\widetilde{\nu}}
\def\tilinf{\widetilde{\infty}}
\def\tilDel{\widetilde{\Delta}}

\def\tcP{\widetilde{\calP}}
\def\tcA{\widetilde{\calA}}

\def\tcB{\widetilde{\calB}}
\def\tcC{\widetilde{\calC}}

\def\hatT{\widehat{T}}
\def\hatG{\widehat{G}}
\def\hatH{\widehat{H}}
\def\hatB{\widehat{B}}

\def\hatx{\widehat{x}}

\def\unEnd{\underline{\End}}
\def\unHom{\underline{\Hom}}

\def\unW{\underline{W}}
\def\untilW{\underline{\tilW}}

\def\xch{\mathbb{X}^*}
\def\xcoch{\mathbb{X}_*}

\def\Ql{\overline{\QQ}_\ell}

\def\isom{\stackrel{\sim}{\to}}

\def\pH{\leftexp{p}{\mathbf{H}}}
\def\ptau{\leftexp{p}{\tau}}

\def\Out{\textup{Out}}
\def\act{\textup{act}}
\def\rs{\textup{rs}}
\def\reg{\textup{reg}}
\def\red{\textup{red}}
\def\mult{\textup{mult}}

\def\Ad{\textup{Ad}}
\def\ad{\textup{ad}}

\def\nil{\textup{nil}}

\def\st{\textup{st}}

\def\parab{\textup{par}}

\def\Poin{\textup{Poin}}

\def\Bun{\textup{Bun}}
\def\Hit{\textup{Hit}}
\def\Grass{\mathcal{G}r}
\def\tilGr{\widetilde{\Grass}}

\def\tilGr{\widetilde{\Grass}}

\def\Bunpar{\Bun^{\parab}}

\def\stPic{\calP\textup{ic}}
\def\cPic{\overline{\Pic}}

\def\ani{\textup{ani}}
\def\Ah{\calA^{\heartsuit}}
\def\Aa{\calA^{\ani}}
\def\AHit{\calA^{\Hit}}
\def\Ag{\calA^{\diamondsuit}}
\def\tcAg{\tcA^{\diamondsuit}}
\def\tcArs{\tcA^{\rs}}
\def\tcAgrs{\tcA^{\diamondsuit,\rs}}

\def\Mpar{\mathcal{M}^{\parab}}
\def\Mparrs{\mathcal{M}^{\parab,\rs}}
\def\Mparreg{\mathcal{M}^{\parab,\reg}}
\def\MHit{\mathcal{M}^{\Hit}}
\def\MHitreg{\mathcal{M}^{\Hit,\reg}}

\def\Hecke{\mathcal{H}\textup{ecke}}
\def\Heckep{\Hecke^{\parab}}
\def\Heckeprs{\Hecke^{\parab,\rs}}

\def\fpar{f^{\parab}}
\def\fHit{f^{\Hit}}
\def\fQl{f^{\parab}_*\Ql}

\def\tfQl{\widetilde{f}_*\Ql}
\def\tfHQl{\widetilde{f}_{H,*}\Ql}

\def\Cmod{\calC^{\textup{mod}}}
\def\tcCmod{\tcC^{\textup{mod}}}

\def\Wa{W_{\textup{aff}}}

\def\VD{\textup{VD}}
\def\PD{\textup{PD}}
\def\LP{\calL^{\textup{Poin}}}

\def\dual{^\vee}
\def\Gd{G\dual}
\def\Td{T\dual}
\def\frtd{\frt\dual}
\def\frcd{\frc\dual}
\def\calPd{\calP\dual}
\def\calPdc{\calP^{\vee,0}}
\def\tcPd{\tcP\dual}
\def\tfd{\widetilde{f}^{\vee}}
\def\Md{M\dual}
\def\Mde{M^{\vee,\textup{ev}}}
\def\Mdo{M^{\vee,\textup{od}}}

\def\Kg{K_{\diamondsuit}}
\def\Lg{L_{\diamondsuit}}

\def\VP{V_\ell(\calP^0/\Ag)}
\def\VPd{V_\ell(\calPdc/\Ag)}

\swapnumbers 
\theoremstyle{plain}
\newtheorem{theorem}[subsubsection]{Theorem}
\newtheorem{lemma}[subsubsection]{Lemma}
\newtheorem{cor}[subsubsection]{Corollary}
\newtheorem{prop}[subsubsection]{Proposition}

\newtheorem*{tha}{Theorem A}
\newtheorem*{thb}{Theorem B}
\newtheorem*{thc}{Theorem C}

\theoremstyle{definition}
\newtheorem{defn}[subsubsection]{Definition}
\newtheorem{cons}[subsubsection]{Construction}
\newtheorem{remark}[subsubsection]{Remark}

\numberwithin{equation}{section}

\title[Towards a Global Springer Theory III]{Towards a Global Springer Theory III:\\ Endoscopy and Langlands duality}
\author{Zhiwei Yun}
\address{Department of Mathematics, Princeton University, Princeton, NJ 08544, USA}
\email{zyun@math.princeton.edu}
\date{April 2009}
\subjclass[2000]{Primary 14H60, 20G35; Secondary 14F20, 14K30}

\begin{document}

\begin{abstract}
We prove three new results about the global Springer action defined in \cite{GSI}. The first one determines the support of the perverse cohomology sheaves of the parabolic Hitchin complex, which serves as a technical tool for the next results. The second one (the Endoscopic Decomposition Theorem) links certain direct summands of the parabolic Hitchin complex of $G$ to the endoscopic groups of $G$. This result generalizes Ng\^o's geometric stabilization of the trace formula in \cite{NgoFL}. The third result links the stable parts of the parabolic Hitchin complexes for Langlands dual groups, and establishes a relation between the global Springer action on one hand and certain Chern class action on the other. This result is inspired by the mirror symmetry between dual Hitchin fibrations. Finally, we present the first nontrivial example in the global Springer theory.
\end{abstract}

\maketitle

\tableofcontents

\section{Introduction}
This paper is a continuation of \cite{GSI} and \cite{GSII}, but logically independent of \cite{GSII} for the most part. For an overview of the ideas and motivations of this series of papers, see the Introduction of \cite{GSI}. We will use the notations and conventions from \cite[Sec. 2]{GSI}. In particular, we fix a connected reductive group $G$ over an algebraically closed field $k$ with a Borel subgroup $B$, a connected smooth projective curve $X$ over $k$ and a divisor $D$ on $X$ of degree at least twice the genus of $X$. Recall from \cite[Def. 3.1.2]{GSI} that we defined the parabolic Hitchin moduli stack $\Mpar=\Mpar_{G,X,D}$ as the moduli stack of quadruples $(x,\calE,\varphi,\calE^B_x)$ where
\begin{itemize}
\item $x\in X$;
\item $\calE$ is a $G$-torsor on $X$ with a $B$-reduction $\calE^B_x$ at $x$;
\item $\varphi\in\cohog{0}{X,\Ad(\calE)(D)}$ is a Higgs field compatible with $\calE^B_x$.
\end{itemize}
We also defined the (enhanced) parabolic Hitchin fibration (see \cite[Def. 3.1.6]{GSI}):
\begin{equation*}
\fpar:\Mpar\xrightarrow{\tilf}\tcA\xrightarrow{q}\calA\times X.
\end{equation*}
where $\tcA$ is the universal cameral cover in \cite[Def. 3.1.7]{GSI}.

In \cite{GSI}, we have constructed an action of the extended affine Weyl group $\tilW$ on the parabolic Hitchin complex $\fQl$ (\cite[Th. 4.4.3]{GSI}), and an action of the lattice $\xcoch(T)$ on the enhanced parabolic Hitchin complex $\tfQl$ (\cite[Prop. 4.4.6]{GSI}). This paper studies the decomposition of $\tfQl$ into generalized eigen-subcomplexes according to the $\xcoch(T)$-action. In the course of analyzing these subcomplexes, endoscopy and Langlands duality naturally come into the picture.

The generalized eigen-subcomplexes of $\tfQl$ have two types. The first type have supports on a proper subscheme of $\tcA$, and they can be understood using the endoscopic groups of $G$. This is the content of the Endoscopic Decomposition Theorem (Th. \ref{th:endo}). The second type have supports on the whole $\tcA$, and they are essentially the same as the {\em stable part} of $\tfQl$ (i.e., the part on which the $\xcoch(T)$-action is unipotent, see Def. \ref{def:stpart}). To understand the stable part, we consider the parabolic Hitchin complex for the Langlands dual group $\Gd$, and prove a result of mirror-symmetry style (Th. \ref{th:LD}).

Recall from \cite[Rem. 3.5.6]{GSI} that we have chosen an open subset $\calA$ of the anisotropic Hitchin base $\Aa$ on which the codimension estimate $\codim_{\AHit}(\calA_\delta)\geq\delta$ holds for any $\delta\in\ZZ_{\geq0}$. Throughout this paper, we will work over this open subset $\calA$. All stacks originally over $\AHit$ or $\Aa$ will be restricted to $\calA$ without changing notation. Note that when $\textup{char}(k)=0$, we may take $\calA=\Aa$.

\subsection{Main results}

\subsubsection{The Support Theorem}
The first result is about the perverse cohomology of the enhanced parabolic Hitchin complex $\tfQl$. By the decomposition theorem (see \cite[Th. 6.2.5]{BBD}), $\tfQl$ is {\em non-canonically} a direct sum of shifted perverse sheaves $\oplus_i(\pH^i\tfQl)[-i]$. We would like to understand the supports of the simple constituents of $\pH^i\tfQl$.

Recall from \cite[Sec. 2.3]{GSI} that $(\calA\times X)^{\rs}$ is the locus of $(a,x)$ where the value $a(x)\in\frc$ is regular semisimple, and $\tcArs$ is the preimage of $(\calA\times X)^{\rs}$ in $\tcA$. 

\begin{tha}[the Support Theorem, see Th. \ref{th:supp}]
Any simple constituent of $\oplus_i\pH^i\tfQl$ is the middle extension of its restriction to $\tcArs$; i.e.,
\begin{equation*}
\pH^i\tfQl=\tilj_{!*}\tilj^*\pH^i\tfQl,
\end{equation*}
where $\tilj:\tcArs\hookrightarrow\tcA$ is the open inclusion. Similar result holds for the perverse cohomology sheaves of $\fQl$.
\end{tha}

This theorem is analogous to the following statement in the classical Springer theory: the Springer sheaf $\pi_*\Ql$ is the middle extension of its restriction on $\frg^{\rs}$. A corollary of Th. A is that the action of $\tilW$ on the perverse cohomology sheaves $\pH^i\fQl$ factors through a finite quotient, hence semisimple. However, this is no longer true on the level of complexes: we will see an example in Sec. \ref{sss:ex} where the action of the lattice part on the stalks of the complex $\fQl$ is not semisimple. This non-semisimplicity will be further investigated in Th. C below.

\subsubsection{The Endoscopic Decomposition Theorem}

By \cite[Prop. 4.4.6]{GSI}, the action of the lattice part $\xcoch(T)\subset\tilW$ on $\tfQl$ gives a decomposition of into generalized eigen-subcomplexes:
\begin{equation}\label{eq:introdecomp}
\tfQl=\bigoplus_{\kappa\in\hatT(\Ql)}(\tfQl)_{\kappa}.
\end{equation}
Here $\hatT=\Hom(\xcoch(T),\GG_m)$ is a torus over $\Ql$.

Let us concentrate on one direct summand $(\tfQl)_\kappa$. For a {\em rigidified endoscopic datum} $(\kappa,\rho)$ (see Def. \ref{def:endo}), we can associate a quasi-split reductive group scheme $H$ over $X$, called an {\em endoscopic group} of $G$. Our previous results in the global Springer theory all go through for quasi-split groups such as $H$ (see Sec. \ref{ss:qs}). In particular, we have the enhanced parabolic Hitchin fibration $\tilf_{H}:\Mpar_H\to\tcA_H$ for $H$, and we can consider the complex $\tfHQl$ as well as its $\kappa$-part $(\tfHQl)_{\kappa}$. Now this $\kappa$-part is isomorphic to the stable part $(\tfHQl)_{\st}$ (see Rem. \ref{rm:qsstable}).

The bases $\tcA_{H}$ and $\tcA$ are related by the following diagram
\begin{equation*}
\xymatrix{\tcA_{H,\Theta}=\tcA_H\times_XX_\Theta\ar[r]^{\tilnu_{H,\Theta}}\ar[d]^{\theta_{H}} & \tcA_{\Theta}=\tcA\times_XX_\Theta\ar[d]^{\theta}\\
\tcA_H & \tcA}
\end{equation*}
We briefly clarify our notations here. The morphism $X_\Theta\to X$ is a connected \'etale Galois cover with Galois group $\Theta$. We choose $\Theta$ large enough so that the endoscopic groups $H$ associated to any $(\kappa,\rho)$ are split over $X_\Theta$. Let $W_\kappa$ be the stabilizer of $\kappa$ in $W$, then $W_\kappa$ can be written as a semidirect product $W_H\rtimes\pi_0(\kappa)$, where $W_H$ is the Weyl group of the split form of $H$, and $\pi_0(\kappa)$ is a finite group acting on $H$ by outer automorphisms. The datum of $\rho$ in the definition of the rigidified endoscopic datum is a homomorphism $\rho:\Theta\to\pi_0(\kappa)$.

The complex $\theta_H^*(\tfHQl)_\kappa$ is $\tilW_H\rtimes^{\rho}\Theta=\xcoch(T)\rtimes(W_H\rtimes^{\rho}\Theta)$-equivariant with respect to the $W_H\rtimes^{\rho}\Theta$-action on $\tcA_{H,\Theta}$; the complex $\theta^*(\tfQl)_\kappa$ is $\tilW_\kappa\times\Theta=(\xcoch(T)\rtimes W_\kappa)\times\Theta$-equivariant with respect to the $W_\kappa\times\Theta$-action on $\tcA_\Theta$.

\begin{thb}[the Endoscopic Decomposition Theorem, see Th. \ref{th:endo}]
There is a natural quasi-isomorphism of complexes on $\tcA_\Theta$:
\begin{equation*}
\theta^*(\tfQl)_\kappa\cong\bigoplus_{\rho,H=H_\rho}\tilnu_{H,\Theta,*}\theta_{H}^*(\tfHQl)_{\kappa}[-2r_\kappa](-r_\kappa).
\end{equation*}
where the direct sum runs over all homomorphisms $\rho:\Theta\to\pi_0(\kappa)$ (so that the endoscopic groups $H$ in the summands are determined accordingly) and $r_\kappa=\dim(\fpar)-\dim(\fpar_H)$. Moreover, for each $\rho$, the embedding of the direct summand indexed by $\rho$ in the above isomorphism
\begin{equation*}
\tilnu_{H,\Theta,*}\theta_{H}^*(\tfHQl)_{\kappa}[-2r_\kappa](-r_\kappa)\hookrightarrow\theta^*(\tfQl)_\kappa
\end{equation*}
is $(\tilW_H\rtimes^{\rho}\Theta,\tilW_\kappa\times\Theta)$-equivariant under the embedding
\begin{eqnarray*}
\tiliota_\rho:\tilW_H\rtimes^{\rho}\Theta=\xcoch(T)\rtimes(W_H\rtimes^{\rho}\Theta)&\to& (\xcoch(T)\rtimes W_\kappa)\times\Theta=\tilW_\kappa\times\Theta\\
(\lambda,w_H,\sigma)&\mapsto&(\lambda,w_H\rho(\sigma),\sigma).
\end{eqnarray*}
\end{thb}

On one hand, this theorem is analogous to the Induction Theorems in the classical Springer theory (see \cite{AL},\cite{Tr}); on the other hand, this result is also an analogue of the Geometric Stabilization of the Trace Formula proved by Ng\^o in \cite[Th. 6.5.3]{NgoFL}. Note that our result is a statement on the level of complexes (which contains information about the non-semisimplicity of the affine Weyl group actions), whereas \cite[Th. 6.4.3]{NgoFL} is a statement on the level of perverse cohomology which does not see the non-semisimple actions but suffices for the purpose of proving the Fundamental Lemma.

\subsubsection{Langlands duality}

In Th. B, if $\kappa\in Z\hatG(\Ql)$, we have $H=G$, therefore the theorem is an empty statement. In this case, the $\kappa$-part of $\tfQl$ is essentially the stable part $(\tfQl)_{\st}$ of $\tfQl$ (see Prop. \ref{p:stable}). The name ``stable'' comes from the fact that for the usual Hitchin complex, the Frobenius traces of the stable part give stable orbital integrals. Our next result gives a way to partially understand the non-semisimplicity of the action of $\xcoch(T)$ on $(\tfQl)_{\st}$, using the Langlands dual group $\Gd$ (defined also over $k$, not over $\Ql$).

We can identify the enhanced Hitchin bases $\tcA_G$ and $\tcA_{\Gd}$ using a Killing form on $\frt$. We denote them simply by $\tcA$. Consider two complexes on $\tcA$:
\begin{equation*}
K=(\tfQl)_{\st}[d](d/2);\hspace{1cm}L=(\tfd_*\Ql)_{\st}[d](d/2),
\end{equation*}
where $d$ is the common dimension of $\Mpar_G$ and $\Mpar_{\Gd}$. Let $K^i,L^i$ be the perverse cohomology sheaves of $K$ and $L$.

On one hand, we have the action of the lattice $\xcoch(T)$ on $K$ by \cite[Prop. 4.4.6]{GSI}. Since $K$ is the stable part of $\tfQl$, this action is unipotent (i.e., equals the identity when passing to perverse cohomology). Hence for $\lambda\in\xcoch(T)$, the endomorphism $\lambda-\id$ on $K$ induces a map
\begin{equation*}
\Sp^i(\lambda):=\pH^i(\lambda-\id):K^i\to K^{i-1}[1],
\end{equation*}
which can be viewed as a ``subdiagonal entry'' of the unipotent action $\lambda$ with respect to the perverse filtration on $K$.

On the other hand, we have the action of the lattice $\xch(\Td)$ on $L$ by cup product with the Chern classes $c_1(\calL(\lambda))$, $\lambda\in\xch(\Td)$. These Chern classes are what we used to extend the affine Weyl group action to the DAHA action in \cite[Sec. 3]{GSII}. Passing to perverse cohomology, the induced map
\begin{equation*}
\cup c_1(\calL(\lambda)): L^i\to L^{i+2}(1)
\end{equation*}
is in fact zero (see Lem. \ref{l:chernvan}). Therefore, it makes sense to talk about the ``subdiagonal entries'':
\begin{equation*}
\Ch^i(\lambda):=\pH^i(\cup c_1(\lambda)): L^i\to L^{i+1}[1](1).
\end{equation*}

The following theorem gives a Verdier duality between the perverse sheaves $K^i$ and $L^i$, together with an identification of the above two lattice actions.

\begin{thc}[See Th. \ref{th:VD} and \ref{th:LD}]
For each $i\in\ZZ$, there are natural isomorphisms of perverse sheaves on $\tcA$:
\begin{equation}\label{eq:introVD}
K^{-i}\cong\DD K^i\cong L^i(i).
\end{equation}
Moreover, for each $\lambda\in\xcoch(T)=\xch(\Td)$, we have a commutative diagram
\begin{equation}\label{d:introLD}
\xymatrix{K^{-i}\ar[r]^{\sim}\ar[d]_{\Sp^{-i}(-\lambda)}^{\wr} & \DD K^i\ar[d]^{(\DD\Sp^{i+1}(\lambda))[1]}_{\wr}\ar[rr]^{\sim} & & L^i(i)\ar[d]^{\Ch^i(\lambda)(i)}_{\wr}\\
K^{-i-1}[1]\ar[r]^{\sim} & \DD K^{i+1}[1]\ar[rr]^{\sim} & & L^{i+1}[1](i+1)}
\end{equation}
where the two rows are the isomorphisms in (\ref{eq:introVD}).
\end{thc}

This theorem is inspired by the mirror symmetric approach to the geometric Langlands conjecture. Using the physicists' language (cf. \cite{KW}), the $\xcoch(T)$-action on $K$ is the analogue of the 'tHooft operators on the Hitchin stack $\MHit_G$ while the $\xch(\Td)$-action on $L$ is the analogue of the Wilson operators on the Hitchin stack $\MHit_{\Gd}$. It is expected that under the classical limit of the geometric Langlands correspondence (i.e., an equivalence between the derived categories of coherent sheaves on $\MHit_G$ and $\MHit_{\Gd}$), the 'tHooft and Wilson operators are intertwined. The commutative diagram (\ref{d:introLD}) is the shadow of this expectation on the level of perverse cohomology sheaves, after passing first from the derived categories to the K-groups, and then from the K-groups to cohomology.

\subsubsection{An example}\label{sss:ex}
The first nontrivial example of the global Springer actions is a parabolic Hitchin fiber $M$ for $G=\SL(2)$, which is a union of two $\PP^1$'s intersecting transversally at two points. We write $\tilW=\ZZ\alpha^\vee\rtimes\jiao{s}$ where $\alpha^\vee$ is the unique positive coroot and $s$ is the reflection in $W=S_2$. Then the action of $\alpha^\vee$ and $s$ on $\cohog{2}{M}$, under the natural basis dual to the fundamental classes of the two $\PP^1$'s, are given by the matrices
\begin{equation*}
s_1=\left(\begin{array}{cc}1 & 2 \\0 & -1\end{array}\right);\alpha^\vee=\left(\begin{array}{cc}-1 & -2 \\2 & 3\end{array}\right).
\end{equation*}
In particular, we see that the action of $\alpha^\vee$ is unipotent but not identity (i.e., non-semisimple).

This example, together with its ``Langlands dual'' example, also gives an instance in which the commutative diagram (\ref{d:introLD}) proved Th. C carries nontrivial information. In fact, this example belongs to an interesting class of parabolic Hitchin fibers which are called {\em subregular}. We will study these subregular fibers in more detail in \cite{Subreg}.

\subsection{Applications}
There is an application of the global Springer theory to $p$-adic harmonic analysis, suggested by Ng\^o. For this we work over the ground field $k=\FF_q$ (although in the previous papers we assumed $k$ to be algebraically closed, this assumption is not necessary). 

The motivation of Springer's study of Weyl group representations, as the author understands, is their relationship with Green functions. For each conjugacy class $[w]$ in $W$, there is a Green function $Q_{[w]}$, which is the character value of certain Deligne-Lusztig representations on the unipotent elements of $G(\FF_q)$. For large $\textup{char}(\FF_q)$, we can work with nilpotent elements in the Lie algebra instead of unipotent elements in the group, and view $Q_{[w]}$ as a function on the nilpotent elements of $\frg(\FF_q)$. The work of Springer (\cite{Sp}) and Kazhdan (\cite{Kaz}) shows that
\begin{equation*}
Q_{[w]}(\gamma)=\Tr(\Frob_\gamma\circ w,\pi_*\Ql)
\end{equation*}
for any representative $w\in[w]$ and any nilpotent element $\gamma\in\frg^{\nil}(\FF_q)$.

In the global situation, we can consider the following quantity
\begin{equation}\label{eq:twtr}
\Tr(\Frob_{a,x}\circ\tilw,\fQl)
\end{equation}
for $\tilw\in\tilW$ and $(a,x)\in(\calA\times X)(\FF_q)$. This is essentially a product of local orbital integrals, one of which (at the place $x\in X(\FF_q)$) is the orbital integral of a {\em Deligne-Lusztig function} (which is a compactly supported function on $\frg(F_x)$ inflated from a usual Deligne-Lusztig virtual character). 

Kottwitz (\cite{Kot}) conjectured an identity between orbital integrals of Deligne-Lusztig functions for $G$ and its endoscopic groups. Using Th. B above, and taking twisted Frobenius traces as in (\ref{eq:twtr}), we expect to prove a Lie algebra analogue of Kottwitz's conjecture in the function field case. Note that Kottwitz's conjecture has already been proved by Kazhdan-Varshavsky in \cite{KV} using group-theoretic methods, based on Waldspurger's deep work on the Fourier transform of stable distributions.

Other applications to representation theory are joint work in progress with R.Bezrukavnikov and Y.Varshavsky. 

\subsection{Organization of the paper and remarks on the proofs}
In Sec. \ref{s:supp}, we prove Th. A about the support of simple constituents of $\pH^i(\fQl)$. This is the technical heart of all the subsequent study of $\fQl$. The proof is based on Ng\^o's idea in proving his ``Th\'eor\`eme du support'' in \cite{NgoFL}, which is the key geometric ingredient in his proof of the Fundamental Lemma. We then study the $\kappa$-decomposition of $\tfQl$ in Section \ref{ss:kappadecomp} and \ref{ss:kappaincenter}.

In Sec. \ref{s:endo}, we prove Th. B about the endoscopic decomposition. Before stating the theorem, we first make some remarks about how to generalize our results up to this point to quasi-split group schemes, because they will soon show up as endoscopic groups. The proof of Th. B is a bit lengthy. The complication results from the attempt to establish an isomorphism between {\em complexes}, not just perverse cohomology sheaves. The proof relies on Ng\^o's unpublished results on endoscopic correspondences, which we record in App. \ref{s:endocorr}. Eventually we reduce the proof to a calculation of the endoscopic correspondence over the generic locus, which essentially only involves the geometry of nodal curves. 

In Sec. \ref{s:LD}, we prove Th. C about Langlands duality. For this, we need an explicit description of the Picard stack $\calP$ and its Tate module, which we give in Sec. \ref{ss:PfVD}. This description uses the result of Donagi-Gaitsgory on the regular centralizer group scheme (\cite{DG}). The proof of Th. C has two major ingredients: one is the simple observation that $\Ext^1$ between middle extensions of local systems are determined by the $\Ext^1$ between the local systems; the other is a manipulation of the Abel-Jacobi map for curves.

In Sec. \ref{s:SL2}, we present the calculation of the $\tilW$-action on the cohomology of a subregular parabolic Hitchin fiber for $G=\SL(2)$. We also verify (partially) Th. C in this case in Sec. \ref{ss:pervSL2}.

\subsection*{Acknowledgment}
I would like thank B-C.Ng\^o for allowing me to use his unpublished results. I would like to thank my advisor R.MacPherson who, among many other things, taught me how to actually do calculations with cohomological correspondences. I would also like to thank R.Bezrukavnikov, R.Kottwitz and Y.Varshavsky for helpful discussions.


\section{The Support Theorem and consequences}\label{s:supp}

In this section, we study the perverse cohomology of the parabolic Hitchin complex. The main result is the Support Theorem (Th. \ref{th:supp}), which is a key technical result for later sections. From the Support Theorem and the results of Ng\^o in \cite{NgoFL}, we determine the supports of the $\kappa$-parts of the enhanced parabolic Hitchin complex.

\subsection{The Support Theorem}\label{ss:supp}

Consider the proper map $\tilf:\Mpar\to\tcA$ whose source is a smooth Deligne-Mumford stack. By the decomposition theorem \cite[Th. 6.2.5]{BBD}, we have a {\em non-canonical} decomposition
\begin{equation}\label{eq:simpleps}
\tfQl=\bigoplus_{i\in I}\calF_i[-n_i]
\end{equation}
where $I$ is a finite index set and each $\calF_i$ is a simple perverse sheaf on $\tcA$. 

The Support Theorem we will state is the analogue of the following classical statement: the direct image sheaf $\pi_*\Ql$ of the Grothendieck simultaneous resolution $\pi:\widetilde{\frg}\to\frg$ is the middle extension from its restriction to $\frg^{\rs}$.

Recall the following maps
\begin{equation*}
\xymatrix{\tcA\ar[r]_{q}\ar@/^1pc/[rr]^{\tilp} & \calA\times X\ar[r]_{p} & \calA}
\end{equation*}

\begin{theorem}\label{th:supp}
For each simple perverse sheaf $\calF_i$ appearing in the decomposition (\ref{eq:simpleps}), its support $Z$ is an irreducible component of $\tilp^{-1}(\tilp(Z))$. In particular, $\calF_i$ is the middle perverse extension of its restriction to $\tcA^{rs}$. Equivalently,
\begin{equation*}
\pH^*\tfQl=\tilj^{\rs}_{!*}\tilj^{\rs,*}(\pH^*\tfQl)
\end{equation*}
where $\tilj^{\rs}:\tcArs\hookrightarrow\tcA$ is the open embedding.
\end{theorem}
\begin{proof}
Recall that the global invariant $\delta: \calA\to \ZZ_{\geq0}$ is upper semi-continuous. Let $\delta_Z$ be the generic (minimal) value of $\delta$ on $\tilp(Z)$. Let us recall the notion of the amplitude of $Z$ following \cite[7.2]{NgoFL}. Let $\textup{occ}(Z)=\{n_i|i\in I,\Supp\calF_i=Z\}$. Then the amplitude of $Z$ is defined to be the difference between the largest and smallest elements in $\textup{occ}(Z)$.

Consider the $\tcP$-action on $\Mpar$ over $\tcA$. In this situation, we can apply \cite[Prop. 7.2.3]{NgoFL} to conclude that the amplitude of $Z$ is at least $2(\dim(\calP/\calA)-\delta_Z)$.

Now we can apply the argument of \cite[7.3]{NgoFL} to show that $\codim_{\tcA}(Z)\leq\delta_Z$. For completeness, we briefly reproduce the argument here. By Poincar\'e duality, the set $\textup{occ}(Z)$ is symmetric with respect to $\dim(\Mpar)$. Let $n_+$ be the largest element in $\textup{occ}(Z)$. Since the amplitude of $Z$ is at least $2(\dim(\calP/\calA)-\delta_Z)$, we conclude that $n_+\geq\dim(\Mpar)+\dim(\calP/\calA)-\delta_Z$. Suppose $\calF_j$ has support $Z$ and $n_j=n_+$. Let $U$ be an open dense subset of $Z$ over which $\calF_j$ is a local system placed in degree $-\dim(Z)$. Pick a point $\tila\in U(k)$. Since
\begin{equation*}
0\neq i^*_{\tila}\calF_j\subset \cohog{n_+-\dim(Z)}{\Mpar_{\tila},\Ql},
\end{equation*}
we necessarily have
\begin{equation*}
\dim(\Mpar)+\dim(\calP/\calA)-\delta_Z-\dim(Z)\leq n_+-\dim(Z)\leq2\dim(\Mpar_{\tila})=2\dim(\calP/\calA).
\end{equation*}
This implies that $\codim_{\tcA}(Z)\leq\delta_Z$. 

Recall from \cite[Rem. 3.5.6]{GSI} that $\codim_{\calA}(\calA_\delta)\geq\delta$. This implies that $\codim_{\tcA}(Z)\geq\codim_{\calA}(\tilp(Z))\geq\delta_Z$. Therefore the inequalities must be equalities, i.e.,
\begin{equation*} 
\codim_{\tcA}(Z)=\codim_{\calA}(\tilp(Z))=\delta_Z.
\end{equation*}
This forces that $q(Z)=\tilp(Z)\times X$ and that $Z$ is an irreducible component of $q^{-1}(\tilp(Z)\times X)=\tilp^{-1}(\tilp(Z))$.

Since $(\{a\}\times X)^{\rs}$ is dense in $\{a\}\times X$ for any $a\in\calA(k)$, we conclude that $q(Z)^{\rs}$ is dense in $q(Z)$ and therefore $Z^{\rs}$ is dense in $Z$. Therefore the simple perverse sheaf $\calF_i$ is the middle extension of its restriction on $Z^{\rs}$. This completes the proof.
\end{proof}

\begin{remark} Using the same arguments, we can show that the $\fQl$ decompositions as a direct sum of shifted simple perverse sheaves, and each simple perverse sheaf has support of the form $Z'\times X$, where $Z'$ is an irreducible closed subscheme of $\calA$. In particular,
\begin{equation*}
\pH^*\fQl=j^{\rs}_{!*}j^{\rs,*}(\pH^*\fQl)
\end{equation*}
where $j^{\rs}:(\calA\times X)^{\rs}\hookrightarrow\calA\times X$ is the open embedding.
\end{remark}


\subsection{The $\kappa$-decomposition}\label{ss:kappadecomp}
By \cite[Prop. 4.4.6]{GSI}, there is an action of $\xcoch(T)$ on the enhanced parabolic Hitchin complex $\tfQl$. By \cite[Rem. 4.3.7]{GSI}, over $\tcArs$, this action is induced by the morphism
\begin{equation}\label{eq:latticePsection}
s:\xcoch(T)\times\tcArs\to\tcP\to\calP.
\end{equation}
This morphism induces a surjective homomorphism of sheaves of abelian groups
\begin{equation}\label{eq:latticetopi0}
\pi_0(s):\xcoch(T)\to\pi_0(\tcP/\tcA)=\tilp^*\pi_0(\calP/\calA).
\end{equation}
where $\tilp:\tcA\to\calA$ is the projection and $\xcoch(T)$ stands for the constant sheaf on $\tcA$ with stalks $\xcoch(T)$.

On the other hand, since $\tcP$ acts on $\Mpar$ over $\tcA$, by homotopy invariance (see \cite[Lemme 3.2.3]{LauN}), it induces an action of $\pi_0(\tcP/\tcA)$ on $\pH^i\tfQl$, by the same argument as in \cite[6.2.1]{NgoFL}.

\begin{lemma}\label{l:latticefinite}
For each $i\in\ZZ$, the action of $\xcoch(T)$ on $\pH^i\tfQl$ factors through the $\pi_0(\tcP/\tcA)$-action via $\pi_0(s)$ in (\ref{eq:latticetopi0}). In particular, the action of $\xcoch(T)$ on $\pH^i\tfQl$ factors through a finite quotient, hence semisimple.
\end{lemma}
\begin{proof}
By Th. \ref{th:supp}, it suffices to check this statement over $\tcArs$. Over $\tcArs$, the action of $\xcoch(T)$ comes from the map $s$ in (\ref{eq:latticePsection}), therefore the action of $\xcoch(T)$ on $\tfQl|_{\tcArs}$ factors through the homomorphism $\pi_0(s)$ in (\ref{eq:latticetopi0}). Since $\pi_0(\calP/\calA)$ is a constructible sheaf of {\em finite} abelian groups (see \cite[Def. 3.2.10]{GSI}), the homomorphism $\pi_0(s)$ necessarily factors through a finite quotient of $\xcoch(T)$.
\end{proof}

\begin{remark}
In contrast to the semisimplicity of the $\xcoch(T)$-action on the perverse cohomology sheaves $\pH^i\tfQl$, the action of $\xcoch(T)$ on the ordinary cohomology sheaves $\bR^i\tfQl$ is {\em not} semisimple. We will give an example of this non-semisimplicity in Sec. \ref{s:SL2}.
\end{remark}

By Lem. \ref{l:latticefinite}, it makes sense to decompose the object $\tfQl\in D^b_c(\tcA)$ into generalized eigen-subcomplexes of $\xcoch(T)$:
\begin{equation}\label{eq:kappadecomp}
\tfQl=\bigoplus_{\kappa\in\hatT(\Ql)}(\tfQl)_\kappa.
\end{equation}
where $\hatT=\Hom(\xcoch(T),\GG_m)$ is an algebraic torus over $\Ql$ and $\kappa$ runs over finite order elements in $\hatT(\Ql)$. Each subcomplex $(\tfQl)_\kappa$ is characterized by the property that the action of $\xcoch(T)$ on $\pH^i(\tfQl)_\kappa$ factors through the character $\kappa:\xcoch(T)\to\Ql^\times$. Over $\tcArs$, the action of $\xcoch(T)$ comes from the map $s$ in (\ref{eq:latticePsection}), therefore this decomposition is also the decomposition according to generalized eigen-subcomplexes of $\pi_0(\tcP/\tcA)$. Hence, passing to perverse cohomology sheaves, the decomposition (\ref{eq:kappadecomp}) coincides with the $\kappa$-decomposition defined by Ng\^o in \cite[6.2]{NgoFL}. Similarly, we have a decomposition of $\fQl$ into $\kappa$-summands. 

\begin{defn}\label{def:stpart}
The {\em stable part} $(\tfQl)_{\st}$ of the complex $\tfQl$ is defined as the direct summand in the decomposition (\ref{eq:kappadecomp}) corresponding to $\kappa=1$.
\end{defn}
In other words, $(\tfQl)_{\st}$ is the direct summand of $\tfQl$ on which $\xcoch(T)$ acts unipotently.

Following \cite[6.2.3]{NgoFL}, for $\kappa\in\hatT(\Ql)$, let $\tcArs_{\kappa}$ the locus of $(a,\tilx)\in\tcArs$ such that $\kappa:\xcoch(T)\to\Ql^\times$ factors through the homomorphism $s(a,\tilx):\xcoch(T)\to\pi_0(\calP_a)$. Let $\tcA_\kappa$ be the closure of $\tcArs_\kappa$ in $\tcA$.

\begin{prop}\label{p:kappasupp}
The support of any simple constituent of $\oplus_i\pH^i(\tfQl)_\kappa$ is an irreducible component of $\tcA_\kappa$.
\end{prop}
\begin{proof}
By Th. \ref{th:supp}, it suffices to show that for any simple constituent $\calF$ of $\oplus_i\pH^i(\tfQl)_\kappa|_{\tcArs}$, the support of $\calF$ is an irreducible component of $\tcArs_\kappa$. 

We would like to apply the general result \cite[Corollaire 7.1.18]{NgoFL} to the fibration $\tilf^{\rs}:\Mpar|_{\tcArs}\to\tcArs$ together with the action of $\tcP|_{\tcArs}$. Note that $\tilf^{\rs}:\Mpar|_{\tcArs}\to\tcArs$ is the base change of the usual Hitchin fibration $f^{\Hit}:\MHit\to\calA$, therefore all the conditions in \cite[Corollaire 7.1.18]{NgoFL} are satisfied by the discussion in \cite[7.6]{NgoFL}. We conclude that the support of $\calF$ is an irreducible component of $\tcArs_\kappa$, and the proposition follows.
\end{proof}

Let $W_\kappa$ be the stabilizer of $\kappa$ in $W$. Then according to \cite[Prop. 4.4.6]{GSI}, $(\tfQl)_{\kappa}$ has a natural $\tilW_\kappa:=\xcoch(T)\rtimes W_\kappa$-equivariant structure with respect to the action of $\tilW_\kappa$ on $\tcA_\kappa$ via the quotient $\tilW_{\kappa}\to W_\kappa$. 

\begin{remark}\label{rm:Akappa}
Let $\calA_{|\kappa|}\subset\calA$ be the image of $\tcA_\kappa$. Then, as the notation suggests, $\calA_{|\kappa|}$ only depends on the $W$-orbit $|\kappa|$ of $\kappa\in\hatT(\Ql)$. The finite morphism $q_\kappa:\tcA_\kappa\to\calA_{|\kappa|}\times X$ is invariant under $W_\kappa$. However, in general, $q_\kappa$ is {\em not} a $W_\kappa$-branched cover.
\end{remark}


\subsection{The $\kappa$-part for $\kappa\in Z\hatG(\Ql)$}\label{ss:kappaincenter}
Let $\hatG$ be the connected Langlands dual group of $G$ defined over $\Ql$, then $\hatT$ is a maximal torus in $\hatG$. In this subsection, we concentrate on the direct summands $(\tfQl)_\kappa$ where $\kappa\in Z\hatG(\Ql)$, the center of $\hatG$. Our goal is to show that such $(\tfQl)_\kappa$ is naturally isomorphic to the stable part of $\tfQl$ (see Def. \ref{def:stpart}).

\subsubsection{Decomposition according to connected components} Recall from \cite[Sec. 3.5]{GSII} that the connected components of $\Bun^{\parab}_G$ are naturally parametrized by the finite abelian group $\Omega=\xcoch(T)/\ZZ\Phi^\vee$. For each $\omega\in\Omega$, let $\Mpar_\omega$ be the corresponding component and $\tilf_\omega:\Mpar_\omega\to\tcA$ be the restriction of $\tilf$. Let $e\in\Omega$ be the identity element. By \cite[Lem. 3.4.2]{GSII}, the correspondences $\calH_{\tilw}$ for $\tilw\in\Wa$ preserve the components of $\Mpar$ because the correspondence $\Hecke^{\Bun}_{\leq\tilw}$ preserves the componenets of $\Bunpar_G$; i.e.,
\begin{equation*}
\overleftrightarrow{h_{\tilw}}(\calH_{\tilw})\subset\bigsqcup_{\omega\in\Omega}\Mpar_\omega\times_{\calA\times X}\Mpar_\omega.
\end{equation*}
Therefore, each $\tilf_{\omega,*}\Ql$ has a $\Wa$-equivariant structure compatible with the $W$-action on $\tcA$. In particular, the coroot lattice $\ZZ\Phi^\vee$ acts on each complex $\tilf_{\omega,*}\Ql$, and we similarly have the $\overline{\kappa}$-decomposition according to the characters of $\ZZ\Phi^\vee$, here $\overline{\kappa}\in\hatT^{\ad}(\Ql)$ and $\hatT^{\ad}$ is the image of $\hatT$ in the adjoint form of $\hatG$. Let $(\tilf_{\omega,*}\Ql)_{\st}$ be the part corresponding to $\overline{\kappa}=1\in\hatT^{\ad}(\Ql)$, i.e., the part on which the $\ZZ\Phi^\vee$-action is unipotent. Each $(\tilf_{\omega,*}\Ql)_{\st}$ is direct summand of $\tfQl$. The direct summand $(\tilf_{e,*}\Ql)_{\st}$ carries a $\Wa$-action.

When $\kappa\in Z\hatG(\Ql)$, it is fixed by the $W$-action on $\hatT$, hence $(\tfQl)_\kappa$ carries a $\tilW$-action.

\begin{prop}\label{p:stable} Let $\kappa\in Z\hatG(\Ql)$. Then the projection to the direct summand $(\tilf_{e,*}\Ql)_{\st}$ defines a $\Wa$-equivariant isomorphism
\begin{equation}\label{eq:kappaisstable}
(\tfQl)_\kappa\cong(\tilf_{e,*}\Ql)_{\st}.
\end{equation}
In particular, there is a canonical $\Wa$-equivariant isomorphism
\begin{equation}\label{eq:kappastable}
(\tfQl)_\kappa\cong(\tfQl)_{\st}.
\end{equation}
\end{prop}

\begin{proof}
We first claim that the following two direct summands of $\tfQl$ are the same
\begin{equation}\label{eq:twosummand}
\bigoplus_{\kappa\in Z\hatG(\Ql)}(\tfQl)_\kappa=\bigoplus_{\omega\in\Omega}(\tilf_{\omega,*}\Ql)_{\st}.
\end{equation}
In fact, both sides are the subcomplex of $\tfQl$ on which $\ZZ\Phi^\vee$ acts unipotently.

The projection to the direct summand $(\tilf_{e,*}\Ql)_{\st}$ is clearly equivariant under $\Wa$. To check (\ref{eq:kappaisstable}) is an isomorphism, it suffices to check that it induces an isomorphism on the level of perverse cohomology. The equality (\ref{eq:twosummand}) implies an equality
\begin{equation}\label{eq:twosumperv}
\bigoplus_{\kappa\in Z\hatG(\Ql)}\pH^*(\tfQl)_\kappa=\bigoplus_{\omega\in\Omega}\pH^*(\tilf_{\omega,*}\Ql)_{\st}.
\end{equation}

On one hand, the action of $\xcoch(T)$ on the LHS of (\ref{eq:twosumperv}) factors through the quotient $\Omega$, and acts on each direct summand $\pH^*(\tfQl)_\kappa$ via the character $\kappa$. On the other hand, for $\lambda\in\xcoch(T)$ with image $\omega_1\in\Omega$, we can view $\omega_1$ as an element in $\Omega_{\bI}$ (see \cite[Sec. 3.3]{GSII}), hence an element in $\tilW$. Since $\lambda\Wa=\omega_1\Wa$, the $\lambda$-action permutes the summands on the RHS of (\ref{eq:twosumperv}) in the same way as $\omega_1$ permutes the summands: $\pH^*(\tilf_{\omega,*}\Ql)_{\st}$ is sent to $\pH^*(\tilf_{\omega-\omega_1,*}\Ql)_{\st}$ (the minus sign appears because the action of $\omega_1$ on complexes are given by the {\em pull-back} functor $R_{\omega_1}^*$). In other words, as $\Omega$-modules, we have
\begin{equation}\label{eq:ddd}
\bigoplus_{\omega\in\Omega}\pH^*(\tilf_{\omega,*}\Ql)_{\st}\cong\pH^*(\tilf_{e,*}\Ql)_{\st}\otimes\Ql[\Omega]'
\end{equation}
Here the $\Ql[\Omega]'$ is the group algebra of $\Omega$ with the inverse action of the regular representation of $\Omega$, and the $\Omega$-action on the RHS of (\ref{eq:ddd}) is trivial on $(\tilf_{e,*}\Ql)_{\st}$.
Combining (\ref{eq:twosumperv}) and (\ref{eq:ddd}) we conclude that there is an $\Omega$-equivariant isomorphism
\begin{equation*}
\bigoplus_{\kappa\in Z\hatG(\Ql)}\pH^*(\tfQl)_\kappa\cong\pH^*(\tilf_{e,*}\Ql)_{\st}\otimes\Ql[\Omega]'.
\end{equation*}
In other words,
\begin{equation*}
\pH^*(\tfQl)_\kappa\cong\pH^*(\tilf_{e,*}\Ql)_{\st}\otimes\Ql(\kappa)
\end{equation*}
where $\Ql(\kappa)$ is the one-dimensional summand of $\Ql[\Omega]'$ where $\Omega$ acts through the character $\kappa$. Since the projection $\Ql(\kappa)\hookrightarrow\Ql[\Omega]'\twoheadrightarrow\Ql e$ is an isomorphism, the projection to the direct summand $\pH^*(\tilf_{e,*}\Ql)_{\st}$ also induces an isomorphism
\begin{equation*}
\pH^*(\tfQl)_\kappa\cong\pH^*(\tilf_{e,*}\Ql)_{\st}\otimes\Ql(\kappa)\isom\pH^*(\tilf_{e,*}\Ql)_{\st}.\qedhere
\end{equation*}
\end{proof}

\begin{remark}
The isomorphism $(\tfQl)_\kappa\cong(\tfQl)_{\st}$ in Prop. \ref{p:stable} is $\Wa$-equivariant but {\em not} $\tilW$-equivariant for $\kappa\neq1$. However, the unipotent parts of the actions of $\xcoch(T)$ on $(\tfQl)_\kappa$ and $(\tfQl)_{\st}$ are intertwined under the isomorphism (\ref{eq:kappastable}). In fact, since $\xcoch(T)_{\QQ}=\QQ\Phi^\vee$, the logarithm of the unipotent part of the $\xcoch(T)$-action is determined by its restriction to $\ZZ\Phi^\vee$.
\end{remark}


\section{The Endoscopic Decomposition Theorem}\label{s:endo}
In this section, we study the $\kappa$-part of the enhanced parabolic Hitchin complex using endoscopic groups of $G$. The main result is the Endoscopic Decomposition Theorem (Th. \ref{th:endo}), which reduces the study of such a $\kappa$-part to the study of the stable parts of the enhanced parabolic Hitchin complexes for the endoscopic groups of $G$. We expect to apply this theorem to the harmonic analysis of $p$-adic groups in the future. In the proof, we need the notion of endoscopic correspondences due to Ng\^o (unpublished). For the reader's convenience, we record Ng\^o's results on endoscopic correspondences in App. \ref{s:endocorr}.

\subsection{Remarks on quasi-split groups}\label{ss:qs}
In this subsection, we briefly explain how the results in \cite[Sec. 4]{GSI} generalize to quasi-split reductive group schemes over $X$.

We fix a pinning $(B,T,x_{\alpha})$ of $G$, where $T$ is a maximal torus contained in the Borel $B$ and $x_{\alpha}\in\frg_{\alpha}-\{0\}$ for every simple root $\alpha$. We identify the outer automorphism group $\Out(G)=\Aut(G)/G^{ad}$ with the subgroup of $\Aut(G)$ stabilizing this pinning.

Let $X_\Theta\to X$ be an \'etale Galois cover with Galois group $\Theta$ and let $\rho$ be a homomorphism $\rho:\Theta\to\Out(G)\subset\Aut(G)$. This homomorphism gives a quasi-split form of $G$:
\begin{equation*}
H=G_{\rho}:=X_\Theta\twtimes{\Theta,\rho}G.
\end{equation*}
which is a connected reductive group scheme over $X$ with a Borel subgroup scheme $B_H=X_{\Theta}\twtimes{\Theta,\rho}B$.

The definition of the parabolic Hitchin moduli stack $\Mpar_H$ extends to the case of the quasi-split group scheme $H$ in a straightforward manner. We can also formulate it using the curve $X_\Theta$: the stack $\Mpar_H$ classifies parabolic Hitchin quadruples $(y,\calE_\Theta,\varphi_\Theta,\calE^B_{y})$ for the curve $X_\Theta$ and the constant group $G$, together with a $\Theta$-equivariant structure.

Since $\Out(G)$ acts on $\frt$ and $\frc$, $\Theta$ also acts on them via $\rho$. Let
\begin{eqnarray*}
\frt_{H,D}&:=&X_\Theta\twtimes{\Theta,\rho}\frt_D;\\
\frc_{H,D}&:=&X_\Theta\twtimes{\Theta,\rho}\frc_D.
\end{eqnarray*}
We can define the Hitchin base $\AHit_H$ as the affine space $\cohog{0}{X,\frc_{H,D}}=\cohog{0}{X_\Theta,\frc_D}^{\Theta}$. We can also define the enhanced Hitchin base $\tcA_\Theta$ as the Cartesian product $(\AHit_H\times X)\times_{\frc_{H,D}}\frt_{H,D}$. We also have the parabolic Hitchin fibration
\begin{equation*}
\fpar_H:\Mpar_H\xrightarrow{\tilf_H}\tcA_H\xrightarrow{q_H}\AHit_H\times X.
\end{equation*}

Let $\theta:\tcA_{H,\Theta}=\tcA_{H}\times_XX_\Theta\to\tcA_{H}$ be the projection. Note that $W\rtimes^{\rho}\Theta$ acts on $\tcA_{H,\Theta}$ with GIT quotient equal to $\AHit_H\times X$.

In the following, we restrict to the open subset $\calA_H\subset\calA^{\Hit,\ani}_H$ as in \cite[Rem. 3.5.6]{GSI}. Now we can state the counterpart of \cite[Th. 4.4.3]{GSI} (or rather \cite[Prop. 4.4.6]{GSI}) in the case of quasi-split group schemes.

\begin{theorem}\label{th:qsaction}
There is a natural $\tilW\rtimes^{\rho}\Theta$-equivariant structure on the complex $\theta^*\tfHQl$, compatible with the action of $\tilW\rtimes^{\rho}\Theta$ on $\tcA_{H,\Theta}$ via the quotient $\tilW\rtimes^{\rho}\Theta\twoheadrightarrow W\rtimes^{\rho}\Theta$. Here the $\Theta$-action on $\theta^*\tfHQl$ is the tautological one given by the $\Theta$-torsor $\theta:\tcA_{H,\Theta}\to\tcA_H$.
\end{theorem}

The proof of this theorem is similar to that of \cite[Th. 4.4.3]{GSI}, which we omit here.

\begin{remark}
We can restate the above theorem without passing to the cover $\tcA_{H,\Theta}$ of $\tcA_H$. Let
\begin{equation*}
\unW:=X_\Theta\twtimes{\Theta,\rho}W;\hspace{1cm}\untilW:=X_\Theta\twtimes{\Theta,\rho}\tilW.
\end{equation*}
be finite group schemes over $X$. Then the group scheme $\unW$ acts on $\tcA_H$. Th. \ref{th:qsaction} can then be reformulated as follows: there is a natural $\untilW$-equivariant structure on $\tfHQl$, compatible with the action of $\untilW$ on $\tcA_H$ via the quotient $\untilW\twoheadrightarrow\unW$.
\end{remark}

\begin{remark}\label{rm:qsstable}
We can decompose the complex $\tfHQl$ into the generalized eigen-subcomplexes of $\xcoch(T)$, but with $\Theta$-ambiguity. More precisely, we have
\begin{equation*}
\tfHQl=\bigoplus_{\overline{\kappa}\in\hatT(\Ql)/\Theta}(\tfHQl)_{\overline{\kappa}}
\end{equation*}
where the direct sum is over the $\Theta$-orbits on $\hatT(\Ql)$. In particular, for $\kappa\in\hatT(\Ql)^{\Theta}$, we have a well-defined summand $(\tfHQl)_\kappa$. We also have an analogue of Prop. \ref{p:stable} for quasi-split groups: if $\kappa\in Z\hatH(\Ql)^\Theta$, then $(\tfHQl)_{\kappa}$ is isomorphic to the stable part $(\tfHQl)_{\st}$. In this case, the support of any simple constituent of $\pH^*(\tfHQl)_{\kappa}$ is an irreducible component of $\tcA_H$, by the argument of Prop. \ref{p:kappasupp}.
\end{remark}

\begin{remark}\label{rm:qsHecke}
We finally make a few remarks about the Hecke correspondence for $\Mpar_H$ when $H$ is quasi-split. As in \cite[Sec. 4.1]{GSI}, we can introduce the Hecke correspondence $\Heckep_H$ of $\Mpar_H$ over $\calA_H\times X$, however, over $(\calA_H\times X)^{\rs}$, the reduced structure of $\Hecke_H$ does not necessarily split into a disjoint union of graphs of automorphisms of $\Mparrs_H$. It is more convenient to pass to the $\Theta\times\Theta$-cover $\Heckep_{H,\Theta}:=X_\Theta\times_X\Heckep_{H}\times_XX_\Theta$, which is a correspondence
\begin{equation*}
\xymatrix{&\Heckep_{H,\Theta}\ar[dl]_{\overleftarrow{h_{H,\Theta}}}\ar[dr]^{\overrightarrow{h_{H,\Theta}}} &\\
\Mpar_{H,\Theta}\ar[dr]_{\fHit_{H,\Theta}} & & \Mpar_{H,\Theta}\ar[dl]^{\fHit_{H,\Theta}}\\
& \calA_H\times X}.
\end{equation*}
Here $\Mpar_{H,\Theta}=\Mpar_H\times_XX_\Theta$, and the morphisms
\begin{eqnarray*}
\Heckep_{H,\Theta}\xrightarrow{\overrightarrow{h_{H,\Theta}}}\Mpar_{H,\Theta}\to X_\Theta;\\ \Heckep_{H,\Theta}\xrightarrow{\overrightarrow{h_{H,\Theta}}}\Mpar_{H,\Theta}\to X_\Theta
\end{eqnarray*}
are the projections of $\Heckep_{H,\Theta}=X_\Theta\times_X\Heckep_{H}\times_XX_\Theta$ onto the first and third factors.

Analogous to \cite[Cor. 4.3.8]{GSI}, there is a right action of $\tilW\rtimes^{\rho}\Theta$ on $\Mpar_{H,\Theta}|_{(\calA_H\times X)_0}$, compatible with the $W\rtimes^{\rho}\Theta$-action on $\tcA_{H,\Theta}$, such that the reduced structure of $\Heckeprs_{H,\Theta}$ is the disjoint union of the graphs of the $\tilW\rtimes^{\rho}\Theta$-action. We denote the closures of the graph of $\tilw\in\tilW\rtimes^{\rho}\Theta$ by $\calH_{H,\tilw}$. The cohomological correspondences $[\calH_{H,\tilw}]\in\Corr(\Mpar_{H,\Theta};\Ql,\Ql)$ are used to construct the $\tilW\rtimes^{\rho}\Theta$-action on $\fpar_{H,\Theta,*}\Ql$.
\end{remark}


\subsection{Statement of the theorem}\label{ss:endo}
In Sec. \ref{ss:kappadecomp}, we have written $\tfQl$ into a direct sum of $\kappa$-parts. In Sec. \ref{ss:kappaincenter}, we studied the $\kappa$-parts whose supports are the whole of $\tcA$. In this subsection, we will study the rest $\kappa$-parts, which are supported on proper subschemes $\tcA_\kappa$ of $\tcA$. We will relate the complex $(\tfQl)_{\kappa}$ to the $\kappa$-parts of the parabolic Hitchin complexes $\tfHQl$ of the {\em endoscopic groups} $H$ of $G$, in a way which respects the affine Weyl group actions. We will see from the construction that $\kappa\in Z\hatH(\Ql)^\Theta$ (notation to be introduced later), therefore $(\tfHQl)_\kappa$ is isomorphic to the stable part of $\tfHQl$ by Rem. \ref{rm:qsstable}. In this way, we have reduced the study of the various $\kappa$-parts of $\tfQl$ to the study of the stable parts of the enhanced parabolic Hitchin complexes for $G$ and its endoscopic groups. 

\subsubsection{Endoscopic groups}

Let $\hatG$ be the connected reductive group over $\Ql$ which is Langlands dual to $G$. More precisely, $\hatG$ is a connected reductive group with a pinning $(\hatB,\hatT,\hatx_{\alpha})$, where $\hatT=\Hom(\xcoch(T),\GG_m)$, such that the based root system determined by $(\hatG,\hatB,\hatT)$ is identified with the based coroot system $\Phi^{\vee,+}\subset\Phi^\vee\subset\xcoch(T)$ of $G$. Let $\kappa\in\hatT(\Ql)$ be an element of finite order. Let $\hatG_{\kappa}$ be the centralizer of $\kappa$ in $\hatG$. Let $\hatH=\hatG^{0}_\kappa$ be the identity component of $\hatG_\kappa$, which is a connected reductive group over $\Ql$ containing the maximal torus $\hatT$. The group $\hatH$ inherits a pinning from that of $\hatG$. Let $W_H$ be the Weyl group associated to the pair $(\hatH,\hatT)$. Let $\pi_0(\kappa)$ be the component group of $\hatG_\kappa$. The conjugation action gives a natural homomorphism
\begin{equation}\label{eq:outer}
o_{\kappa}:\pi_0(\kappa)\to\Out(\hatH).
\end{equation}
where the outer automorphism group $\Out(\hatH)$ can be identified with the finite group of automorphisms of $\hatH$ stabilizing the pinning.

Recall that $W_\kappa$ is the stabilizer of $\kappa$ under the $W$-action on $\hatT$. We have an exact sequence of groups (see \cite[Lemme 10.1]{NgoFib})
\begin{equation*}
1\to W_H\to W_\kappa\to \pi_0(\kappa)\to1.
\end{equation*}
We can also identify $\pi_0(\kappa)$ with the subgroup of $W_\kappa$ which stabilizes the pinning of $\hatH$, hence we can write $W_\kappa$ as a semi-direct product $W_\kappa=W_H\rtimes\pi_0(\kappa)$.

As in \cite[Lemme 6.3.6]{NgoFL}, we choose a large enough quotient $\Theta$ of $\pi_1(X,\eta_X)$ ($\eta_X$ is the geometric generic point of $X$) so that any homomorphism $\pi_1(X,\eta_X)\to\pi_0(\kappa)$ factors through $\Theta$. Let $X_\Theta\to X$ be the associated $\Theta$-torsor over $X$.

\begin{defn}[\cite{NgoFL}, Def. 1.8.2, called a ``donn\'e endoscopique point\'ee'' there]\label{def:endo} A {\em rigidified endoscopic datum} is a pair $(\kappa,\rho)$ where $\kappa\in\hatT(\Ql)$ is a finite order element and $\rho:\Theta\to\pi_0(\kappa)$ is a homomorphism. 
\end{defn}

Given a rigidified endoscopic datum $(\kappa,\rho)$, we can form a quasi-split group scheme $H$ over $X$, called the endoscopic group associated to $(\kappa,\rho)$. Recall the construction of $H$. Take $H^{\textup{sp}}$ to be the connected reductive group over $k$ with a pinning which is Langlands dual to that of $\hatH$. The pinnings of $H^{\textup{sp}}$ and $\hatH$ give a canonical isomorphism $\Out(H^{\textup{sp}})\cong\Out(\hatH)$. We identify $\Out(H^{\textup{sp}})$ with the group of automorphisms of $H^{\textup{sp}}$ which stabilize the pinning of $H^{\textup{sp}}$. Therefore $\Theta$ acts on $H^{\textup{sp}}$ via
\begin{equation*}
o_{\kappa}\circ\rho:\Theta\xrightarrow{\rho}\pi_0(\kappa)\xrightarrow{o_{\kappa}}\Out(\hatH)=\Out(H^{\textup{sp}}).
\end{equation*}

\begin{defn}\label{def:endogp} The group scheme
\begin{equation*}
H=H_\rho:=X_{\Theta}\twtimes{\Theta,o_{\kappa}\circ\rho}H^{\textup{sp}}.
\end{equation*} 
is the {\em endoscopic group scheme} over $X$ associated to the rigidified endoscopic datum $(\kappa,\rho)$.

A quasi-split group scheme $H$ over $X$ is said to be {\em relevant to $\kappa$} if it arises as $H_{\rho}$ for some rigidified endoscopic datum $(\kappa,\rho)$.
\end{defn}

\subsubsection{Relating the Hitchin bases}
As remarked in Sec. \ref{ss:qs}, we can define the Hitchin base $\calA_H$, the universal cameral cover $\tcA_{H}$ and the parabolic Hitchin fibration
\begin{equation*}
\fpar_H:\Mpar_H\stackrel{\tilf_{H}}{\longrightarrow}\tcA_{H}\stackrel{q_{H}}{\longrightarrow}\calA_H\times X.
\end{equation*}

The relative dimensions of $\fpar_H$ and $\fpar$ are related by (see \cite[4.4.6]{NgoFL})
\begin{equation}\label{eq:diffr}
r^G_H:=\dim(\fpar)-\dim(\fpar_H)=(\#\Phi-\#\Phi_H)\deg(D)/2.
\end{equation}
where $\Phi$ and $\Phi_H$ are the sets of roots of $G$ and $H$ respectively. Note that $r^G_H$ depends only on $\kappa$ and not on $\rho$. We will simply write $r_\kappa$ for $r^G_H$.

According to \cite[4.15]{NgoFL} and \cite[7.2]{NgoFib}, there is a finite, unramified
morphism $\mu_H:\calA_H\to\calA$ whose image is contained in $\calA_{|\kappa|}$ (see Rem. \ref{rm:Akappa}).

\begin{cons}\label{cons:AHG}
We will relate $\tcA_H$ and $\tcA$ via the spaces $\tcA_{H,\Theta}:=\tcA_H\times_X X_\Theta$ and $\tcA_\Theta:=\tcA\times_XX_\Theta$. We have a Cartesian diagram
\begin{equation*}
\xymatrix{\tcA_{H,\Theta}\ar[r]\ar@/_1.5pc/[dd]_{q_{H,\Theta}}\ar[d] & X_\Theta\times_X\frt_D\ar[d] \\
\tcA_H\ar[r]\ar[d]^{q_H}  & X_\Theta\twtimes{\Theta,\rho}_X\frt_D\ar[d]\ar@{=}[r] & \frt_{H,D}\ar[d]\\
\calA_H\times X\ar[r] & X_\Theta\twtimes{\Theta,\rho}_X(\frt\sslash W_H)_D\ar@{=}[r] & \frc_{H,D}}
\end{equation*}
From this we get a commutative diagram
\begin{equation*}
\xymatrix{\tcA_{H,\Theta}\ar[rr]\ar[d] & & X_\Theta\times_X\frt_D\ar[d]\\
\calA_H\times X\ar[r]^{\mu_H\times\id} & \calA\times X\ar[r]^{\ev} & \frc_D\\}
\end{equation*}
which gives a finite morphism
\begin{equation*}
\tilmu_{H,\Theta}:\tcA_{H,\Theta}\to\tcA_\Theta=(X_\Theta\times_X\frt_D)\times_{\frc_D}(\calA_H\times X).
\end{equation*}
\end{cons} 


\begin{lemma}\label{l:AHdisj}
\begin{enumerate}
\item []
\item The image of $\tilmu_{H,\Theta}$ is contained in $\tcA_{\kappa,\Theta}:=\tcA_\kappa\times_XX_\Theta$.
\item The morphism
\begin{equation*}
\tilmu_{\kappa,\Theta}:\bigsqcup_{\rho:\Theta\to\pi_0(\kappa)}\tcA_{H_{\rho},\Theta}\xrightarrow{\bigsqcup_{\rho}\tilmu_{H_\rho,\Theta}}\tcA_{\kappa,\Theta}
\end{equation*}
is an isomorphism over $\tcArs_{\kappa,\Theta}$.
\end{enumerate}
\end{lemma}
\begin{proof}
The argument of \cite[Prop. 6.3.4]{NgoFL} shows that for each $\rho$ and $H=H_\rho$, the restriction of $\tilmu_{H,\Theta}$ to $\tcArs_{\Theta}$ is a closed embedding. The argument of of \cite[Prop. 6.3.7]{NgoFL} shows that $\tilmu_{H,\Theta}(\tcArs_{H,\Theta})\subset\tcArs_{\kappa,\Theta}$ and that $\tcArs_{\kappa,\Theta}$ is the disjoint union of the images of $\tcArs_{H_\rho,\Theta}$ for different $\rho$. Taking closures, we get the conclusion. 
\end{proof}

We summarize the relation among the various Hitchin bases by the following commutative diagram (for fixed $\rho$ and $H=H_\rho$)
\begin{equation}\label{eq:basetheta}
\xymatrix{\tcA_{H,\Theta}\ar[r]^{\tilmu_{H,\Theta}}\ar[d]^{\theta_{H}} & \tcA_{\kappa,\Theta}\ar[d]^{\theta}\ar@{^{(}->}[r] & \tcA_\Theta\ar[d]\\
\tcA_H\ar[d]^{q_H} & \tcA_\kappa\ar[d]^{q_\kappa}\ar@{^{(}->}[r] & \tcA\ar[d]^{q}\\
\calA_H\times X\ar[r]^{\mu_H\times\id_X} & \calA_{[\kappa]}\times X\ar@{^{(}->}[r] & \calA\times X}
\end{equation}

\begin{remark}\label{rm:Wkeq}
According to Th. \ref{th:qsaction}, we have a $\tilW_H\rtimes^{\rho}\Theta$-equivariant structure on the complex $\theta_H^*\tfHQl$, which is compatible with the action of $W_H\rtimes^{\rho}\Theta$ on $\tcA_{H,\Theta}$. By Rem. \ref{rm:qsstable}, there is a well-defined generalized eigen-subcomplex $(\tfHQl)_{\kappa}$ of $\tfHQl$ since $\kappa\in\hatT(\Ql)^{\Theta}$. Since $\kappa$ is fixed by $W_H\rtimes^{\rho}\Theta$, the subcomplex $(\theta_H^*\tfHQl)_\kappa\subset\theta_H^*\tfHQl$ also has a $\tilW_H\rtimes^{\rho}\Theta$-equivariant structure.

On the other hand, since $(\tfQl)_{\kappa}$ is $\tilW_\kappa$-equivariant, the pull-back $\theta^*(\tfQl)_{\kappa}$ is $\tilW_\kappa\times\Theta$-equivariant, compatible with the $W_\kappa\times\Theta$-action on $\tcA_{\kappa,\Theta}$.

Finally, the morphism $\tilmu_{H,\Theta}$ is equivariant with respect to the embedding
\begin{eqnarray}\label{eq:iotarho}
\iota_\rho:W_H\rtimes^{\rho}\Theta&\hookrightarrow& W_\kappa\times\Theta\\
(w_H,\sigma)&\mapsto&(w_H\rho(\sigma),\sigma).
\end{eqnarray}
The embedding $\iota_\rho$ extends uniquely to an embedding
\begin{equation}\label{eq:tiliota}
\tiliota_\rho:\tilW_H\rtimes^{\rho}\Theta=\xcoch(T)\rtimes(W_H\rtimes^{\rho}\Theta)\xrightarrow{\id\rtimes\iota_\rho}\xcoch(T)\rtimes(W_\kappa\times\Theta)=\tilW_\kappa\times\Theta.
\end{equation}
\end{remark}

Now we can state the Endoscopic Decomposition Theorem.

\begin{theorem}\label{th:endo} Fix a finite order element $\kappa\in\hatT(\Ql)$. Fix global Kostant sections for the Hitchin fibrations $f^{\Hit}:\MHit\to\calA$ and $f^{\Hit}_H:\MHit_H\to\calA_H$ for all endoscopic groups $H$ relevant to $\kappa$. Then there is a natural isomorphism in $D^b_c(\tcA_{\kappa,\Theta})$:
\begin{equation}\label{eq:endo}
\theta^*(\tfQl)_\kappa\cong\bigoplus_{\rho,H=H_\rho}\tilmu_{H,\Theta,*}\theta_{H}^*(\tfHQl)_{\kappa}[-2r_\kappa](-r_\kappa).
\end{equation}
where the direct sum runs over all homomorphisms $\rho:\Theta\to\pi_0(\kappa)$. Moreover, for each $\rho$, the embedding of the direct summand indexed by $\rho$ in (\ref{eq:endo})
\begin{equation*}
\tilmu_{H,\Theta,*}\theta_{H}^*(\tfHQl)_{\kappa}[-2r_\kappa](-r_\kappa)\hookrightarrow\theta^*(\tfQl)_\kappa
\end{equation*}
is $(\tilW_H\rtimes^{\rho}\Theta,\tilW_\kappa\times\Theta)$-equivariant under the embedding $\tiliota_\rho:\tilW_H\rtimes^{\rho}\Theta\hookrightarrow\tilW_\kappa\times\Theta$ in (\ref{eq:tiliota}).
\end{theorem}

The proof will be given in Sec. \ref{ss:pfendo}.


\subsection{Proof of Theorem \ref{th:endo}--Relating the parabolic Hitchin complexes}\label{ss:pfendo}
In this subsection, we give the proof of Th. \ref{th:endo}. The strategy of the proof is the following: we first use the parabolic version of the endoscopic correspondence to construct a map as in (\ref{eq:endo}), and then we check it is indeed an isomorphism, which, thanks to the Support Theorem \ref{th:supp} and Prop. \ref{p:kappasupp}, reduces to a calculation over the generic points of $\tcA_{\kappa,\Theta}$.

We first fix a rigidified endoscopic datum $(\kappa,\rho)$, hence the endoscopic group scheme $H$. Let $d_H=\dim(\Mpar_{H,\Theta})$. Let $\tcC_{H,\Theta}$ be the parabolic endoscopic correspondence between $\Mpar_{H,\Theta}$ and $\Mpar_{\Theta}$ over $\tcA_{\Theta}$. Since $\tcC_{H,\Theta}$ is the closure of the graph $\Gamma(\tilh_{\Theta})$, we have $\dim(\tcC_{H,\Theta}/\tcA_{H,\Theta})=\dim(\tilf_\Theta)=r_\kappa+\dim(\tilf_{H,\Theta})$ by the equality (\ref{eq:diffr}). Therefore $\dim(\tcC_{H,\Theta})=d_H+r_\kappa$. Since we work with a fixed $\kappa$ throughout this subsection, we simply write $r$ for $r_\kappa$.

The fundamental class of $\tcC_{H,\Theta}$ gives an element
\begin{eqnarray*}
[\tcC_{H,\Theta}]\in \hBM{2(d_H+r)}{\tcC_{H,\Theta}}&\cong&\cohog{-2(d_H+r)}{\tcC_{H,\Theta},\overleftarrow{c_{H,\Theta}}^!\Ql[2d_H](d_H)}\\
&=&\Corr(\tcC_{H,\Theta};\Ql[-2r](-r),\Ql).
\end{eqnarray*}
Here the first isomorphism is given by the choice of a fundamental class of $\Mpar_{H,\Theta}$, which is smooth.

The cohomological correspondence $[\tcC_{H,\Theta}]$ gives a map
\begin{equation*}
[\tcC_{H,\Theta}]_\#:\tilf_{\Theta,*}\Ql\to(\tilmu_{H,\Theta}\circ\tilf_{H,\Theta})_*\Ql[-2r](-r).
\end{equation*}
Using proper base change, we can also write $[\tcC_{H,\Theta}]_\#$ as
\begin{equation}\label{eq:tcC}
[\tcC_{H,\Theta}]_\#:\theta^*\tfQl\to\tilmu_{H,\Theta,*}\theta_H^*\tfHQl[-2r](-r).
\end{equation}
By Rem. \ref{rm:Wkeq}, the LHS of (\ref{eq:tcC}) admits a $\tilW_\kappa\times\Theta$-equivariant structure, and the RHS of (\ref{eq:tcC}) admits a $\tilW_H\rtimes^{\rho}\Theta$-equivariant structure.


\begin{prop}\label{p:Ceq}
The map $[\tcC_{H,\Theta}]_\#$ is $\tilW_H\rtimes^{\rho}\Theta$-equivariant under the embedding $\tiliota_\rho:\tilW_H\rtimes^{\rho}\Theta\hookrightarrow\tilW_\kappa\times\Theta$ defined in (\ref{eq:tiliota}).
\end{prop}
\begin{proof}
Recall from Rem. \ref{rm:qsHecke} that for $\tilw\in\tilW_H\rtimes^{\rho}\Theta$, we have the reduced Hecke correspondence $\calH_{H,\tilw}$ of $\Mpar_{H,\Theta}$ over $\calA_H\times X$, whose fundamental class gives the $\tilw$-action on $\tilf_{H,\Theta,*}\Ql$. Similarly, we have the reduced Hecke correspondence $\calH_{\tiliota_\rho(\tilw)}$ of $\Mpar_\Theta$ over $\calA\times X$, whose fundamental class gives the $\tiliota_\rho(\tilw)$-action on $\tilf_{\Theta,*}\Ql$. We also view $\tcC_{H,\Theta}$ as a correspondence between $\Mpar_{H,\Theta}$ and $\Mpar_\Theta$ over $\calA\times X$. To prove the proposition, we have to show that
\begin{equation}\label{eq:compequal}
[\tcC_{H,\Theta}]_\#\circ[\calH_{\tiliota_\rho(\tilw)}]_\#=[\calH_{H,\tilw}]_\#\circ[\tcC_{H,\Theta}]_\#:\fpar_{\Theta,*}\Ql\to(\mu_{H}\times\id_X)_*\fpar_{H,\Theta,*}\Ql[-2r](-r).
\end{equation}
We will show in Lem. \ref{l:endoG2} that both compositions $\tcC_{H,\Theta}*\calH_{\tiliota_\rho(\tilw)}$ and $\calH_{H,\tilw}*\tcC_{H,\Theta}$ satisfy the condition (G-2) in \cite[Def. A.5.1]{GSI} with respect to the open subset $U=(\calA\times X)_0\subset(\calA\times X)$. Since $\calH_{H,\tilw}$ and $\calH_{\tiliota_\rho(\tilw)}$ are graph-like, $\tcC_{H,\Theta}$ is right graph-like, we can apply \cite[Prop. A.5.5]{GSI} to conclude
\begin{eqnarray*}
& [\tcC_{H,\Theta}]_{\#}{\circ}[\calH_{\tiliota_\rho(\tilw)}]_{\#}= [\tcC_{H,\Theta}*\calH_{\tiliota_\rho(\tilw)}]_{\#};\\
& [\calH_{H,\tilw}]_{\#}{\circ}[\tcC_{H,\Theta}]_{\#}=[\calH_{H,\tilw}*\tcC_{H,\Theta}]_{\#}.
\end{eqnarray*}
Therefore, in order to prove (\ref{eq:compequal}), it suffices to show
\begin{equation*}
[\tcC_{H,\Theta}*\calH_{\tiliota_\rho(\tilw)}]_\#=[\calH_{H,\tilw}*\tcC_{H,\Theta}]_\#.
\end{equation*}
Again, by Lem. \ref{l:endoG2} and \cite[Lem. A.5.2]{GSI}, it suffices to establish an isomorphism of correspondences between $\Mpar_{H,\Theta}|_U$ and $\Mpar_\Theta|_U$ over $U$:
\begin{equation}\label{eq:correq}
(\tcC_{H,\Theta}*\calH_{\tiliota_\rho(\tilw)})|_U=(\calH_{H,\tilw}*\tcC_{H,\Theta})|_U.
\end{equation}

Over the locus $U=(\calA\times X)_0$, $\calH_{\tiliota_\rho(\tilw)}|_U$ is the graph of the right  $\tiliota_\rho(\tilw)$-action on $\Mpar_\Theta|_U$. Therefore, $(\tcC_{H,\Theta}*\calH_{\tiliota_\rho(\tilw)})|_U$ is the closure of the graph of
\begin{equation}\label{eq:maphw}
\tilh_{\Theta}\circ\tiliota_{\rho}(\tilw^{-1}):\Mparreg_\Theta|_U\to\Mparreg_\Theta|_U\to\Mparreg_{H,\Theta}|_U.
\end{equation}

On the other hand, since $\delta_{H}(a_H,x)\leq\delta(\mu_H(a_H),x)$, the preimage of $U$ under $\mu_H\times\id_X$ is contained in $(\calA_H\times X)_0$. Therefore $\calH_{H,\tilw}|_U$ is the graph of the right $\tilw$-action on $\Mpar_{H,\Theta}|_U$. Hence $(\calH_{H,\tilw}*\tcC_{H,\Theta})|_U$ is the closure of the graph of
\begin{equation}\label{eq:mapwh}
\tilw^{-1}\circ\tilh_{\Theta}:\Mparreg_\Theta|_{U}\to\Mparreg_{H,\Theta}|_U\to\Mparreg_{H,\Theta}|_U.
\end{equation}

We claim that the two morphisms in (\ref{eq:maphw}) and (\ref{eq:mapwh}) are the same. In fact, for $\tilw^{-1}=(\lambda,u)\in\xcoch(T)\rtimes(W_H\rtimes^{\rho}\Theta)$, (where $\lambda\in\xcoch(T)$ and $u\in W_H\rtimes^{\rho}\Theta$), the morphism $\tilh_{\Theta}\circ\tiliota_\rho(\tilw^{-1})$ is induced by (see \cite[Cor. 4.3.8]{GSI})
\begin{eqnarray}\label{eq:actionh}
\tcA_{H,\Theta}|_U\times_{\calA}\calP\to\tcA_{H,\Theta}|_U\times_{\calA}\calP\to\tcA_{H,\Theta}|_U\times_{\calA_H}\calP_H\\
\notag
(\tila_H,y) \mapsto (u^{-1}\tila_H,y+s_\lambda(\tila))\mapsto(u^{-1}\tila_H,h_{\calP}(a_H,y)+h_{\calP}(a_H,s_\lambda(\tila)))
\end{eqnarray}
where $\tila_H\in\tcA_{H,\Theta}|_U$ has images $\tila\in\tcA|_U$ and $a_H\in\calA_H$, and $s_\lambda:\tcArs\to\calP$ is the morphism defined in \cite[Rem. 4.3.7]{GSI}. 

Similarly, the morphism $\tilw^{-1}\circ\tilh_{\Theta}$ is induced by
\begin{eqnarray}\label{eq:haction}
\tcA_{H,\Theta}|_U\times_{\calA}\calP\to\tcA_{H,\Theta}|_U\times_{\calA_H}\calP_H\to\tcA_{H,\Theta}|_U\times_{\calA_H}\calP_H\\
\notag
(\tila_H,y)\mapsto(\tila_H,h_{\calP}(a_H,y))\mapsto(u^{-1}\tila_H,h_{\calP}(a_H,y)+s_{H,\lambda}(\tila_H))
\end{eqnarray}
where $s_{H,\lambda}:\tcA_{H,\Theta}|_U\to\calP_H$ is the morphism defined similarly as $s_\lambda$.

By the construction of the morphisms $h_{\calP},s_\lambda$ and $s_{H,\lambda}$, we have a commutative diagram
\begin{equation*}
\xymatrix{\tcA_{H,\Theta}|_U\ar[dr]_{s_{H,\lambda}}\ar[r] & \calA_H\times_{\calA}\tcA|_U\ar[r]^{\id\times s_\lambda} & \calA_H\times_{\calA}\calP\ar[dl]^{h_{\calP}}\\
& \calP_H}
\end{equation*}
In other words,
\begin{equation*}
h_{\calP}(a_H,s_\lambda(\tila))=s_{H,\lambda}(\tila_H)
\end{equation*}
Therefore, the two morphisms in (\ref{eq:actionh}) and (\ref{eq:haction}) are the same. This implies that the morphisms in (\ref{eq:maphw}) and (\ref{eq:mapwh}) are the same. Since both correspondences in (\ref{eq:correq}) are closures of the graph of the same morphism, (\ref{eq:correq}) is proved. This completes the proof of the proposition.
\end{proof}


\begin{lemma}\label{l:endoG2}
Both correspondences $\tcC_{H,\Theta}*\calH_{\tiliota_\rho(\tilw)}$ and $\calH_{H,\tilw}*\tcC_{H,\Theta}$ satisfy the condition (G-2) in \cite[Def. A.5.1]{GSI} with respect to the open subset $U=(\calA\times X)_0\subset\calA\times X$.
\end{lemma}
\begin{proof}
We let $\calH$ be any finite-type closed substack of $\Heckep_\Theta$, and $\calH_H$ be any finite-type closed substack of $\Heckep_{H,\Theta}$. We will show that $\tcC_{H,\Theta}*\calH$ and $\calH_H*\tcC_{H,\Theta}$ both satisfy the condition (G-2) with respect to $U$. Let $d=\dim\tcC_{H,\Theta}=\dim(\tcA_{H,\Theta}\times_{\tcA_{\Theta}}\Mpar_{\Theta})$.

Since the reduced structures of $\calH|_U$ and $\calH_H|_U$ are substacks of unions of graphs of automorphisms of $\Mpar_\Theta|_U$ and $\Mpar_{H,\Theta}|_U$, therefore
\begin{eqnarray*}
\dim(\tcC_{H,\Theta}*\calH|_U)\leq\dim(\tcC_{H,\Theta}|_U)=d;\\
\dim(\calH_H*\tcC_{H,\Theta}|_U)\leq\dim(\tcC_{H,\Theta}|_U)=d.
\end{eqnarray*}
This verifies the first condition of (G-2). Let $\partial U=\calA\times X-U$. It remains to verify that the images of $\tcC_{H,\Theta}*\calH|_{\partial U}$ and $\calH_H*\tcC_{H,\Theta}|_{\partial U}$ in $\Mpar_{H,\Theta}\times_{\tcA_{\Theta}}\Mpar_\Theta$ have dimensions less than $d$. We call these conditions (G-2)'.

Let $\tcCmod_{H,\Theta}=\tcA_{H,\Theta}\times_{\calA_H}\Cmod_H$ be the base change of the modular endoscopic correspondence. By Constructions \ref{cons:modcorr}, the graph $\Gamma(h_{\calM})$ naturally lifts to an open substack of $\Cmod_H$. By Lem. \ref{l:Cmodproper} and the fact that $\MHit_H\times_{\calA_H}(\calA_H\times_{\calA}\MHit)\to\calA_H$ is separated, the morphism $\overleftrightarrow{m}:\Cmod_H\to\MHit_H\times_{\calA_H}(\calA_H\times_{\calA}\MHit)$ also satisfies the existence part of the valuative criterion, up to a finite separable extension (as in Lem. \ref{l:Cmodproper}). Therefore the image of $\Cmod_H$ in $\MHit_H\times_{\calA_H}(\calA_H\times_{\calA}\MHit)$ is closed under specializations (cf. \cite[Th. 7.10]{LM}), hence closed, hence contains $\calC$, the closure of $\Gamma(h_{\calM})$. Similarly, the image of $\tcCmod_{H,\Theta}$ in $\Mpar_{H,\Theta}\times_{\tcA_{\Theta}}\Mpar_\Theta$ contains the graph $\Gamma(\tilh_\Theta)$, hence its closure $\tcC_{H,\Theta}$. Therefore, to check the condition (G-2)' for $\tcC_{H,\Theta}*\calH$ and $\calH_H*\tcC_{H,\Theta}$, it suffices to check the same condition (G-2)' for $\tcCmod_{H,\Theta}*\calH$ and $\calH_H*\tcCmod_{H,\Theta}$. (We did not directly work with $\tcCmod_{H,\Theta}$ because we did not know enough about its geometry; but as far as the dimension estimates are concerned, it is safe and convenient to work directly with $\tcCmod_{H,\Theta}$.)

We first consider the composition of correspondences
\begin{equation*}
\xymatrix{& & \tcCmod_{H,\Theta}*\calH\ar[dl]_{\overleftarrow{e}}\ar[dr]^{\overrightarrow{e}}\\
& \tcCmod_{H,\Theta}\ar[dl]_{\overleftarrow{m}}\ar[dr]^{\overrightarrow{m}} & & \calH\ar[dr]^{\overrightarrow{h}}\ar[dl]_{\overleftarrow{h}}\\
\Mpar_{H,\Theta} & & \Mpar_{\Theta} & & \Mpar_{\Theta}}
\end{equation*}

Fix a point $\tila_H\in\tcA_{H,\Theta}$ with image $a_H\in\calA_H,\tila\in\tcA_{\Theta},a\in\calA$ and $x\in X$. Assume that $\tila\in\partial U$, i.e., $\delta(a,x)\geq1$. Let $\xi\in\Mpar_{a,x}$ and $\tilxi=(\xi,\tila)\in\Mpar_\Theta$. We want to estimate the dimension of the image of the fiber $(\overrightarrow{h}\overrightarrow{e})^{-1}(\tilxi)$ in $\Mpar_{H,\Theta}$. Recall that a point in this fiber determines a triple $(\eta,\xi',\tau)$ where the $\eta\in\Mpar_{H,a_H,x},\xi'\in\Mpar_{a,x}$ and $\tau:\xi|_{X-\{x\}}\isom\xi'|_{X-\{x\}}$.

Let $\Delta_a$ be the pull-back of the discriminant locus $\Delta\subset\frc$ to $X$ via $a$. Since $\delta(a,x)\geq1$, we have $x\in \Delta_a$. Recall the morphism $h$ in (\ref{eq:hstack}) is an isomorphism away from the discriminant locus of $\frc$. Therefore, $(a_H,\xi')$ determines a morphism
\begin{equation*}
h(a_H,\xi'):X-\Delta_a\xrightarrow{(a_H,\xi')}\frc_H\times_{\frc}[\frg^{\rs}/G]_D\xrightarrow{h}[\frh^{\rs}/H]_D.
\end{equation*}
By Construction (\ref{cons:modcorr}), for a pair of point $(\eta,\xi')\in\Mpar_{H,a_H,x}\times\Mpar_{a,x}$ in the image of $\tcCmod_{H,\Theta}$, the morphism $\eta:X-\Delta_a\to[\frh^{\rs}/H]_D$ is canonically isomorphic to $h(a_H,\xi')$. Moreover, $\tau:\xi|_{X-\{x\}}\isom\xi'|_{X-\{x\}}$ gives an isomorphism $h(a_H,\xi)\cong h(a_H,\xi')$, hence we have an isomorphism
\begin{equation*}
\eta|_{X-\Delta_a}\cong h(a_H,\xi):X-\Delta_a\to[\frh^{\rs}/H]_D.
\end{equation*}
In other words, the image of the fiber $(\overrightarrow{h}\overrightarrow{e})^{-1}(\tilxi)$ in $\Mpar_{H,a_H,x}$ is contained in the image of the following stack in $\Mpar_{H,\tila_H}$:
\begin{equation*}
\calH_{\Delta_a,\xi}=\{(\eta,\sigma)|\eta\in\Mpar_{H,a_H,x},\sigma:\eta|_{X-\Delta_a}\cong h(a_H,\xi)\}.
\end{equation*}
Clearly, $\calH_{\Delta_a,\xi}$ is the a product of local Hecke modifications of the parabolic Hitchin triple $h(a_H,\xi)$. By the discussion in \cite[Sec. 3.3]{GSI}, each local Hecke modification at $v\in \Delta_a$ is isomorphic to an affine Springer fiber in either the affine flag variety (if $v=x$) or the affine Grassmannian (if $v\neq x$). By the dimension formula of affine Springer fibers (see \cite{Bezr} and \cite[3.7]{NgoFL}), we have
\begin{equation*}
\dim\calH_{\Delta_a,\xi}\leq\sum_{v\in \Delta_a}\delta_H(a_H,v)=\delta_H(a_H).
\end{equation*}
In other words, 
\begin{equation}\label{eq:smallfiber}
\dim(\overleftarrow{h}\overleftarrow{e}(\overrightarrow{h}\overrightarrow{e})^{-1}(\tilxi))\leq\delta_H(a_H).
\end{equation}

Let $\tcB=\tilmu_{H,\Theta}(\tcA_{H,\Theta})$. For each $\delta\in\ZZ_{>0}$, let $\tcB_\delta=\tilmu_{H,\Theta}(q_{H,\Theta}^{-1}(\calA_{H,\delta}))\subset\tcB$ and $\partial\tcB_\delta=\tcB_{\delta}\cap\partial U$. We have
\begin{eqnarray*}
\dim(\partial\tcB_{\delta})\leq\dim(\tcB_{\delta})-1&=&\dim( q_{H,\Theta}^{-1}(\calA_{H,\delta}))-1\\
&\leq&\dim(\tcA_{H,\Theta})-\delta-1=\dim(\tcB)-\delta-1.
\end{eqnarray*}
Here we have used the codimension estimate \cite[Prop. 3.5.5]{GSI}. Since the morphism $\fpar_{\Theta}:\Mpar_\Theta\to\tcA_\Theta$ is flat, we have
\begin{eqnarray*}
\dim(\fpar_{\Theta})^{-1}(\partial\tcB_\delta)&=&\dim(\fpar_\Theta)+\dim(\partial\tcB_\delta)\\
&\leq&\dim(\fpar_\Theta)+\dim(\tcB)-\delta-1=d-\delta-1.
\end{eqnarray*}

By the inequality (\ref{eq:smallfiber}), the image of  $(\overrightarrow{h}\overrightarrow{e})^{-1}(\tilxi)$ for any $\tilxi\in(\fpar)^{-1}(\partial\tcB_\delta)$ has dimension $\leq\delta$ in $\Mpar_{H,\Theta}$, we conclude that the image of $(\tcCmod_{H,\Theta}*\calH)|_{\partial\tcB_\delta}$ in $\Mpar_{H,\Theta}\times_{\tcA_\Theta}{\Mpar_\Theta}$ is bounded by $\dim(\fpar_{\Theta})^{-1}(\partial\tcB_\delta)+\delta\leq(d-\delta-1)+\delta=d-1.$ Since $\partial\tcB$ is the union of $\partial\tcB_{\delta}$ for finitely many $\delta\in\ZZ_{>0}$, this verifies the condition (G-2)' for $\tcCmod_{H,\Theta}*\calH$.

Similarly, one can verify (G-2)' for $\calH_H*\tcCmod_{H,\Theta}$. This completes the proof of the lemma.
\end{proof}


\subsection{Proof of Theorem \ref{th:endo}--Reduction to the generic points}

Since $[\tcC_{H,\Theta}]_\#$ is $\tilW_H\rtimes^{\rho}\Theta$-equivariant, it is in particular equivariant under the lattice $\xcoch(T)$. Therefore, $[\tcC_{H,\Theta}]_\#$ induces a map between the $\kappa$-parts
\begin{equation}\label{eq:Ckappa}
[\tcC_{H,\Theta}]_{\#,\kappa}:(\tilf_{\Theta,*}\Ql)_{\kappa}\to\tilmu_{H,\Theta,*}\theta_H^*(\tfHQl)_{\kappa}[-2r](-r)
\end{equation}
which is also equivariant under $(\tilW_{H}\rtimes^{\rho}\Theta,\tilW_\kappa\times\Theta)$. To finish the proof of Th. \ref{th:endo}, it remains to show that the sum of $[\tcC_{H,\Theta}]_{\#,\kappa}$ for all $\rho:\Theta\to\pi_0(\kappa)$:
\begin{equation}\label{eq:sumC}
\bigoplus_{\rho}[\tcC_{H,\Theta}]_{\#,\kappa}:(\tilf_{\Theta,*}\Ql)_{\kappa}\to\bigoplus_{\rho}\tilmu_{H,\Theta,*}\theta_H^*(\tfHQl)_{\kappa}[-2r](-r)
\end{equation}
is an isomorphism. Equivalently, it suffices to show that $\bigoplus_{\rho}[\tcC_{H,\Theta}]_{\#,\kappa}$ induces an isomorphism on the perverse cohomology sheaves.

Since $\tilmu_{H,\Theta}:\tcA_{H,\Theta}\to\tcA_{\Theta}$ is finite, $\tilmu_{H,\Theta,*}$ is exact in the perverse $t$-structure. Since $\theta:\tcA_{\Theta}\to\tcA$ and $\theta_H:\tcA_{H,\Theta}\to\tcA_H$ are \'etale, $\theta^*$ and $\theta_H^*$ are also exact in the perverse $t$-structure. Therefore, the map (\ref{eq:sumC}) induces the following map on perverse cohomology:
\begin{equation}\label{eq:pHC}
\bigoplus_{\rho}\pH^i([\tcC_{H,\Theta}]_{\#,\kappa}):\theta^*\pH^i(\tfQl)_{\kappa}\to\bigoplus_{\rho}\tilmu_{H,\Theta,*}\theta_H^*\pH^{i-2r}(\tfHQl)_{\kappa}(-r).
\end{equation}

By Prop. \ref{p:kappasupp}, the support of the simple constituents of $\pH^i(\tfQl)_{\kappa}$ are irreducible components of $\tcA_{\kappa}$. This implies that $\pH^i(\tfQl)_{\kappa}$ is the middle extension of its restriction to any open dense subset of $\tcA_\kappa$. Since $\theta$ is \'etale, $\theta^*\pH^i(\tfQl)_{\kappa}$ is also the middle extension of its restriction to any open dense subset of $\tcA_{\kappa,\Theta}$.

Again by Prop. \ref{p:kappasupp}, the support of the simple constituents of $\pH^i(\tfHQl)_{\kappa}$ are irreducible components $\tcA_H$, because $\kappa\in(Z\hatH)^{\Theta}$ (see Rem. \ref{rm:qsstable}). Since $\theta_H$ is \'etale, $\theta^*\pH^i(\tfHQl)_{\kappa}$ is also the middle extension of its restriction to any open dense subset of $\tcA_{H,\Theta}$. By Lem. \ref{l:AHdisj}, the morphism
\begin{equation}\label{eq:allC}
\bigsqcup_{\rho}\tilmu_{H,\Theta}:\bigsqcup_{\rho}\tcA_{H,\Theta}\to\tcA_{\kappa,\Theta}
\end{equation}
is finite and birational. Therefore, the perverse sheaf $\bigoplus_{\rho}\tilmu_{H,\Theta,*}\theta_H^*\pH^{i}(\tfHQl)_{\kappa}$ is still a middle extension of any open dense subset of $\tcA_{\kappa,\Theta}$.

From the above discussions, we conclude that both sides of (\ref{eq:pHC}) are middle extensions of their restrictions to any open dense subset of $\tcA_{\kappa,\Theta}$. Therefore, to prove that (\ref{eq:pHC}) is an isomorphism, it suffices to check it induces an isomorphism over every geometric generic point of $\tcA_{\kappa,\Theta}$. Since the morphism (\ref{eq:allC}) is finite and birational, every geometric generic point of $\tcA_{\kappa,\Theta}$ is the image of a unique geometric generic point of $\tcA_{H,\Theta}$ for a unique $H$ (or a unique $\rho$). Hence, to prove that (\ref{eq:pHC}) is an isomorphism, it suffices to show that for every $\rho$ (hence the corresponding $H$) and every geometric generic point $\eta\in\tcA_{H,\Theta}$, the restriction of $[\tcC_{H,\Theta}]_{\#,\kappa}$ on the stalks at $\eta$:
\begin{equation*}
[\tcC_{H,\Theta,\eta}]_{\#,\kappa}:(\theta^*(\tfQl)_{\kappa})_\eta\to(\theta_H^*(\tfHQl)_{\kappa})_{\eta}[-2r](-r)
\end{equation*}
is an isomorphism. For this we can use the explicit description of $\calC_H$ over the generic points of $\calA_H$ given in App. \ref{s:endocorr}.

Let $(a_H,\tilx_H)\in\tcA^{\rs}_H$ and $(a,\tilx)\in\tcA^{\rs}$ be the projections of $\eta$, which are geometric generic points of $\tcA_H$ and $\tcA$ respectively. We have
\begin{eqnarray}\label{eq:MHgen}
(\theta^*(\tfQl)_\kappa)_{\eta}&\cong&\cohog{*}{\Mpar_{a,\tilx}}_\kappa\cong\cohog{*}{\MHit_a}_\kappa;\\
\label{eq:MHHgen}
(\theta_H^*(\tfHQl)_{\kappa})_{\eta}&\cong&\cohog{*}{\Mpar_{H,a_H,\tilx_H}}_{\kappa}\cong\cohog{*}{\MHit_{H,a_H}}_\kappa.
\end{eqnarray}
Since $(a,\tilx)$ and $(a_H,\tilx_H)$ are in the regular semisimple locus, the $\xcoch(T)$-action on $\cohog{*}{\Mpar_{a,\tilx}}$ and $\cohog{*}{\Mpar_{H,a_H,\tilx_H}}$ factor through the action of $\pi_0(\calP_a)$ and $\pi_0(\calP_{a_H})$. Therefore, the $\kappa$-parts on the last two terms of (\ref{eq:MHgen}) and (\ref{eq:MHHgen}) can be understood as the parts on which $\pi_0(\calP_a)$ or $\pi_0(\calP_{a_H})$ acts through the character
\begin{equation*}
\pi_0(\calP_a)\to\pi_0(\calP_{a_H})\xrightarrow{\kappa}\Ql^\times.
\end{equation*}

By Lem. \ref{l:tcCUisC}, $\tcC_{H,\Theta,\eta}$ is the pull-back of $\calC_H$ along the morphism $\eta\hookrightarrow\tcArs_{H,\Theta}\to\calA_H$. Therefore, $\tcC_{H,\Theta,\eta}=\calC_{H,a_H}$. Combining this fact with the isomorphisms (\ref{eq:MHgen}) and (\ref{eq:MHHgen}), we get 
\begin{equation*}
[\tcC_{H,\Theta,\eta}]_\#=[\calC_{H,a_H}]_\#:\cohog{*}{\MHit_a}\to\cohog{*}{\MHit_{H,a_H}}[-2r](-r).
\end{equation*}

Hence, we have reduced Th. \ref{th:endo} to the following


\begin{lemma}\label{l:nodal}
The $\kappa$-part of $[\calC_{H,a_H}]_\#$:
\begin{equation*}
[\calC_{H,a_H}]_{\#,\kappa}:\cohog{*}{\MHit_a}_\kappa\to\cohog{*}{\MHit_{H,a_H}}_\kappa[-2r](-r)
\end{equation*}
is an isomorphism.
\end{lemma}
\begin{proof}
By Lem. \ref{l:simpleC}(2), we can use induction to treat one correspondence $\calC_i$ at a time. The map $[\calC_i]_\#$ is given by
\begin{equation}\label{eq:gysin}
\cohog{*}{\calM_i}\xrightarrow{\overrightarrow{c_i}^*}\cohog{*}{\calC_i}\xrightarrow{\overleftarrow{c_i}_!}\cohog{*}{\calM_{i-1}}[-2](-1).
\end{equation}
where $\overleftarrow{c_i}_!$ is the Gysin map of the $\PP^1$-fibration $\overleftarrow{c_i}$. Since $\calP_a$ acts on the diagram (\ref{d:compC}), it makes sense to talk about the $\kappa$-part of the cohomology groups in (\ref{eq:gysin}). We need to show that
\begin{equation*}
[\calC_i]_{\#,\kappa}:\cohog{*}{\calM_i}_\kappa\to\cohog{*}{\calC_i}_\kappa\to\cohog{*}{\calM_{i-1}}_\kappa[-2](-1)
\end{equation*}
is an isomorphism.

We first analyze the $\PP^1$-fibration $\overleftarrow{c_i}$. We have a distinguished triangle
\begin{equation*}
\const{\calM_{i-1}}\to\overleftarrow{c_i}_*\const{\calC_i}\to\const{\calM_{i-1}}[-2](-1)\to
\end{equation*}
where the second map is the Gysin map. This triangle gives a long exact sequence
\begin{equation*}
\cdots\to\cohog{j}{\calM_{i-1}}\xrightarrow{\overleftarrow{c_i}^*}\cohog{j}{\calC_i}\xrightarrow{\overleftarrow{c_i}_!}\cohog{j-2}{\calM_{i-1}}(1)\to\cdots
\end{equation*}

We then analyze the birational morphism $\overrightarrow{c_i}$. Let $\calN_i\subset\calM_i$ be the common image of $\calC^0_i$ and $\calC^\infty_i$. Since $\overrightarrow{c_i}$ identifies $\calC^0_i$ and $\calC^\infty_i$ and is an isomorphism elsewhere, we have a distinguished triangle
\begin{equation*}
\const{\calM_{i}}\to\overrightarrow{c_i}_*\const{\calC_i}\xrightarrow{\delta}\dfrac{b^0_*\const{\calC_i^0}\oplus b^\infty_*\const{\calC^\infty_i}}{b_*\const{\calN_i}}\to,
\end{equation*}
where $b^0,b^\infty,b$ are the natural morphisms from $\calC^0_i,\calC^\infty_i$ and $\calN_i$ to $\calM_i$. The map $\delta$ is induced from the restrictions from $\calC_i$ to $\calC^0_i$ and $\calC^\infty_0$. This triangle gives a long exact sequence
\begin{equation*}
\cdots\to\cohog{j}{\calM_i}\xrightarrow{\overrightarrow{c_i}^*}\cohog{j}{\calC_i}\xrightarrow{\delta}\dfrac{\cohog{j}{\calC^0_i}\oplus\cohog{j}{\calC^\infty_i}}{\cohog{j}{\calN_i}}\to\cdots
\end{equation*}

These exact sequences all carry natural actions of $\pi_0(\calP_a)$. Taking the $\kappa$-parts of the $\pi_0(\calP_a)$-action, we still get long exact sequences. In particular, we get a diagram
\begin{equation}\label{d:shizi}
\xymatrix{& \cohog{j}{\calM_{i-1}}_\kappa\ar[d]^{\overleftarrow{c_i}^*}\ar[dr]^{\epsilon}\\
\cohog{j}{\calM_i}_\kappa\ar[r]^{\overrightarrow{c_i}^*}\ar[dr]^{[\calC_i]_{\#,\kappa}} & \cohog{j}{\calC_i}_\kappa\ar[r]^{\delta}\ar[d]^{\overleftarrow{c_i}_!} & \dfrac{\cohog{j}{\calC^0_i}_\kappa\oplus\cohog{j}{\calC^\infty_i}_\kappa}{\cohog{j}{\calN_i}_\kappa}\\
& \cohog{j-2}{\calM_{i-1}}_\kappa(-1)}
\end{equation}
where the vertical and horizontal lines are exact in the middle.
 
We claim that the map $\epsilon$ in the above diagram is an isomorphism. In fact, if we use the sections $s^0_i$ and $s^\infty_i$ to fix identifications:
\begin{equation*}
\cohog{j}{\calC^0_i}_\kappa\cong\cohog{j}{\calM_{i-1}}_\kappa\cong\cohog{j}{\calC^\infty_i}_\kappa,
\end{equation*}
the image of $\cohog{j}{\calM_{i-1}}_\kappa$ in $\cohog{j}{\calC^0_i}_\kappa\oplus\cohog{j}{\calC^\infty_i}_\kappa\cong\cohog{j}{\calM_{i-1}}_\kappa^{\oplus2}$ is the diagonal. On the other hand, by Lem. \ref{l:simpleC}(4), the image of $\cohog{j}{\calN_i}_\kappa$ in $\cohog{j}{\calC^0_i}_\kappa\oplus\cohog{j}{\calC^\infty_i}_\kappa\cong\cohog{j}{\calM_{i-1}}_\kappa^{\oplus2}$ is the graph of the scalar map $\kappa(\beta^\vee_i)$. Since $\kappa(\beta^\vee_i)\neq1$, the images of $\cohog{j}{\calM_{i-1}}_\kappa$ and $\cohog{j}{\calN_i}_\kappa$ are transversal with intersection $0$. This proves that $\epsilon$ is an isomorphism.

Since $\epsilon$ is an isomorphism, we know that the map $\overleftarrow{c_i}^*$ in the diagram (\ref{d:shizi}) is injective and $\delta$ is surjective. This being true for all $j\in\ZZ$, we conclude that the vertical and horizontal lines in diagram (\ref{d:shizi}) are also exact at the ends, i.e., they are short exact sequences. Again, because the upper tilted arrow $\epsilon$ is an isomorphism, we easily conclude that the lower tilted arrow $[\calC_i]_{\#,\kappa}$ must be an isomorphism. This completes the proof.
\end{proof}


\section{Parabolic Hitchin complexes for Langlands dual groups}\label{s:LD}

In this section, we study the stable part $(\tfQl)_{\st}$ of the enhanced parabolic Hitchin complex $\tfQl$ using the Langlands dual group of $G$. We will consider the enhanced parabolic Hitchin complexes for $G$ and its Langlands dual $\Gd$ simultaneously, and establish a relation between the lattice part of the global Springer action on $(\tfQl)_{\st}$ and certain Chern class action on $(\tfd_*\Ql)_{\st}$. Roughly speaking, on the level of perverse cohomology, the two pieces of lattice actions coming from the graded DAHA action constructed in \cite[Sec. 3]{GSII} get interchanged under Langlands duality. This is an evidence of the mirror symmetry (or T-duality) in the new formulation of the geometric Langlands conjecture.

\subsection{Statement of the results}\label{ss:stateLD}

Let $G$ be an almost simple algebraic group over $k$ and let $\Gd$ be the almost simple algebraic group over $k$ which is Langlands dual to $G$. More precisely, we fix maximal tori $T\subset G$ and $\Td\subset\Gd$ and an isomorphism $\xcoch(T)\isom\xch(\Td)$, which maps the coroot system of $(G,T)$ to the root system of $(\Gd,\Td)$. We use the same curve $X$ and the same divisor $D$ to define the Hitchin fibrations
\begin{equation*}
\fHit_G:\MHit_G\to\calA_G;\hspace{1cm}f^{\vee,\Hit}:\MHit_{\Gd}\to\calA_{\Gd}
\end{equation*}
and the enhanced parabolic Hitchin fibrations
\begin{equation}\label{eq:dualfpar}
\fpar:\Mpar_G\xrightarrow{\tilf}\tcA_G;\hspace{1cm} f^{\vee,\parab}:\Mpar_{\Gd}\xrightarrow{\tfd}\tcA_{\Gd}.
\end{equation}


\subsubsection{Identification of Hitchin bases} Fix a $W$-equivariant isomorphism $\iota:\frt\isom\frtd$. This allows us to identify $\frc_D$ with $\frcd_D$ and hence gives isomorphisms
\begin{eqnarray*}
\iota_{\calA}:\calA_{G}\isom\calA_{\Gd};\\
\iota_{\tcA}:\tcA_{G}\isom\tcA_{\Gd}.
\end{eqnarray*}
Since $G$ is almost simple, the choice of $\iota$ is unique up to scalar. Therefore, the resulting $\iota_{\tcA}$ is unique up to the natural action of $\GG_m$ on $\tcA_{G}$ (recall in the diagram defining $\tcA$ in \cite[Def. 3.1.7]{GSI}, all the terms have natural $\GG_m$-actions). Since all the objects over $\tcA_G$ or $\tcA_{\Gd}$ we consider will be $\GG_m$-equivariant, this ambiguity is harmless. We therefore fix the identification $\iota$ once and for all.


\begin{lemma}
The stacks $\Mpar_{G}$ and $\Mpar_{\Gd}$ have the same dimension.
\end{lemma}
\begin{proof}
By \cite[4.4.6]{NgoFL}, we have
\begin{equation*}
\dim(\MHit_G)=\dim(G)\deg(D)=\dim(\Gd)\deg(D)=\dim(\MHit_{\Gd}).
\end{equation*}
Since the parabolic Hitchin stacks have one more dimension than the usual Hitchin stacks (see \cite[Cor. 3.3.4]{GSI}), we conclude that $\dim(\Mpar_G)=\dim(\Mpar_{\Gd})$.
\end{proof}

Let $d=\dim(\Mpar_G)=\dim(\Mpar_{\Gd})$. Let
\begin{equation}\label{eq:defnKL}
K:=(\tilf_*\Ql)_{\st}[d](d/2);\hspace{1cm}L:=(\tfd_*\Ql)_{\st}[d](d/2)
\end{equation}
be the {\em stable parts} of the (shifted and twisted) enhanced parabolic Hitchin complexes of $G$ and $\Gd$. For each $i\in\ZZ$, let
\begin{equation*}
K^i:=\pH^iK;\hspace{1cm}L^i:=\pH^iL
\end{equation*}
be their perverse cohomology sheaves.

Our first result in this section is a Verdier duality between $K$ and $L$:

\begin{theorem}\label{th:VD}
For each $i\in\ZZ$, there is a natural isomorphism of perverse sheaves:
\begin{equation*}
\widetilde{\VD}^i:\DD K^i\cong L^i(i).
\end{equation*}
\end{theorem}

The proof will be given in Sec. \ref{ss:PfVD}.

\begin{remark}\label{rm:anotherVD}
Since $\Mpar_G$ is smooth and $\tilf$ is proper, the complex $K$ is Verdier self-dual, i.e., we have a natural isomorphism $K\cong\DD K$ once we fix a fundamental class $[\Mpar]$ of $\Mpar$. Hence we have a canonical isomorphism of perverse sheaves on $\tcA$:
\begin{equation*}
\tilv^i:K^i\cong\DD K^{-i}.
\end{equation*}
Note that this Verdier self-duality is same as the one mentioned in \cite[Prop. 4.5.1]{GSI}. Therefore, Th. \ref{th:VD} can be reformulated as an isomorphism
\begin{equation*}
\widetilde{\VD'}^i: K^i\cong L^{-i}(-i).
\end{equation*}
\end{remark}

To state the second (and the more interesting) result, we need to recall two lattice actions on $K$ and $L$.


\subsubsection{The Springer action} We have constructed an action of $\xcoch(T)$ on $K$ in \cite[Prop. 4.4.6]{GSI}. By Lem. \ref{l:latticefinite}, the $\xcoch(T)$-action on $K^i$ is semisimple. Since $K$ is the stable part, $\xcoch(T)$ acts trivially on $K^i$, therefore we may consider its ``subdiagonal'' entries under the perverse filtration. More precisely, for any $\lambda\in\xcoch(T)$, the action of $\lambda-\id$ on $K$ induces the zero map on $K^i$, therefore $\ptau_{\leq i}(\lambda-\id)$ factors through:
\begin{equation*}
\ptau_{\leq i}(\lambda-\id):\ptau_{\leq i}K\to\ptau_{\leq i-1}K.
\end{equation*}
Taking the $i$-th perverse cohomology, we get
\begin{equation}\label{eq:defSp}
\Sp^i(\lambda):=\pH^i(\lambda-\id):K^i\to K^{i-1}[1],
\end{equation}
which is an extension class between the perverse sheaves $K^i$ and $K^{i-1}$.


\subsubsection{The Chern class action} In \cite[Con. 3.3.4]{GSII}, we have constructed an action of $\xcoch(\Td)$ on the parabolic Hitchin complex of $\Gd$ via certain Chern classes. We recall the construction. The evaluation morphism
\begin{equation*}
\Mpar_{\Gd}\to[\frtd/\Td]_D\to\BB\Td
\end{equation*}
gives a universal $\Td$-torsor $\calL^{\Td}$ on $\Mpar_{\Gd}$. For each $\lambda\in\xch(\Td)=\xcoch(T)$, we have the line bundle $\calL(\lambda)$ on $\Mpar_{\Gd}$ associated to $\calL^{\Td}$ and the character $\lambda:\Td\to\GG_m$. The Chern class of $\calL(\lambda)$ gives a map $c_1(\calL(\lambda)):\const{\Mpar_{\Gd}}\to\const{\Mpar_{\Gd}}[2](1)$. Taking direct image under $\tfd_*$, we get: 
\begin{equation*}
\cup c_1(\calL(\lambda)):\tfd_*\Ql\to\tfd_*\Ql[2](1).
\end{equation*}
We will concentrate on the stable part $L$ of $\tfd_*\Ql[d](d/2)$:
\begin{equation*}
\cup c_1(\calL(\lambda))_{\st}:L\subset\tfd_*\Ql[d](d/2)\xrightarrow{\cup c_1(\calL(\lambda))}\tfd_*\Ql[d+2](d/2+1)\twoheadrightarrow L[2](1).
\end{equation*}

\begin{lemma}\label{l:chernvan}
The map $\cup c_1(\calL(\lambda))_{\st}$ induces the zero map $L^i\to L^{i+2}(1)$ for each $i\in\ZZ$.
\end{lemma}
We postpone the proof of this lemma to Sec. \ref{ss:Ch}.

By Lem. \ref{l:chernvan}, we may also consider the ``subdiagonal'' entries of the map $\cup c_1(\calL(\lambda))_{\st}$ under the perverse filtration. More precisely, by Lem. \ref{l:chernvan}, the map $\ptau_{\leq i}(\cup c_1(\calL(\lambda))_{\st})$ factors through
\begin{equation*}
\ptau_{\leq i}(\cup c_1(\calL(\lambda))_{\st}):\ptau_{\leq i}L\to\ptau_{\leq i-1}(L[2](1)).
\end{equation*}
Taking the $i$-th perverse cohomology, we get
\begin{equation}\label{eq:defCh}
\Ch^i(\lambda)=\pH^i(\cup c_1(\calL(\lambda))_{\st}):L^i\to L^{i+1}[1](1).
\end{equation}

Now we can state the second result in this section, which is an identification of the above two lattice actions under the Verdier duality given in Th. \ref{th:VD}.

\begin{theorem}\label{th:LD}
For each $i\in\ZZ$ and $\lambda\in\xcoch(T)=\xch(\Td)$, we have a commutative diagram
\begin{equation}\label{d:LD}
\xymatrix{K^{-i}\ar[r]^{\tilv^{-i}}_{\sim}\ar[d]_{\Sp^{-i}(-\lambda)}^{\wr} & \DD K^i\ar[d]^{(\DD\Sp^{i+1}(\lambda))[1]}_{\wr}\ar[rr]^{\widetilde{\VD}^i}_{\sim} & & L^i(i)\ar[d]^{\Ch^i(\lambda)(i)}_{\wr}\\
K^{-i-1}[1]\ar[r]^{\tilv^{-i-1}[1]}_{\sim} & \DD K^{i+1}[1]\ar[rr]^{\widetilde{\VD}^{i+1}}_{\sim} & & L^{i+1}[1](i+1)}
\end{equation}
Here the isomorphism $\widetilde{\VD}^i$ is given in Th. \ref{th:VD} and the isomorphism $\tilv^i$ is given in Rem. \ref{rm:anotherVD}.
\end{theorem}

We will analyze the two lattice actions in more details in Sec.s \ref{ss:Sp} and \ref{ss:Ch} respectively, and finally prove Th. \ref{th:LD} in Sec. \ref{ss:PfLD}.


\subsection{The Tate modules of the Picard stack}\label{ss:homoP}
Let $\Ag\subset\calA$ be the open dense locus where the cameral curve is transversal to the discriminant divisor in $\frt_D$ (cf. \cite[4.6]{NgoFL}). By \cite[Lemme 4.6.3]{NgoFL}, for $a\in\Ag(k)$ the cameral curve $X_a$ is smooth. In this subsection, we study the stable parts of the homology of the Picard stack $\calP$ and the cohomology of the Picard stack $\calPd$ (the Picard stack acting on $\MHit_{\Gd}$) over $\Ag$. By \cite[Lem. A.2.3]{GSII}, we only need to understand the Tate modules of $\calP^0/\Ag$ and $\calPdc/\Ag$.

\begin{lemma}\label{l:Tatemod}
There are natural isomorphisms of local systems on $\Ag$:
\begin{eqnarray}\label{eq:TateP}
\VP\cong(\homo{1}{\tcAg/\Ag}\otimes_\ZZ\xcoch(T))^W;\\
\label{eq:TatePd}
\VPd\cong(\coho{1}{\tcAg/\Ag}\otimes_\ZZ\xch(\Td))_W.
\end{eqnarray}
\end{lemma}
\begin{proof}
We prove the first isomorphism; the second one is similar. By \cite[Construction 3.2.8]{GSI}, we have a morphism of Picard stacks over $\calA$:
\begin{equation}\label{eq:PtoWPic2}
\jmath_{\calP}:\calP\to\stPic_{T}(\tcA/\calA)^W.
\end{equation}
In Lem. \ref{l:Pisogeny}, we will show that $\jmath_{\calP}$ is an isogeny (i.e., the kernel and cokernel of $\jmath_{\calP}$ are finite over $\calA$). In particular, over $\Ag$, $\jmath_{\calP}$ induces an isomorphisms on the $\Ql$-Tate modules of the neutral components:
\begin{equation}\label{eq:compare}
V_\ell(\jmath_{\calP}):V_\ell(\calP^0/\Ag)\cong V_\ell(\stPic^0_{T}(\tcAg/\Ag)^W/\Ag).
\end{equation}
Here $\stPic^0_T(\tcAg/\Ag)$ is the neutral component of $\stPic_T(\tcAg/\Ag)$. Note however that the fibers of $\stPic^0_{T}(\tcAg/\Ag)^W$ are not necessarily connected; by Tate modules we mean the Tate module of their neutral component.
 
For $a\in\calA(k)\subset\calA^{\ani}(k)$, the automorphisms of objects in $\stPic_T(X_a)^W$ are $T(X_a)^W=T^W$ which is finite (here we use the fact that $X_a$ is connected). Therefore $\stPic_T(\tcA)^W$ is a Deligne-Mumford stack, and it is harmless to replace $\stPic(\tcA)^W$ by its coarse moduli space $\Pic(\tcA/\calA)^W$ in calculating Tate modules.

We have an isomorphism
\begin{equation}\label{eq:WinvV}
V_{\ell}(\Pic^0_{T}(\tcA/\calA)^W/\calA)\cong(V_\ell(\Pic^0(\tcA/\calA))\otimes_{\ZZ}\xcoch(T))^W.
\end{equation}
For $a\in\Ag(k)$, the cameral curve $X_a$ is smooth, therefore we have a {\em canonical} isomorphism
\begin{equation}\label{eq:curve}
\homo{1}{\tcAg/\Ag}\cong V_\ell(\Pic^0(\tcAg/\Ag)).
\end{equation}
In fact, the Abel-Jacobi map $\tcAg\to\Pic^1(\tcAg/\Ag)$ gives a canonical isomorphism $\homo{1}{\tcAg/\Ag}\cong\homo{1}{\Pic^1(\tcAg/\Ag)}$. But since $\Pic^0(\tcAg/\Ag)$ has connected fibers over $\Ag$, we can canonically identify the local system $\homo{1}{\Pic^1(\tcAg/\Ag)}$ with $\homo{1}{\Pic^0(\tcAg/\Ag)}$ (the argument is as in Lem. \ref{l:KLMP}), hence we get the canonical isomorphism (\ref{eq:curve}).

The isomorphisms (\ref{eq:WinvV}) and (\ref{eq:curve}) give a canonical isomorphism of local systems on $\Ag$:
\begin{equation*}
V_\ell(\stPic^0_{T}(\tcAg/\Ag)^W/\Ag)\cong(\homo{1}{\tcAg/\Ag}\otimes_\ZZ\xcoch(T))^W.
\end{equation*}
This, together with the isomorphism (\ref{eq:compare}), implies the isomorphism (\ref{eq:TateP}).
\end{proof}

It remains to prove:

\begin{lemma}\label{l:Pisogeny}
The morphism $\jmath_{\calP}$ in (\ref{eq:PtoWPic2}) is an isogeny over $\calA$.
\end{lemma}
\begin{proof}
Fix $a\in\calA(S)$. We need a description of $\calP_a$ due to Donagi-Gaitsgory in \cite[16.3]{DG} which we briefly recall here. For each root $\alpha\in\Phi$, let $D_\alpha\subset X_a$ be the divisor given by the pull-back of the wall $\frt_{\alpha,D}\subset\frt_D$. In other words, $D_\alpha$ is the fixed point locus of the action of the reflection $r_\alpha\in W$ on $X_a$. For each object $Q^T\in\stPic_T(X_a)^W(S)$, the $r_\alpha$-equivariant structure of $Q^T$ gives an isomorphism of $T$-torsors over $D_\alpha$:
\begin{equation*}
Q^T|_{D_\alpha}\twtimes{T,r_\alpha}T\cong Q^T|_{D_\alpha}.
\end{equation*}
Equivalently, spelling out the action of $r_{\alpha}$ on $T$, we get an isomorphism of $T$-torsors over $D_\alpha$:
\begin{equation}\label{eq:triv}
\left(Q^T|_{D_\alpha}\twtimes{T,\alpha}\GG_m\right)\twtimes{\GG_m,\alpha^\vee}T\cong T\times D_\alpha.
\end{equation}
The result of Donagi-Gaitsgory in {\em loc.cit.} says that $\calP_a(S)$ is the Picard groupoid of tuples $(Q^T,\{\gamma_w\}_{w\in W},\{\beta_{\alpha}\}_{\alpha\in\Phi})$ where $(Q^T,\{\gamma_w\}_{w\in W})$ is a strongly $W$-equivariant $T$-torsor on $X_a$ and $\beta_{\alpha}$ is a trivialization of the $\GG_m$-torsor $Q(\alpha)|_{D_\alpha}:=Q^T|_{D_\alpha}\twtimes{T,\alpha}\GG_m$, which is compatible with the trivialization (\ref{eq:triv}) and the $W$-equivariant structure: i.e., $\gamma_w$ sends the trivialization $\beta_{\alpha}$ to the trivialization $\beta_{w\alpha}$. 

We give a reformulation of this result of Donagi-Gaitsgory. For each $\alpha\in\Phi$, let
\begin{equation}
\mu_{\alpha}:=\ker(\GG_m\xrightarrow{\alpha^\vee}T).
\end{equation}
This is either the trivial group or the group $\mu_2$, depending on whether $\alpha^\vee$ is a primitive element of $\xcoch(T)$ or not. For $Q^T\in\stPic_T(X_a)^W(S)$, by the trivialization (\ref{eq:triv}), the $\GG_m$-torsor $Q(\alpha)|_{D_\alpha}$ in fact comes from a $\mu_{\alpha}$-torsor $Q^{\mu_\alpha}$ over $D_\alpha$. An object in $\calP_a(S)$ is just an object $(Q^T,\{\gamma_w\}_{w\in W})$ in $\stPic_T(X_a)^W(S)$ together with a trivialization of the $\mu_\alpha$ torsor $Q^{\mu_\alpha}$ over $D_\alpha$ for each $\alpha\in\Phi$, compatible with the $W$-equivariant structure of $Q^T$. Since the above discussion works for any test scheme $S$, we get an exact sequence of Picard stacks
\begin{equation}\label{eq:midexact}
\left(\prod_{\alpha\in\Phi}\Res_{\tcA_\alpha/\calA}(\mu_\alpha\times\tcA_\alpha)\right)^W\to\calP\xrightarrow{\jmath_{\calP}}\stPic_T(\tcA/\calA)^W\to\left(\prod_{\alpha\in\Phi}\stPic_{\mu_{\alpha}}(\tcA_\alpha/\calA)\right)^W.
\end{equation}
Here $\tcA_\alpha\subset\tcA$ is the pull-back of $\frt_{\alpha,D}$, and the last arrow in (\ref{eq:midexact}) sends $Q^T\in\stPic_T(X_a)^W(S)$ to the $\mu_\alpha$-torsor $Q^{\mu_\alpha}$ over $D_\alpha=S\times_{\calA}\tcA_\alpha$.

Since $\tcA_\alpha$ is finite over $\calA$, and $\mu_\alpha$ is a finite group scheme, the two ends of the sequence (\ref{eq:midexact}) are finite Picard stacks. Therefore the morphism $\jmath_{\calP}$ in the middle of (\ref{eq:midexact}) is an isogeny. This completes the proof.
\end{proof}


\subsection{Verdier duality between the Hitchin complexes}\label{ss:PfVD}
In this subsection, we give the proof of Th. \ref{th:VD}.


Using the global Kostant section $\epsilon:\calA\to\MHit_G$ (resp. $\epsilon\dual:\calA\to\MHit_{\Gd}$), we get a morphism $\tau:\calP\to\MHit_G$ (resp. $\tau\dual:\calPd\to\MHit_{\Gd}$). For $a\in\Ag(k)$, we have $\delta(a)=0$, hence by \cite[Lem. 3.5.2]{GSI}, $\MHit_a$ is a torsor under $\calP_a$. This implies that $\tau$ is an isomorphism over $\Ag$, and we also get the induced isomorphisms on homology and cohomology. However, such isomorphisms are not canonical in the sense that they depend on the choice of the Kostant section. Fortunately, if we restrict ourselves to the stable parts of the homology or cohomology, and passing to cohomology sheaves, we do get canonical isomorphisms:

\begin{lemma}\label{l:KLMP}
For each $i\in\ZZ$, there are {\em canonical} isomorphisms
\begin{eqnarray}\label{eq:canKL}
\homo{i}{\MHit_G/\Ag}_{\st}\cong\homo{i}{\calP/\Ag}_{\st};\\
\coho{i}{\MHit_{\Gd}/\Ag}_{\st}\cong\coho{i}{\calPd/\Ag}_{\st}.
\end{eqnarray}
\end{lemma}
\begin{proof}
We prove the first isomorphism. For any \'etale map $S\to\Ag$, and any lifting $m:S\to\MHit_G$, we get a trivialization of the $\calP$-torsor:
\begin{equation*}
\tilm:S\times_{\calA}\calP\isom S\times_{\calA}\MHit_G.
\end{equation*}
Hence we get an isomorphism
\begin{equation}\label{eq:tauisom}
\tilm_{!,\st}:\homo{i}{S\times_{\calA}\calP/S}_{\st}\cong\homo{i}{S\times_{\calA}\MHit_G/S}_{\st}.
\end{equation}
If we choose another lifting $m':S\to\MHit_G$, then $m$ and $m'$ differ by the translation of a section $\varpi:S\to\calP$, and the two isomorphisms $\tilm_{!,\st}$ and $\tilm'_{!,\st}$ differ by the action of $\varpi\in\calP(S)$ on $\homo{i}{S\times_{\calA}\MHit_G/S}_{\st}$. By the homotopy invariance of actions on cohomology (see \cite[Lemme 3.2.3]{LauN}), the action of $\calP(S)$ on $\homo{i}{S\times_{\calA}\MHit_G/S}$ factors through $\pi_0(\calP)(S)$. Since $\pi_0(\calP)$ acts trivially on the stable part of $\homo{*}{S\times_{\calA}\MHit_G/S}$, the action of $\calP(S)$ on $\homo{i}{S\times_{\calA}\MHit_G/S}_{\st}$ is trivial. Therefore different choices of local sections give rise to the same isomorphism as in (\ref{eq:tauisom}), hence the canonicity of the isomorphisms in (\ref{eq:canKL}).
\end{proof}

We also have the following corollary of Prop. \ref{p:kappasupp}:

\begin{cor}\label{c:midext}
For any $i\in\ZZ$, $K^i$ and $L^i$ are middle extensions from their restrictions to any Zariski open dense subset of $\tcA$.
\end{cor}


\begin{proof}[Proof of Theorem \ref{th:VD}]
Let $\tcAg$ be the preimage of $\Ag$ in $\tcA$, and let $\Kg$ and $\Lg$ be the restrictions of $K$ and $L$ to $\tcAg$. By Lem. \ref{c:midext}, both $\DD K^i$ and $L^i$ are middle extensions from $\tcAg$, therefore it suffices to establish a natural isomorphism $\DD\Kg^i\cong\Lg^{i}(i)$.

Let $\tilp:\tcAg\to\Ag$ be the natural projection, which is smooth and proper. Since for any $a\in\Ag(k)$, we have $\delta(a)=0$, hence $\delta(a,x)=0$ for any $x\in X(k)$, therefore $\Mpar_G|_{\Ag}\isom\MHit|_{\Ag}\times_{\Ag}\tcAg$ by \cite[Lem. 3.5.4]{GSI}. Similar remark applies to $\Mpar_{\Gd}|_{\Ag}$. Therefore we get canonical isomorphisms
\begin{eqnarray}\label{eq:tilpK}
&&\DD\Kg^i\cong\homo{i+r}{\Mpar_G/\tcAg}_{\st}(d/2-r)=\tilp^*\homo{i+r}{\MHit_G/\Ag}_{\st}(d/2-r);\\
\label{eq:tilpL}
&&\Lg^i\cong\coho{i+r}{\Mpar_{\Gd}/\tcAg}_{\st}(d/2)=\tilp^*\coho{i+r}{\MHit_{\Gd}/\Ag}_{\st}(d/2). 
\end{eqnarray}
where $r=\dim\Mpar_G-\dim\tcA=\dim\Mpar_{\Gd}-\dim\tcA$ is the relative dimension of $\tilf$ and $\tfd$.

By the isomorphisms in \cite[Lem. A.2.3(2), Rem. A.2.4]{GSII}, and Lem. \ref{l:KLMP}, we have natural isomorphisms
\begin{eqnarray}\label{eq:MGwedge}
\homo{*}{\MHit_G/\Ag}_{\st}\cong\homo{*}{\calP/\Ag}_{\st}\cong\bigwedge(\VP[1]);\\
\label{eq:MGdwedge}
\coho{*}{\MHit_{\Gd}/\Ag}_{\st}\cong\coho{*}{\calPd/\Ag}_{\st}\cong\bigwedge(\VPd^*[-1]).
\end{eqnarray}
Combining with the isomorphisms (\ref{eq:tilpK}) and (\ref{eq:tilpL}), we get
\begin{eqnarray}\label{eq:Kisom}
\DD\Kg^i\cong\tilp^*\bigwedge^{i+r}(\VP)(d/2-r);\\
\label{eq:Lisom}
\Lg^i\cong\tilp^*\bigwedge^{i+r}(\VPd^*)(d/2).
\end{eqnarray}
Therefore, to prove the theorem, it suffices to give a natural isomorphism of local systems on $\Ag$:
\begin{equation}\label{eq:beta}
\beta:\VP\cong \VPd^*(1).
\end{equation}

Now we can use the explicit formula proved in Lem. \ref{l:Tatemod}:
\begin{eqnarray}\label{eq:compV}
V_\ell(\calP^{0}/\Ag)\cong(\homo{1}{\tcAg/\Ag}\otimes_\ZZ\xcoch(T))^W;\\
V_\ell(\calP^{\vee,0}/\Ag)^*\cong(\coho{1}{\tcAg/\Ag}\otimes_\ZZ\xch(\Td))_W.
\end{eqnarray}
Recall from \cite[Prop. 4.5.4]{NgoFL} that the cameral curves $X_a$ are connected. The cup product for the smooth connected projective family of curves $\tcAg\to\Ag$ gives a perfect pairing
\begin{equation*}
\coho{1}{\tcAg/\Ag}\otimes\coho{1}{\tcAg/\Ag}\xrightarrow{\cup}\coho{2}{\tcAg/\Ag}\cong\Ql(-1).
\end{equation*}
Therefore we have a natural isomorphism of local systems
\begin{equation}\label{eq:PDcurve}
\PD:\homo{1}{\tcAg/\Ag}\cong\coho{1}{\tcAg/\Ag}(1).
\end{equation}
Since $\xcoch(T)=\xch(\Td)$, we can define a natural isomorphism
\begin{eqnarray}\notag
\beta^{-1}:(\coho{1}{\tcAg/\Ag}(1)\otimes_\ZZ\xch(\Td))_W&\isom&(\homo{1}{\tcAg/\Ag}\otimes_\ZZ\xcoch(T))^W\\
\label{eq:betaform}
h\otimes\lambda&\mapsto&\sum_{w\in W}w_*\PD^{-1}(h)\otimes w\lambda.
\end{eqnarray}
hence the isomorphism (\ref{eq:beta}). This completes the proof of the theorem.
\end{proof}

\subsection{Proof of Theorem \ref{th:LD}--first reductions}
In this subsection, we make a few reduction steps towards the proof of Th. \ref{th:LD}.

First, the commutativity of the left square follows from \cite[Prop. 4.5.1]{GSI} (or rather its counterpart for $\xcoch(T)$-action on $\tfQl$), therefore we only need to prove the commutativity of the right square. We make the following simple observation about extensions of perverse sheaves.

\begin{lemma}\label{l:control}
Suppose $j:U\hookrightarrow Y$ is the inclusion of a Zariski open subset of the scheme $Y$ and $\calF_1, \calF_2$ are perverse sheaves on $U$. Let $j_{!*}\calF_1$ and $j_{!*}\calF_2$ be the middle extension perverse sheaves on $Y$. Then the restriction map
\begin{equation*}
\Ext^{1}_Y(j_{!*}\calF_1,j_{!*}\calF_2)\to\Ext^1_U(\calF_1,\calF_2)
\end{equation*}
is injective.
\end{lemma}
\begin{proof}
Let $i:Z=Y-U\hookrightarrow Y$ be the closed embedding of the complement of $U$ into $Y$. We have a long exact sequence
\begin{equation}\label{eq:extper}
\to\Hom_Z(i^*j_{!*}\calF_1,i^!j_{!*}\calF_2[1])\to\Hom_Y(j_{!*}\calF_1,j_{!*}\calF_2[1])\xrightarrow{j^*}\Hom_U(\calF_1,\calF_2[1])\to
\end{equation}
By the definition of $j_{!*}$, we have $i^*j_{!*}\calF_1\in\leftexp{p}{D}^{\leq -1}(Z)$ and $i^!j_{!*}\calF_2[1]\in\leftexp{p}{D}^{\geq 0}(Z)$. Hence the first term in (\ref{eq:extper}) vanishes; i.e., $j^*$ is injective. 
\end{proof}

By Lem. \ref{c:midext}, the perverse sheaves $K^i,L^i$ are middle extensions from the open dense subset $\tcAg$ of $\tcA$. By Lem. \ref{l:control}, in order to prove the commutativity of the diagram (\ref{d:LD}), it suffices to prove the commutativity of its restriction to $\tcAg$. 

\begin{remark}\label{rm:adp}
We need the following fact about adjunction. Since $\tilp:\tcAg\to\Ag$ is smooth of relative dimension one, we have a natural isomorphism of functors $\tilp^!\cong\tilp^*[2](1)$. For any two objects $\calF_1,\calF_2\in D^b_c(\Ag)$, we have
\begin{eqnarray*}
\Hom_{\tcAg}(\tilp^*\calF_1,\tilp^*\calF_2)&=&\Hom_{\tcAg}(\tilp^!\calF_1,\tilp^!\calF_2)\\
&=&\Hom_{\Ag}(\tilp_!\tilp^!\calF_1,\calF_2)=\Hom_{\Ag}(\homo{*}{\tcAg/\Ag}\otimes\calF_1,\calF_2)\\
&=&\Hom_{\Ag}(\homo{*}{\tcAg/\Ag},\unHom_{\Ag}(\calF_1,\calF_2))
\end{eqnarray*}
In other words, there is a natural bijection between maps
\begin{equation*}
\phi:\tilp^*\calF_1\to\tilp^*\calF_2
\end{equation*}
and maps
\begin{equation*}
\phi^\natural:\homo{*}{\tcAg/\Ag}\to\unHom_{\Ag}(\calF_1,\calF_2).
\end{equation*}
\end{remark}

By this lemma, and the isomorphisms (\ref{eq:Kisom}) and (\ref{eq:Lisom}), the commutativity of the right square in (\ref{d:LD}) over $\tcAg$ is equivalent to the commutativity of
\begin{equation}\label{eq:LD2}
\xymatrix{& \unHom(\bigwedge^{i}(\VP),\bigwedge^{i+1}(\VP)[1])\ar[dd]^{\unHom(\wedge^i\beta,\wedge^{i+1}\beta[1])}\\
\homo{*}{\tcAg/\Ag}\ar[ur]_{\Sp_i(\lambda)^\natural}\ar[dr]^{\Ch^i(\lambda)^\natural} &\\
& \unHom(\bigwedge^{i}(\VPd^*(1)),\bigwedge^{i+1}(\VPd^*(1))[1]}.
\end{equation}
Note that the labeling of the maps $\Sp(\lambda)^\natural$ and $\Ch(\lambda)^\natural$ have been switched from the perverse degree to the ordinary homological and cohomological degrees.

\subsection{The Springer action by $\xcoch(T)$}\label{ss:Sp}

In (\ref{eq:defSp}), we have defined the map $\Sp^i(\lambda):K^i\to K^{i-1}[1]$. In the end of Sec. \ref{ss:PfVD}, we rewrote the restriction of $\Sp^i(\lambda)$ to $\tcAg$ into the form
\begin{equation*}
\Sp_i(\lambda)^\natural:\homo{*}{\tcAg/\Ag}\to\unHom(\bigwedge^i(\VP),\bigwedge^{i+1}(\VP)[1]).
\end{equation*}
In this subsection, we write the map $\Sp_i(\lambda)^\natural$ in more explicit forms.

\subsubsection{Rewriting the Springer action} Recall from \cite[Rem. 4.3.7]{GSI} that we have a morphism over $\calA$
\begin{equation*}
s_\lambda:\tcAg\to\calP.
\end{equation*}
Since $\tcAg\to\Ag$ have connected fibers, $s_\lambda$ necessarily factors through the neutral component $\calP^0\subset\calP$. By \cite[Cor. 4.3.8]{GSI}, the action of $\lambda$ on $\Mpar|_{\tcAg}$ is given by
\begin{equation*}
\Mpar_G|_{\tcAg}=\tcAg\times_{\Ag}\MHit_G|_{\Ag}\xrightarrow{(s_\lambda,\id)}\calP^0|_{\Ag}\times_{\Ag}\MHit_G|_{\Ag}\xrightarrow{\act}\MHit_G|_{\Ag}.
\end{equation*}
Passing to the level of homology, we get
\begin{eqnarray*}
\homo{*}{\tcAg/\Ag}\boxtimes_{\Ag}\homo{*}{\MHit_G/\Ag}&\xrightarrow{s_{\lambda,*}\boxtimes\id}& \homo{*}{\calP^0/\Ag}\boxtimes_{\Ag}\homo{*}{\MHit_G/\Ag}\\
&\xrightarrow{\act_*}&\homo{*}{\MHit_G/\Ag}.
\end{eqnarray*}
By adjunction, we get
\begin{equation}\label{eq:homend}
\homo{*}{\tcAg/\Ag}\xrightarrow{s_{\lambda,*}}\homo{*}{\calP^0/\Ag}\xrightarrow{\cap}\unEnd(\homo{*}{\MHit_G/\Ag}).
\end{equation}
where $\cap$ is (the dual of) the cap product defined in \cite[App. A.3]{GSII}. Passing to the stable part, and using the isomorphism (\ref{eq:MGwedge}), we can rewrite (\ref{eq:homend}) as
\begin{equation*}
\homo{*}{\tcAg/\Ag}\xrightarrow{s_{\lambda,*}}\bigwedge(\VP[1])\xrightarrow{\wedge}\unEnd(\bigwedge(\VP[1])).
\end{equation*}
Here the cap product action becomes the wedge product in $\bigwedge(\VP[1])$.

We decompose the map $s_{\lambda,*}$ into $\oplus_i s_{\lambda,i}$ according to the canonical decomposition in \cite[Lem. A.1.1]{GSII}, hence get
\begin{equation*}
s_{\lambda,i}:\homo{*}{\tcAg/\Ag}\to\homo{i}{\calP^0/\Ag}[i].
\end{equation*}
Note that
\begin{equation*}
\homo{*}{\tcAg/\Ag}\xrightarrow{s_{\lambda,0}}\homo{0}{\calP^0/\Ag}=\Ql\to\unEnd(\bigwedge(\VP))
\end{equation*}
corresponds to the identity map of $\tilp^*\homo{*}{\calP^0/\Ag}\cong\homo{*}{\Mpar_G/\tcAg}_{\st}$ under the adjunction in Rem. \ref{rm:adp}. Therefore the action of $\lambda-\id$ on $\homo{*}{\Mpar_G/\tcAg}_{\st}$ is adjoint to
\begin{equation*}
\homo{*}{\tcAg/\Ag}\xrightarrow{\sum_{i\geq1}s_{\lambda,i}}\bigoplus_{i\geq1}\bigwedge^{i}(\VP[1])\xrightarrow{\cap}\unEnd(\bigwedge(\VP[1])).
\end{equation*}
Restricting to the degree -1 part and denote $s_{\lambda,1}$ by $\Phi_{\lambda}$, we conclude that $\Sp_i(\lambda)^\natural$ can be written as
\begin{eqnarray}\label{eq:finalSp}
\Sp_i(\lambda)^\natural:\homo{*}{\tcAg/\Ag}\xrightarrow{\Phi_{\lambda}}\VP[1]\\
\xrightarrow{\wedge}\unHom(\bigwedge^i(\VP),\bigwedge^{i+1}(\VP)[1]).
\end{eqnarray}

Now we need to understand the map $\Phi_\lambda$ more explicitly. For this, we first describe the morphism $s_{\lambda}:\tcAg\to\calP$ in more concrete terms. Consider the composition
\begin{equation*}
\tcAg\xrightarrow{s_{\lambda}}\calP\xrightarrow{\jmath_{\calP}}\stPic_{T}(\tcA/\calA)^W\to\stPic_T(\tcA/\calA).
\end{equation*}
This morphism gives a $T$-torsor $\calQ^T_\lambda$ on $\tcAg\times_{\calA}\tcA$. Hence for each $\xi\in\xch(T)$, we get a line bundle $\calQ_\lambda(\xi)$ on $\tcAg\times_{\calA}\tcA=\tcAg\times_{\Ag}\tcAg$. We now describe this line bundle.

\begin{lemma}\label{l:sectionlb} 
Let $\Gamma_w=\{(\tilx,w\tilx)|\tilx\in\tcAg\}\subset\tcAg\times_{\Ag}\tcAg$ be the graph of the left $w$-action, viewed as a divisor of $\tcAg\times_{\Ag}\tcAg$, then there is a canonical isomorphism
\begin{equation}\label{eq:calF}
\calQ_\lambda(\xi)\cong\calO(\sum_{w\in W}\jiao{w\lambda,\xi}\Gamma_w).
\end{equation}
\end{lemma}
\begin{proof}
By \cite[Lem. 4.3.6]{GSI}, the morphism $s_\lambda$ comes from a section $\tils_\lambda:\tcAg\to\tilGr_J$; i.e., we have a $J$-torsor $\calQ^J_\lambda$ on $\tcAg\times X$ with a canonical trivialization away from the graph $\Gamma$ of $\tcAg\to X$. By construction, we have
\begin{equation*}
\calQ^T_\lambda=(\id\times q)^*\calQ^J_\lambda\twtimes{q^*J}T
\end{equation*}
where $\id\times q:\tcAg\times_{\Ag}\tcAg\to\tcAg\times_{\Ag}(\Ag\times X)=\tcAg\times X$. Therefore, $\calQ^T_\lambda$ has a canonical trivialization $\tau$ over $\tcAg\times_{\Ag}\tcAg-(\id\times q)^{-1}(\Gamma)=\tcAg\times_{\Ag}\tcAg-\bigcup_{w\in W}\Gamma_w$.

On the other hand, the section $\tils_\lambda$ over $\tcArs$ is defined by the composition
\begin{equation*}
\tcArs\xrightarrow{\id\times\{\lambda\}}{\tcArs\times\xcoch(T)}\isom(\tilGr^{\rs}_T)^{\red}\hookrightarrow\tilGr^{\rs}_T\isom\tilGr^{\rs}_J.
\end{equation*}
Here the last isomorphism is the inverse of the one defined in \cite[Lem. 4.3.5]{GSI}. We first look at the morphism $\tils'_\lambda:\tcArs\to\tilGr^{\rs}_T$. This amounts to giving a $T$-torsor $\calG^T_{\lambda}$ on $\tcArs\times_{\Ag}\tcAg$ with a trivialization outside the diagonal (i.e., the graph $\Gamma^{\rs}_e$). Clearly, the associated line bundle $\calG_{\lambda}(\xi)$ with the induced trivialization on $\tcAgrs\times_{\Ag}\tcAg-\Gamma^{\rs}_e$ has the form:
\begin{equation*}
\calG_{\lambda}(\xi)=\calO(\jiao{\lambda,\xi}\Gamma^{\rs}_e). 
\end{equation*}
Here, we define $\Gamma^{\rs}_w=\Gamma^{\rs}\cap(\tcAgrs\times_{\Ag}\tcAg)$.

Next, by the construction of the isomorphism $\tilGr^{\rs}_T\isom\tilGr^{\rs}_J$ in \cite[Lem. 4.3.5]{GSI}, the line bundles $\calQ_{\lambda}(\xi)$ and $\calG_{\lambda}(\xi)$ are canonically isomorphic over the open subset
\begin{equation*}
\tcAgrs\times_{\Ag}\tcAg-\bigsqcup_{w\in W,w\neq e}\Gamma^{\rs}_w.
\end{equation*}
In view of the $W$-invariance of $\calQ_{\lambda}(\xi)$, we must have a canonical isomorphism
\begin{equation}\label{eq:Frs}
\calQ_\lambda(\xi)|_{\tcAgrs\times_{\Ag}\tcAg}\cong\calO(\sum_{w\in W}\jiao{w\lambda,\xi}\Gamma^{\rs}_w).
\end{equation}
Moreover, the trivialization of $\calQ^T_\lambda|_{\tcAgrs\times_{\Ag}\tcAg}$ on $(\tcAgrs\times_{\Ag}\tcAg-\bigcup_{w\in W}\Gamma^{\rs}_w)$ given in the isomorphism (\ref{eq:Frs}) is the same as the trivialization $\tau$ of $\calQ^T_\lambda$ on $(\tcAg\times_{\Ag}\tcAg-\bigcup_{w\in W}\Gamma_w)$ given by $\tils_\lambda$. Therefore the expression (\ref{eq:calF}) holds over
\begin{equation*}
(\tcAg\times_{\Ag}\tcAg-\bigcup_{w\in W}\Gamma_w)\bigcup(\tcArs\times_{\Ag}\tcAg)=\tcAg\times_{\Ag}\tcAg-\bigcup_{w\in W}(\Gamma_w-\Gamma^{\rs}_w).
\end{equation*}
Since $\bigcup_{w\in W}(\Gamma_w-\Gamma^{\rs}_w)$ has codimension at least two in the smooth variety $\tcAg\times_{\Ag}\tcAg$, the expression (\ref{eq:calF}) must hold on the whole $\tcAg\times_{\Ag}\tcAg$.
\end{proof}

Consider the degree $-1$ part of $\Phi_\lambda$
\begin{equation*}
\Phi_{\lambda,1}:\homo{1}{\tcA/\calA}\to\VP.
\end{equation*}


\begin{lemma}\label{l:computePhi} Under the isomorphism (\ref{eq:TateP}), the map $\Phi_{\lambda,1}$ is given by
\begin{eqnarray*}
\homo{1}{\tcAg/\Ag}&\to&(\homo{1}{\tcAg/\Ag}\otimes_{\ZZ}\xcoch(T))^W\\
h&\mapsto&\sum_{w\in W}w_*h\otimes w\lambda.
\end{eqnarray*}
\end{lemma}
\begin{proof}
This is a statement about a map between local systems, hence it suffices to check it on the stalks of geometric points. We fix $a\in\Ag(k)$. For each $\xi\in\xch(T)$, consider the morphism
\begin{equation}\label{eq:sss}
X_a\xrightarrow{s_\lambda}\calP_a\to\stPic_{T}(X_a)\xrightarrow{I_\xi}\stPic(X_a).
\end{equation}
where $I_\xi$ sends a $T$-torsor to the line bundle associated to the character $\xi$. Since $\pi_0(\calP_a)$ is torsion, the map (\ref{eq:sss}) must land in $\stPic^0(X_a)$. By Lem. \ref{l:sectionlb}, the morphism (\ref{eq:sss}) takes $\tilx\in X_a$ to the line bundle $\calO(\sum_{w\in W}\jiao{w\lambda,\xi}w\tilx)\in\Pic^0(X_a)$. Therefore it induces the following map on homology:
\begin{eqnarray}\notag
\homog{1}{X_a}&\to&\homog{1}{\Pic^{0}(X_a)}\cong\homog{1}{X_a}\\
\label{eq:xipart}
h&\mapsto&\sum_{w\in W}\jiao{w\lambda,\xi}w_*h.
\end{eqnarray}
Here we use the Picard scheme $\Pic$ rather than the Picard stack $\stPic$ without losing information about Tate modules. We can rewrite (\ref{eq:xipart}) as
\begin{eqnarray*}
\homog{1}{X_a}\xrightarrow{\Phi_{\lambda,1}}&(\homog{1}{X_a}\otimes_{\ZZ}\xcoch(T))^W&\xrightarrow{\id\otimes\jiao{\xi,-}}\homog{1}{X_a}\\
h\mapsto&\Phi_{\lambda,1}(h)&\mapsto\sum_{w\in W}\jiao{w\lambda,\xi}w_*h
\end{eqnarray*}
which immediately implies
\begin{equation*}
\Phi_{\lambda,1}(h)=\sum_{w\in W}w_*h\otimes\leftexp{w}{\lambda}.
\qedhere
\end{equation*}
\end{proof}


\subsection{The Chern class action by $\xch(\Td)$}\label{ss:Ch}

\subsubsection{Rewriting the Chern class action} Recall from \cite[Construction 3.2.8]{GSI} that we have a tautological $\Td$-torsor $\calQ^{\Td}$ over $\tcPd$, and the associated line bundle $\calQ(\lambda)$ for each $\lambda\in\xch(\Td)$. By \cite[Lem. 3.2.5]{GSI} that we have a commutative diagram
\begin{equation}\label{d:PMtorsor}
\xymatrix{\tcPd\times_{\tcA}\Mpar_{\Gd}\ar[rr]^{\act}\ar@<-4ex>[d]_{\calQ^{\Td}}\ar@<3ex>[d]^{\calL^{\Td}} & & \Mpar_{\Gd}\ar[d]^{\calL^{\Td}}\\
\BB\Td\times\BB\Td\ar[rr]^{\textup{mult.}} & & \BB \Td}
\end{equation} 

Using the Kostant section $\epsilon:\calA\to\MHitreg_{\Gd}$, we get a section $\tilep:\tcA\to\MHitreg_{\Gd}\times_{\calA}\tcA\cong\Mparreg_{\Gd}\subset\Mpar_{\Gd}$ (see \cite[Lem. 3.2.7]{GSI}) and a morphism
\begin{equation*}
\tiltau:\tcPd=\tcPd\times_{\tcA}\tcA\xrightarrow{(\id,\tilep)}\tcPd\times_{\tcA}\Mpar_{\Gd}\xrightarrow{\act}\Mpar_{\Gd}.
\end{equation*}
which is an isomorphism over $\tcAg$. By diagram (\ref{d:PMtorsor}), we get
\begin{eqnarray}\notag
\tiltau^*\calL(\lambda)&=&(\id\times\tilep)^*(\calQ(\lambda)\boxtimes_{\tcA}\calL(\lambda))\cong\calQ(\lambda)\boxtimes_{\tcA}\tilep^*\calL(\lambda)\\
\label{eq:LPandL}
&=&\calQ(\lambda)\otimes\widetilde{g}^{\vee,*}\tilep^*\calL(\lambda).
\end{eqnarray}
where $\widetilde{g}^\vee:\tcPd\to\tcA$ is the projection.

By \cite[Lem. 5.1.3]{GSII}, the Chern class of the line bundle $\calQ(\lambda)$ can be written as
\begin{equation*}
c_1(\calQ(\lambda))^\natural:\homo{*}{\tcA/\calA}\to\coho{1}{\calPd/\calA}_{\st}[1](1)\subset\coho{*}{\calPd/\calA}[2](1).
\end{equation*}
The line bundle $\tilep^*\calL(\lambda)$ on $\tcA$ also induces a map
\begin{equation*}
c_1(\tilep^*\calL(\lambda))^\natural:\homo{*}{\tcA/\calA}\to\Ql[2](1)\cong\coho{0}{\calPd/\calA}_{\st}[2](1)\subset\coho{*}{\calPd/\calA}[2](1).
\end{equation*}
Putting together, using (\ref{eq:LPandL}), we can write the Chern class of $\tiltau^*\calL(\lambda)$ as:
\begin{eqnarray}\label{eq:chern12}
c_1(\tiltau^*\calL(\lambda))^\natural&=&c_1(\tilep^*\calL(\lambda))^{\natural}\oplus c_1(\calQ(\lambda))^{\natural}:\\
\notag &&
\homo{*}{\tcA/\calA}\to\coho{0}{\calPd/\calA}_{\st}[2](1)\oplus\coho{1}{\calPd/\calA}_{\st}[1](1).
\end{eqnarray}

\begin{proof}[Proof of Lem. \ref{l:chernvan}]
Since both $L^i$ and $L^{i+2}$ are middle extensions from $\tcAg$, it is enough to check this statement over $\tcAg$. Using the adjunction in Rem. \ref{rm:adp} and the isomorphism (\ref{eq:tilpL}), we can write the action of $c_1(\calL(\lambda))$ as:
\begin{equation}\label{eq:firstch}
c_1(\calL(\lambda))^\natural:\homo{*}{\tcAg/\Ag}\to\coho{*}{\MHit_{\Gd}/\Ag}[2](1)\xrightarrow{\cup}\unEnd(\coho{*}{\MHit_{\Gd}/\Ag})[2](1)
\end{equation}
where $\cup$ is the cup product on $\coho{*}{\MHit_{\Gd}/\Ag}$. Using the trivialization $\tau\dual:\calPd|_{\Ag}\isom\MHit_{\Gd}|_{\Ag}$ to identify $\coho{*}{\MHit_{\Gd}/\Ag}$ with $\coho{*}{\calPd/\Ag}$, the isomorphism (\ref{eq:firstch}) becomes
\begin{equation*}
c_1(\tiltau^*\calL(\lambda))^\natural:\homo{*}{\tcAg/\Ag}\to\coho{*}{\calPd/\Ag}[2](1)\xrightarrow{\cup}\unEnd(\coho{*}{\calPd/\Ag})[2](1).
\end{equation*}

Using (\ref{eq:chern12}), the effect of $c_1(\tiltau^*\calL(\lambda))^\natural$ on the stable part is:
\begin{eqnarray}\notag
\homo{*}{\tcAg/\Ag}&\xrightarrow{c_1(\tilep^*\calL(\lambda))^{\natural}\oplus c_1(\calQ(\lambda))^{\natural}}&\coho{0}{\calPd/\Ag}_{\st}[2](1)\oplus\coho{1}{\calPd/\Ag}_{\st}[1](1)\\
\label{eq:almostch}
&\xrightarrow{\cup}&\unEnd(\coho{*}{\calPd/\Ag}_{\st})[2](1).
\end{eqnarray}
Since the image of $c_1(\tilep^*\calL(\lambda))^{\natural}\oplus c_1(\calQ(\lambda))^{\natural}$ only involves cohomology sheaves in degree $\leq1$, using Rem. \ref{rm:adp} backwards, we see that $\cup c_1(\calL(\lambda))_{\st}$ sends $\ptau_{\leq i}\Lg$ to $\ptau_{\leq i+1}\Lg$. This proves the lemma.
\end{proof}

Since we will be concentrating on the degree 1 part of $c_1(\calL(\lambda))$, we can ignore the contribution of $c_1(\tilep^*\calL(\lambda))$ in (\ref{eq:almostch}). Using (\ref{eq:almostch}) and the isomorphism in \cite[Rem. A.2.4]{GSII}, we can finally write $\Ch^i(\lambda)^\natural$ as
\begin{eqnarray}\notag
\Ch^i(\lambda)^\natural:\homo{*}{\tcAg/\Ag}&\xrightarrow{\Psi_\lambda}&\VPd^*(1)[1]\\
\label{eq:finalCh}
&\xrightarrow{\wedge}&\unHom(\bigwedge^i(\VPd^*(1)),\bigwedge^{i+1}(\VPd^*(1))[1])
\end{eqnarray}
where $\Psi_\lambda=c_1(\calQ(\lambda))^{\natural}$, and the cup product becomes the wedge product.

Consider the degree $-1$ part of $\Psi_\lambda$
\begin{equation*}
\Psi_{\lambda,1}:\homo{1}{\tcAg/\Ag}\to\VPd^*(1).
\end{equation*}


\begin{lemma}\label{l:computepsi} Under the isomorphism (\ref{eq:TatePd}), the map $\Psi_{\lambda,1}$ is given by
\begin{eqnarray}\notag
\homo{1}{\tcAg/\Ag}&\to&(\coho{1}{\tcAg/\Ag}(1)\otimes_{\ZZ}\xch(\Td))_W\\
\label{eq:computepsi}
h&\mapsto&\PD(h)\otimes\lambda.
\end{eqnarray}
where $\PD$ is the Poincar\'e duality isomorphism defined in (\ref{eq:PDcurve}).
\end{lemma}
\begin{proof}
This is a statement about a map between local systems, hence it suffices to check it on the stalks of geometric points. We fix $a\in\Ag(k)$. Consider the morphism
\begin{equation*}
\id\times\jmath_\lambda:X_a\times\calP_a\xrightarrow{\id\times\jmath_a}X_a\times\stPic_{\Td}(X_a)\xrightarrow{\id\times I_\lambda}X_a\times\stPic(X_a)
\end{equation*}
where $I_\lambda$ sends a $\Td$-torsor to the line bundle induced by that $\Td$-torsor and the character $\lambda\in\xch(\Td)$. Since $\pi_0(\calP_a)$ is torsion, the morphism $\jmath_\lambda$ necessarily lands in $\stPic(X_a)$. Let $\Poin$ is the Poincar\'e line bundle on $X_a\times\stPic^0(X_a)$, then
\begin{equation*}
\calQ(\lambda)=(\id\times\jmath_a)^*\Poin(\lambda)=(\id\times\jmath_\lambda)^*\Poin,
\end{equation*}
It is well-known that $c_1(\Poin)$ takes the form
\begin{equation}
c_1(\Poin)=\sum_ih^i\otimes\PD(h_i)\in\cohog{1}{X_a}\otimes\cohog{1}{\Pic^0(X_a)}(1)\cong\cohog{1}{X_a}\otimes\cohog{1}{X_a}(1)
\end{equation}
where $\{h^i\}$ and $\{h_i\}$ are dual bases of $\cohog{1}{X_a}$ and $\homog{1}{X_a}$. Therefore, $c_1(\calQ(\lambda))=(\id\times\jmath_\lambda)^*c_1(\LP)$ is the image of $\sum_ih^i\otimes\PD(h_i)\otimes\lambda\in\cohog{1}{X_a}\otimes\cohog{1}{X_a}(1)\otimes_{\ZZ}\xch(\Td)$ in $\cohog{1}{X_a}\otimes(\cohog{1}{X_a}(1)\otimes_{\ZZ}\xch(\Td))_W\cong\cohog{1}{X_a}\otimes\cohog{1}{\calP_a}_{\st}(1)$. This immediately implies (\ref{eq:computepsi}).
\end{proof}


\subsection{Proof of Theorem \ref{th:LD}}\label{ss:PfLD}
In this subsection we finish the proof of Th. \ref{th:LD}. By the reduction in the end of Sec. \ref{ss:PfVD}, the expression (\ref{eq:finalSp}) for $\Sp_i(\lambda)^\natural$ and the expression (\ref{eq:finalCh}) for $\Ch^i(\lambda)^\natural$, it remains to prove the commutativity of
\begin{equation*}
\xymatrix{& \VP[1]\ar[dd]^{\beta[1]}\\
\homo{*}{\tcAg/\Ag}\ar[ur]^{\Phi_\lambda}\ar[dr]_{\Psi_\lambda} &\\
& \VPd^*[1](1)}.
\end{equation*}

Both maps $\Phi_\lambda$ and $\Psi_\lambda$ necessarily factor through $\tau_{\geq-1}\homo{*}{\tcAg/\Ag}$. We computed $\Phi_{\lambda,1}$ in Lem. \ref{l:computePhi} and computed $\Psi_{\lambda,1}$ in Lem. \ref{l:computepsi}. Comparing the two results with the way we defined the isomorphism $\beta$ in (\ref{eq:betaform}), we conclude that for all $\lambda\in\xcoch(T)=\xch(\Td)$,
\begin{equation*}
\beta\circ\Phi_{\lambda,1}=\Psi_{\lambda,1}.
\end{equation*}
Therefore the difference $\beta\circ\Phi_\lambda-\Psi_\lambda$ must factor through a map
\begin{equation*}
\Delta_\lambda:\homo{0}{\tcAg/\Ag}\cong\Ql\to\VPd^*[1](1).
\end{equation*}
All we need to show is $\Delta_\lambda=0$. We want to reduce the problem to showing $\Delta_{\alpha^\vee}=0$ for simple coroots $\alpha^\vee$. For this we need


\begin{lemma}\label{l:add}
The maps $\Phi_{\lambda}$ and $\Psi_{\lambda}$ are additive in $\lambda$. 
\end{lemma}
\begin{proof}
For $\Psi_{\lambda}$, since $\calQ(\lambda+\mu)\cong\calQ(\lambda)\otimes\calQ(\mu)$, we have
\begin{equation*}
c_1(\calQ(\lambda+\mu))=c_1(\calQ(\lambda))+c_1(\calQ(\mu)),
\end{equation*}
which implies the additivity of $\Psi_\lambda$.

For $\Phi_{\lambda}$, recall that it comes from the morphism $s_\lambda:\tcAg\to\calP$. These morphisms are additive in $\lambda$ (using the multiplication of $\calP$):
\begin{equation*}
\xymatrix{\tcAg\ar[r]^(.4){(s_\lambda,s_\mu)}\ar@/_1.5pc/[rr]_{s_{\lambda+\mu}} & \calP\times_{\calA}\calP\ar[r]^(.6){\mult} & \calP}.
\end{equation*}
Therefore the induced maps on homology satisfies the commutative diagram
\begin{equation*}
\xymatrix{\homo{*}{\tcAg/\Ag}\ar[rr]^(.4){s_{\lambda,*}\otimes s_{\mu,*}}\ar@/_1.8pc/[rrrr]_{s_{\lambda+\mu,*}} & & \homo{*}{\calP/\Ag}\otimes\homo{*}{\calP/\Ag}\ar[rr]^(.6){\textup{Pontryagin}} & & \homo{*}{\calP/\Ag}}
\end{equation*}
Taking the degree -1 stable parts, we get the commutative diagram
\begin{equation}\label{d:Phiadd}
\xymatrix{& \VP[1]\otimes\Ql\oplus\Ql\otimes\VP[1]\ar[dd]^{\wedge}\\
\homo{*}{\tcAg/\Ag}\ar[ur]_{\Phi_{\lambda}\otimes\tilp_!+\tilp_!\otimes\Phi_\mu}\ar[dr]^{\Phi_{\lambda+\mu}} &\\&\VP[1]}
\end{equation}
where $\tilp_!:\homo{*}{\tcAg/\Ag}\to\homo{*}{\Ag/\Ag}=\Ql$ is the push-forward along $\tilp:\tcAg\to\Ag$. The diagram (\ref{d:Phiadd}) implies that $\Phi_{\lambda+\mu}=\Phi_{\lambda}+\Phi_{\mu}$.
\end{proof}

Using Lem. \ref{l:add}, and observe that the RHS of $\Delta_\lambda$ is torsion-free, we conclude that in order to show $\Delta_\lambda=0$ for all $\lambda\in\xcoch(T)$, it suffices to show it for a $\QQ$-basis of $\xcoch(T)_{\QQ}$. Hence we can reduce the problem to the following lemma.

\begin{lemma}
For each simple coroot $\alpha^\vee\in\Phi^\vee$, the map $\Delta_{\alpha^\vee}=0$.
\end{lemma}
\begin{proof}
Let $\frt_{\alpha}$ be the wall corresponding to the simple root $\alpha$ in $\frt$. The Killing form $\frt\isom\frtd$ identifies $\frt_{\alpha}$ with $\frtd_{\alpha^\vee}$. Let $\tcAg_\alpha\subset\tcAg$ be the preimage of $\frt_{\alpha,D}$ under the evaluation morphism $\tcAg\to\frt_{D}$. In other words, $\tcAg$ is the fixed point locus of $r_\alpha$ (the reflection in $W$ corresponding to $\alpha$) on $\tcAg$. The morphism $\tcAg_\alpha\to\Ag$ is finite of degree $\deg(D)\#W$ (in particular it is surjective), it induces a surjection $\homo{0}{\tcAg_\alpha/\Ag}\to\homo{0}{\tcAg/\Ag}=\Ql$. Therefore, in order to show that $\Delta_{\alpha^\vee}=0$, it suffices to show that composition
\begin{equation*}
\homo{0}{\tcAg_\alpha/\Ag}\to\homo{0}{\tcAg/\Ag}\xrightarrow{\Delta_{\alpha^\vee}}\VPd^*[1](1)
\end{equation*}
is zero. This composition is given by the difference of the restrictions of $\beta\circ\Phi_{\alpha^\vee}$ and $\Psi_{\alpha^\vee}$ on $\homo{*}{\tcAg_\alpha/\Ag}$. We claim a much stronger vanishing, namely both maps
\begin{eqnarray}\label{eq:vanSp}
s_{\alpha^\vee}:\homo{*}{\tcAg_\alpha/\Ag}\to\homo{*}{\calP/\Ag};\\
\label{eq:vanCh}
c_1(\calQ(\alpha^\vee))^\natural:\homo{*}{\tcAg_\alpha/\Ag}\to\coho{*}{\calPd/\Ag}[2](1)
\end{eqnarray}
are zero.

We first prove the vanishing of (\ref{eq:vanSp}). For this, it suffices to show that the morphism
\begin{equation}\label{eq:toPicW}
\tcAg_\alpha\xrightarrow{s_{\alpha^\vee}}\calP\xrightarrow{\jmath_{\calP}}\stPic_{T}(\tcA/\calA)^W
\end{equation}
is trivial (i.e., factors through the identity element). In other words, the line bundles $\calQ_{\alpha^\vee}(\xi)$ are canonically trivialized on $\tcAg_\alpha\times_{\calA}\tcA$. By Lem. \ref{l:sectionlb}, we have
\begin{equation*}
\calQ_{\alpha^\vee}(\xi)|_{\tcAg_\alpha\times_{\calA}\tcA}\cong\calO\left(\sum_{w\in W}\jiao{w\alpha^\vee,\xi}\Gamma_w(\tcAg_\alpha)\right).
\end{equation*}
where $\Gamma_w(\tcAg_{\alpha})$ is the restriction of the graph $w$ to $\tcAg_\alpha\times_{\calA}\tcA$. Since $r_\alpha$ fixes $\tcAg_\alpha$, we have
\begin{equation*}
\Gamma_w(\tcAg_\alpha)=\Gamma_{wr_\alpha}(\tcAg_\alpha).
\end{equation*}
Therefore, we have an equality of divisors
\begin{eqnarray*}
\sum_{w\in W}\jiao{w\alpha^\vee,\xi}\Gamma_w(\tcAg_\alpha) &=& \sum_{w\in W/\jiao{r_\alpha}}\jiao{w\alpha^\vee+wr_\alpha\alpha^\vee,\xi}\Gamma_{w}(\tcAg_\alpha)\\
&=& \sum_{w\in W/\jiao{r_\alpha}}\jiao{w\alpha^\vee-w\alpha^\vee,\xi}\Gamma_{w}(\tcAg_\alpha)=0.
\end{eqnarray*}
Here $\sum_{w\in W/\jiao{r_\alpha}}$ means summing over the representatives of the cosets $W/\jiao{r_\alpha}$. Hence $\calQ_{\alpha^\vee}(\xi)|_{\tcAg_{\alpha}\times_{\calA}\tcA}$ is canonically trivialized, i.e., the map (\ref{eq:toPicW}) is zero.

We then prove the vanishing of (\ref{eq:vanCh}). For this, it suffices to show that the tautological line bundle $\calQ(\alpha^\vee)$ is trivial on $\tcA_\alpha\times_{\calA}\calPd$. But this follows from the description of $\calPd$ given in \cite[16.3]{DG}, as we recalled in the proof of Lem. \ref{l:Pisogeny}. Therefore the map (\ref{eq:vanCh}) is also zero. This proves the lemma.
\end{proof}

Tracing the above reductions backwards, we have already completed the proof of Th. \ref{th:LD}.


\section{A sample calculation}\label{s:SL2}
The goal of this section is to calculate the affine Weyl group action on the cohomology of the parabolic Hitchin fibers in the first nontrivial case. We will also partially verify the phenomenon of Langlands duality proved in Sec. \ref{s:LD}. In particular, we will see that the lattice action on the cohomology of parabolic Hitchin fibers is {\em not} semisimple in general.

\subsection{Description of parabolic Hitchin fibers}\label{ss:SL2Mpar}
Throughout this section, we specialize to the case $X=\PP^1$, $\calO_X(D)=\calO(2)$ and $G=\SL(2)$. Now $\AHit=\cohog{0}{\PP^1,\calO(4)}$ parametrizes degree 4 homogeneous polynomials $a(\xi,\eta)=\sum_{i=0}^{4}a_i\xi^i\eta^{4-i}$. For $a\in\AHit$, the cameral curve $X_a$ coincides with the spectral curve $Y_a$, which is a curve in the total space of $\calO(2)$ defined by the equation 
\begin{equation*}
t^2=a(\xi,\eta).
\end{equation*}
Let $p_a:Y_a\to X$ be the projection, which can be viewed as the GIT quotient under the involution $\tau:(\xi,\eta,t)\mapsto(\xi,\eta,-t)$ of $Y_a$. We have $a\in\Aa(k)$ if and only if $Y_a$ is irreducible.

\subsubsection{The Hitchin fibers} For $a\in\Ah(k)$, the Hitchin fiber $\MHit_a$ is:
\begin{equation*}
\MHit_a=\{(\calF,\alpha)|\calF\in\overline{\stPic}(Y_a),\alpha:\det(p_{a,*}\calF)\isom\calO_X\}.
\end{equation*}
For the stack $\overline{\stPic}(Y_a)$ see \cite[Example 3.1.10]{GSI}. For any $(\calF,\alpha)\in\MHit_a$, $\calE=p_{a,*}\calF$ is a rank 2 vector bundle on $X$ with trivial determinant, therefore $\chi(Y_a,\calF)=\chi(X,\calE)=2$, hence $\calF\in\overline{\stPic}^2(Y_a)$. Since
\begin{equation*}
\chi(Y_a,\calO_{Y_a})=\chi(X,\calO_X)+\chi(X,\calO_X(-2))=0,
\end{equation*}
we conclude that for $a\in\Aa(k)$, $Y_a$ is an irreducible curve of arithmetic genus 1. The degree -1 Abel-Jacobi map
\begin{eqnarray}\label{eq:AJ}
Y_a&\to&\overline{\Pic}^{-1}(Y_a)\\
\notag
y&\mapsto&\calI_y \textup{ (the ideal sheaf of }y)
\end{eqnarray}
is an isomorphism, here $\overline{\Pic}^{-1}(Y_a)$ is the compactified Picard scheme, the coarse moduli space of $\overline{\stPic}^{-1}(Y_a)$. Moreover, $\MHit_{a}\cong\overline{\Pic}^2(Y_a)/\mu_2$ where the center $\mu_2\subset\SL(2)$ acts trivially on $\overline{\Pic}^2(Y_a)$. Via the Abel-Jacobi map (\ref{eq:AJ}), $\MHit_a$ is non-canonically isomorphic to $Y_a\times\BB\mu_2$ (we have to choose an isomorphism $\overline{\Pic}^2(Y_a)\cong\overline{\Pic}^{-1}(Y_a)$, which is non-canonical).

The Picard stack $\calP_a$ acting on $\MHit_a$ is the Prym variety
\begin{equation*}
\stPic(Y_a)^{\tau=-1}=\{(\calL,\iota)|\calL\in\stPic^0(Y_a),\iota:\calL\isom\tau^*\calL^{\otimes-1} \textup{ such that }\iota=\tau^*(\iota^{\otimes-1})\}.
\end{equation*}
If $Y_a$ is an irreducible curve of arithmetic genus 1, then $\calP_a\cong\Pic^0(Y_a)\times\BB\mu_2$ (with the trivial action of $\mu_2$).

\subsubsection{The parabolic Hitchin fibers} The parabolic Hitchin fiber $\Mpar_{a,x}$ is
\begin{equation*}
\Mpar_{a,x}=\{(\calF_0,\calF_1,\alpha)|(\calF_0,\alpha)\in\MHit_a,\calF_0\supsetneqq\calF_1\supsetneqq\calF_0(-x)\}.
\end{equation*}
We have two forgetful morphisms:
\begin{eqnarray*}
\rho_0&:&\Mpar_{a,x}\to\overline{\stPic}^2(Y_a),\\
\rho_1&:&\Mpar_{a,x}\to\overline{\stPic}^1(Y_a)
\end{eqnarray*}
sending $(\calF_0,\calF_1,\alpha)$ to $\calF_0$ and $\calF_1$ respectively. As in the case of $\MHit_a$, $\Mpar_{a,x}$ is the quotient of its course moduli schemes by the trivial action of $\mu_2$.

For each partition $\underline{p}$ of 4, let $\calA_{\underline{p}}$ be the locus where the multiplicities of the roots of $a(\xi,\eta)=0$ are given by $\underline{p}$. We have
\begin{equation*}
\delta(a)=\sum_{i}[p_i/2], \textup{ if }a\in\calA_{\underline{p}},\underline{p}=(p_1,p_2,\cdots).
\end{equation*} 
For $x\in X$, let $v_a(x)$ be the order of vanishing of the polynomial $a$ at $x$. Then
\begin{equation*}
\delta(a,x)=[v_a(x)/2].
\end{equation*}
Note that $(\Aa\times X)^{\rs}$ corresponds to the condition $v_a(x)=0$ while $(\Aa\times X)_0$ corresponds to the condition $v_a(x)=0$ or 1. 

Let us analyze the anisotropic parabolic Hitchin fibers on each stratum:
\begin{itemize}
\item $\underline{p}=(1,1,1,1)$. Then $Y_a$ is a smooth curve of genus one; $\Mpar_{a,x}=\MHit_a\times p_a^{-1}(x)$ which is non-canonically isomorphic to $Y_a\times p_a^{-1}(x)\times\BB\mu_2$.

\item $\underline{p}=(2,1,1)$. Then $Y_a$ is a nodal curve of arithmetic genus 1. Let $\pi:\PP^1\to Y_a$ be the normalization. Then the node of $\MHit_a$ (recall the coarse moduli space of $\MHit_{a}$ is isomorphic to $Y_a$, which has a node) corresponds to $\calF=\pi_*\calO_{\PP^1}(1)$. If $v_a(x)=0$ or $1$, then $\Mpar_{a,x}=\MHit_a\times p_a^{-1}(x)$, which is isomorphic to $Y_a\times p_a^{-1}(x)\times\BB\mu_2$. If $v_a(x)=2$, i.e., $x$ is the projection of the node, then the reduced structure of $\Mpar_{a,x}$ (ignore the $\mu_2$-action as well) consists of two $\PP^1$'s meeting transversally at two points: one component (call it $C_1$) corresponds to $\calF_0=\pi_*\calO_{\PP^1}(1)$ and varying $\calF_1$; the other component (call it $C_0$) corresponds to $\calF_1=\pi_*\calO_{\PP^1}$ and varying $\calF_0$.

\item $\underline{p}=(3,1)$. Then $Y_a$ is a cuspidal curve of arithmetic genus 1. Let $\pi:\PP^1\to Y_a$ be the normalization. Then the cusp of $\MHit_a$ corresponds to $\pi_*\calO_{\PP^1}(1)$. If $v_a(x)=0$ or $1$, then $\Mpar_{a,x}=\MHit_a\times p_a^{-1}(x)$, which is isomorphic to $Y_a\times p_a^{-1}(x)\times\BB\mu_2$. If $v_a(x)=3$, i.e., $x$ is the projection of the cusp, then the reduced structure of $\Mpar_{a,x}$ (ignore the $\mu_2$-action as well) consists of two $\PP^1$'s tangent to each other at one point (to the first order): one component (call it $C_1$) has corresponds to $\calF_0=\pi_*\calO_{\PP^1}(1)$ and varying $\calF_1$; the other component (call it $C_0$) corresponds to $\calF_1=\pi_*\calO_{\PP^1}$ and varying $\calF_0$.
\end{itemize}

We also have two other types of spectral curves $Y_a$ which are not irreducible (hence $a\notin\Aa$):
\begin{itemize}
\item $\underline{p}=(2,2)$. Then $Y_a$ is the union of two $\PP^1$ meeting transversally at two points. The two components of $Y_a$ are permuted by the involution $\tau$.

\item $\underline{p}=(4)$. Then $Y_a$ consists of two $\PP^1$'s tangent to each other at one point to the first order. The two components of $Y_a$ are permuted by the involution $\tau$. 
\end{itemize}


\subsection{The $\tilW$-action for a subregular parabolic Hitchin fiber}\label{ss:matrixSL2}
The parabolic Hitchin fiber $\Mpar_{a,x}$ for $a\in\calA_{(2,1,1)}$ and $x$ the projection of the node of $Y_a$ is simplest example of a {\em subregular} parabolic Hitchin fiber (more results on this class of examples will appear in \cite{Subreg}). In this subsection, we compute the action of $\tilW$ on the cohomology of this subregular parabolic Hitchin fiber. To simplify the argument, we work over the ground field $k=\CC$. We ignore the Tate twists in this and the next section.

\subsubsection{A transversal slice} To reduce the dimensions, we will restrict to a transversal slice through $(a,x)$. Fix a point $\infty\in X=\PP^1$ and identify the complement $X-\{\infty\}$ with $\AA^1$. Consider the following map
\begin{eqnarray*}
\iota:\calB=\AA^2_{b,c}&\to&\AHit\times X\\
(b,c)&\mapsto&(x^3+x^2+bx+c,0)
\end{eqnarray*}

We will base change the situation from $\AHit\times X$ to $\calB$. For each $(b,c)\in\calB$, let $Y_{b,c}$ be the spectral curve corresponding to $\iota(b,c)$, and let $Y/\calB$ be the family of spectral curves over $\calB$.

Let us ignore the stack issue from now on because the finite automorphism group $\mu_2$ does not affect the $\Ql$-cohomology. Hence we will work with compactified Picard schemes rather than stacks. Because each spectral curve $Y_{b,c}$ has a unique point $\tilinf$ over the $\infty\in X$, we get a section $\tilinf:\calB\to Y$. We can use this section to identify the various components of $\cPic(Y/\calB)$. More precisely, we can think of $\cPic(Y/\calB)$ as classifying rank one torsion-free coherent sheaves on $Y$ with a rigidification along $\infty$. In particular, via the Abelian-Jacobi map (\ref{eq:AJ}), $\cPic^{-1}(Y/\calB)$ is canonically isomorphic to $Y$; using the section $\tilinf$, all the other components of $\cPic(Y/\calB)$ can also be identified with $Y$.

Let $M$ be the restriction of $\Mpar$ to $\calB$. For each $(b,c)\in\calB$, let $M_{b,c}$ denote the parabolic Hitchin fiber $\Mpar_{\iota(b,c)}$. Therefore we get two maps
\begin{equation*}
\xymatrix{& M\ar[dl]_{\rho_0}\ar[dr]^{\rho_1} &\\ Y\cong\cPic^2(Y/\calB) && Y\cong\cPic^1(Y/\calB)}
\end{equation*}

\subsubsection{Bases in (co)homology} Before calculating the action of $\tilW$ on the cohomology of $M_{0,0}$, we first pick good bases for $\homog{2}{M_{0,0}}$ and $\cohog{2}{M_{0,0}}$. Since the reduced structure of $M_{0,0}$ consists of two components $C_0$ and $C_1$, each isomorphic to $\PP^1$, we get a basis $\{[C_0],[C_1]\}$ for $\homog{2}{M_{0,0}}$. To get a dual basis in $\cohog{2}{M_{0,0}}$, we use the smooth ambient space $M$. Capping with the fundamental class $[M]$ induces an isomorphism
\begin{equation}\label{eq:capM}
\cap[M]:\cohog{2}{M}\isom\hBM{4}{M}.
\end{equation}
Composing the two morphism $\rho_0,\rho_1:M\to Y$ with the the projection $p:Y\to X$, we get two morphisms
\begin{equation*}
p\rho_0,p\rho_1:M\to Y\to X.
\end{equation*}
Fix a general point $x_0\in X-\{0,\infty\}$ such that the above morphisms are smooth over a Zariski neighborhood of $x_0$. For $i=0,1$, let $Z_i=(p\rho_i)^{-1}(x_0)$. Then $Z_i$ is a smooth closed subscheme of $M$. The fundamental classes $[Z_i]\in\hBM{4}{M}$ gives cohomology classes $\zeta_i\in\cohog{2}{M}$ via (\ref{eq:capM}). It is easy to check directly that: $Z_i$ is disjoint from $C_j$ if $i\neq j$; $Z_i$ intersects $C_i$ transversally in two points (corresponding to the two preimage of $x_0$ under $C_i\to Y_{0,0}\to X$). Let $v:M_{0,0}\hookrightarrow M$ be the inclusion, then
\begin{equation*}
\jiao{v^*\zeta_i,[C_j]}=\jiao{\zeta_i,v_*[C_j]}=[Z_i]\cdot[C_j]=2\delta_{i,j}.
\end{equation*}
In other words, $\{\frac{1}{2}v^*\zeta_0,\frac{1}{2}v^*\zeta_1\}$ is a dual basis to $\{[C_0],[C_1]\}$.

We now compute the action of the simple reflections $s_0,s_1$ of the affine Weyl group $\tilW$ of $\SL(2)$.

\subsubsection{Action of $s_1$} Let $H_{s_1}$ be the image of the Hecke correspondence $\calH_{s_1}|_{\calB}$ in $M\times_{\calB}M$. We view $H_{s_1}$ as a correspondence
\begin{equation*}
\xymatrix{& H_{s_1}\ar[dl]_{\overleftarrow{h}}\ar[dr]^{\overrightarrow{h}} \\
M\ar[dr]_{\rho_0} & & M\ar[dl]^{\rho_0}\\ & \cPic^2(Y/\calB)}
\end{equation*}

Since all the classes in $\cohog{0}{M_{0,0}}$ and $\cohog{1}{M_{0,0}}$ and the class $v^*\zeta_0\in\cohog{2}{M_{0,0}}$ are pulled back from the corresponding classes in $\cohog{*}{\cPic^2(Y_{0,0})}$ (i.e., the cohomology of the base of the correspondence $H_{s_1}$), these classes has to be fixed by the cohomological correspondence $[H_{s_1}]_\#$, i.e., $s_1$ acts as identity on $\cohog{0}{M_{0,0}}$, $\cohog{1}{M_{0,0}}$ and $v^*\zeta_0\in\hBM{4}{M}\cong\cohog{2}{M_{0,0}}$.

It remains to calculate the effect of $[H_{s_1}]_\#$ on $v^*\zeta_1$. It is easy to see that
\begin{lemma}
The reduced fiber of $H_{s_1}$ over $(0,0)\in\calB$ is $\Delta(C_0)\cup C_1\times C_1\subset(C_0\cup C_1)\times(C_0\cup C_1)$, where $\Delta(C_0)\subset C_0\times C_0$ is the diagonal.
\end{lemma}

Using this lemma, we see that the fiber of $\overrightarrow{h}^{-1}(Z_1)$ over $(0,0)\in\calB$ consists of two disjoint copies of $C_1$, namely $C_1\times(Z_1\cap C_1)$. Therefore $Z_1':=\overleftarrow{h}\overrightarrow{h}^{-1}(Z_1)$ intersects $C_0$ at two points (possibly with multiplicities). In any case, we must have
\begin{equation}\label{eq:zintc}
[Z_1']\cdot[C_0]\geq2.
\end{equation}

By construction, the action of $[H_{s_1}]_\#$ on $\cohog{2}{M}$ is
\begin{equation*}
[H_{s_1}]_\#:\cohog{2}{M}\xrightarrow{\overrightarrow{h}^*}\cohog{2}{H_{s_1}}\xrightarrow{[H_{s_1}]}\hBM{4}{H_{s_1}}\xrightarrow{\overleftarrow{h}_!}\hBM{4}{M}\xleftarrow{\cap[M]}\cohog{2}{M}.
\end{equation*}
Therefore, it sends $\zeta_1$ to the cohomology class dual to the fundamental class of $\overleftarrow{h}\overrightarrow{h}^{-1}(Z_1)$, i.e., $[Z'_1]$. Therefore, the inequality (\ref{eq:zintc}) implies that
\begin{equation*}
\jiao{s_1(v^*\zeta_1),[C_0]}=\jiao{v^*[H_{s_1}]_\#(\zeta_1),[C_0]}=\jiao{[H_{s_1}]_\#(\zeta_1),v_*[C_0]}\geq2.
\end{equation*}

We conclude that the matrix of the action of $s_1$ on $\cohog{2}{M_{0,0}}$ under the basis $\{\frac{1}{2}v^*\zeta_1,\frac{1}{2}v^*\zeta_2\}$ takes the form
\begin{equation*}
s_1=\left(\begin{array}{cc}1 & a \\0 & b\end{array}\right).
\end{equation*}
for some $a\geq1$. Since $s_1^2=\id$, we must have $b=-1$.

\subsubsection{Action of $s_0$} Next we consider the Hecke correspondence $H_{s_0}$:
\begin{equation*}
\xymatrix{& H_{s_0}\ar[dl]\ar[dr]\\
M\ar[dr]_{\rho_1} & & M\ar[dl]^{\rho_1}\\ & \cPic^1(Y/\calB)}
\end{equation*}
In fact there is an automorphism of $M$ over $\calB$ which interchanges the two morphisms $(\rho_0,\rho_1)$:
\begin{equation*}
(\calF\supset\calF')\mapsto(\calF'(\{\tilinf\})\supset\calF(-\{\tilinf\})).
\end{equation*}
Here the inclusion $\calF(-\{\tilinf\})\hookrightarrow\calF'(\{\tilinf\})$ is induced by
\begin{equation*}
\calF\subset\calF'\otimes p^*\calO_X(\{0\})\cong\calF'\otimes p^*\calO_X(\{\infty\})=\calF'(2\{\tilinf\}).
\end{equation*}
using a fixed an isomorphism $\calO_X(\{0\})\cong\calO_X(\{\infty\})$ in the second step. Therefore, the action of $s_0$ on $\cohog{*}{M_{0,0}}$ can be formally deduced from the $s_1$-action. The action of $s_0$ fixes $\cohog{0}{M_{0,0}}$, $\cohog{1}{M_{0,0}}$ and acts by the matrix
\begin{equation*}
s_0=\left(\begin{array}{cc}-1 & 0 \\a & 1\end{array}\right).
\end{equation*}
under the basis $\{\frac{1}{2}v^*\zeta_0,\frac{1}{2}v^*\zeta_1\}$ of $\cohog{2}{M_{0,0}}$.

\subsubsection{The lattice action} Therefore, the matrix of the action of the translation $\alpha^\vee=s_0s_1$ on $\cohog{2}{M_{0,0}}$ under the same basis takes the from:
\begin{equation*}
\alpha^\vee=\left(\begin{array}{cc}-1 & 0 \\a & 1\end{array}\right)\left(\begin{array}{cc}1 & a \\0 & -1\end{array}\right)=\left(\begin{array}{cc}-1 & -a \\a & a^2-1\end{array}\right).
\end{equation*}
Since the Picard stacks $\calP_a$ are all connected, there is no $\kappa$-part for the parabolic Hitchin complex $\tfQl$ for $\kappa\neq1$ by Lem. \ref{l:latticefinite}. Therefore, the action of $\alpha^\vee$ is unipotent, i.e., the trace of the above matrix must be $2$. This implies that $a^2=4$. Since $a>0$, we must have $a=2$.

In summary, $\tilW$ acts as identity on $\cohog{0}{M_{0,0}}$ and $\cohog{1}{M_{0,0}}$, and under the basis $\{\frac{1}{2}v^*\zeta_0,\frac{1}{2}v^*\zeta_1\}$ of $\cohog{2}{M_{0,0}}$, the elements $s_1,s_0$ and $\alpha^\vee$ act as matrices
\begin{equation}\label{eq:matrix}
s_1=\left(\begin{array}{cc}1 & 2 \\0 & -1\end{array}\right);s_0=\left(\begin{array}{cc}-1 & 0 \\2 & 1\end{array}\right);\alpha^\vee=\left(\begin{array}{cc}-1 & -2 \\2 & 3\end{array}\right).
\end{equation}
In particular, the action of the lattice part of $\tilW$ is unipotent, but not the identity. Also, from the matrices in (\ref{eq:matrix}), we see that $v^*(\zeta_0-\zeta_1)$ is an eigenvector of for the $s_1,s_0$ and $\alpha^\vee$-action with eigenvalue $-1$,$-1$ and $1$.

\subsection{Verification of Theorem \ref{th:LD} in a special case}\label{ss:pervSL2}

In this subsection, we partially verify Th. \ref{th:LD} on the example we calculated in the last section. More precisely, we take $G=\SL(2)$, $\Gd=\PGL(2)$. Let $\tcB=\calB\times_{\AHit\times X}\tcA$ be the restriction of the double cover $\tcA\to\AHit\times X$ to $\calB$. Then there is a unique point in $\tcB$ over $(0,0)\in\calB$, which still denote by $(0,0)$. We will check the commutativity (up to scalar) of the diagram (\ref{d:LD}) after restricting the diagram to the stalk of $(0,0)\in\tcB\subset\tcA$.

\subsubsection{The perverse cohomology} Let $\Delta\subset\calB$ be the locus where the discriminant of $x^3+x^2+bx+c$ vanishes, and let $\tilDel\subset\tcB$ be its preimage. Let $j$ be the open inclusion $\tcB-\tilDel\hookrightarrow\tcB$. Let $i_{0,0}:\{(0,0)\}\hookrightarrow\calB$ be the inclusion. The same argument for the Support Theorem \ref{th:supp} applies to the fibration $M\to\calB$, and gives the following decomposition (again we ignore Tate twists)
\begin{equation*}
\coho{*}{M/\tcB}=\Ql\oplus j_{!*}L[-1]\oplus\Ql[-2].
\end{equation*}
where $L$ is a rank 2 local system on $\tcB-\tilDel$ (we abuse the notation $j_{!*}$ to denote the middle extension of a shifted perverse sheaf). Along $\tilDel-\{(0,0)\}$, the stalks of $j_{!*}L$ are one dimension (concentrated at degree 0) corresponding to $\cohog{1}{\MHit_{\iota(b,c)}}$ when $Y_{b,c}$ is a nodal curve.

The fiber of $M$ over $(0,0)\in\tcB$ is the reduced structure of the fiber over $(0,0)\in\calB$. We use the notation $M_{0,0}$ to mean either of them, as long as we care only about the topology of them. Since $\cohog{2}{M_{0,0}}$ is two dimensional, and $\Ql[-2]$ only contributes one dimension to it, we must have $\dim\cohog{1}{i^*_{0,0}j_{!*}L}=1$. In other words, one dimension of $\cohog{2}{M_{0,0}}$ comes from a lower piece of the perverse filtration of $\coho{*}{M/\tcB}$. Since the action of $\alpha^\vee$ preserves the perverse filtration, $\cohog{1}{i^*_{0,0}j_{!*}L}\subset\cohog{2}{M_{0,0}}$ is invariant under $\alpha^\vee$. Since $v^*(\zeta_0-\zeta_1)$ spans the unique eigenspace of $\alpha^\vee$, we conclude that
\begin{equation}\label{eq:H1j}
\cohog{1}{i^*_{0,0}j_{!*}L}=\Ql\cdot v^*(\zeta_0-\zeta_1)\subset\cohog{2}{M_{0,0}}.
\end{equation}
By the last matrix in (\ref{eq:matrix}), the action of $\alpha^\vee-\id$ on $\cohog{2}{M_{0,0}}$ induces an isomorphism
\begin{equation}\label{eq:uni}
\alpha^\vee-\id:\coho{2}{i^*_{0,0}\Ql[-2]}\cong\cohog{2}{M_{0,0}}/\cohog{1}{i^*_{0,0}j_{!*}L}\isom\cohog{2}{i^*_{0,0}j_{!*}L[-1]}.
\end{equation}
This is the stalk at $(0,0)$ of the ``subdiagonal'' Springer action (see (\ref{eq:defSp})):
\begin{equation*}
\Sp^2(\alpha^\vee):\Ql\to j_{!*}L[1].
\end{equation*}
Here we change the degree labeling from the perverse degree to the usual cohomological degree.

\subsubsection{The dual parabolic Hitchin fiber} To check the result of Th. \ref{th:LD}, we also need to consider the parabolic Hitchin fiber for $\PGL(2)$. Let $\Md/\calB$ be the restriction of $\Mpar_{\Gd}$ to $\calB$. For each $(b,c,\tilx)\in\calB$ (where $\tilx\in Y_{b,c}$ is a point over $\{0\}$), $\Md_{b,c,\tilx}$ classifies pairs $(\calF\supset\calF')$ up to tensoring line bundles from $X$. Here $\calF,\calF'$ are torsion-free rank one coherent sheaves of $Y_{b,c}$, such that $\calF/\calF'$ is of length 1 supported $\tilx\in Y_{b,c}$. Then $\Md$ is the disjoint union of $\Mde$ and $\Mdo$ according to the parity of the degree of $\calF$. Moreover, $\Mde$ is canonically isomorphic to $M$ over $\tcB$. Tensoring with $\calO_{Y}(\tilinf)$ identifies the two components $\Mde$ and $\Mdo$. Under these identifications,
\begin{equation*}
\cohog{*}{\Md/\tcB}_{\st}\subset\cohog{*}{\Mde/\tcB}\oplus\cohog{*}{\Mdo/\tcB}\cong\cohog{*}{M/\tcB}^{\oplus2}
\end{equation*}
is the diagonal. In particular, both projections give the same isomorphism
\begin{equation*}
\cohog{*}{\Md/\tcB}_{\st}\cong\cohog{*}{M/\tcB}_{\st}=\cohog{*}{M/\tcB}.
\end{equation*}

Let $\Md_{0,0}$ denote the fiber of $\Md$ over $(0,0)\in\tcB$. Then it consists of two connected components $\Mde_{0,0}$ and $\Mdo_{0,0}$, each of which is identified with $M_{0,0}$ as above. In particular, we can talk about the $C_0$ and $C_1$ components of $\Mde_{0,0}$ and $\Mdo_{0,0}$.

The line bundle $\calL_{-\alpha^\vee}$ on $M$ corresponding to the root $-\alpha^\vee\in\xch(\Td)$ assigns to each $(\calF\supset\calF')$ the line $\calF/\calF'$. On the $C_0$ component of either $\Mde_{0,0}$ or $\Mdo_{0,0}$, $\calF'$ is fixed, hence $\calL_{-\alpha^\vee}|_{C_0}\cong\calO_{C_0}(-1)$. On the $C_1$ component of either $\Mde_{0,0}$ or $\Mdo_{0,0}$, $\calF$ is fixed, hence $\calL_{-\alpha^\vee}|_{C_1}\cong\calO_{C_1}(1)$. Here we identify $C_0$ and $C_1$ with $\PP^1$. Therefore
\begin{equation}\label{eq:chSL2}
c_1(\calL_{-\alpha^\vee})=\frac{1}{2}v^*(\zeta_1-\zeta_0)\in\cohog{2}{\Md_{0,0}}_{\st}\cong\cohog{2}{M_{0,0}}.
\end{equation}

\subsubsection{The verification} Finally, we check the commutative diagram (\ref{d:LD}) at the point $(0,0)\in\tcB$ up to scalar. The only nontrivial degree of the outer square of the diagram (\ref{d:LD}) in this case reads (after shifting and twisting):
\begin{equation*}
\xymatrix{\Ql\ar[rr]^{\Sp^2(\alpha^\vee)}\ar[d]^{\wr} && j_{!*}L[1]\ar[d]^{\wr}\\
\Ql\ar[rr]^{\Ch^0(-\alpha^\vee)} && j_{!*}L[1]}
\end{equation*}
Restricting to $(0,0)\in\tcB$, by (\ref{eq:uni}) and (\ref{eq:chSL2}), both arrows take the form
\begin{equation*}
i^*_{0,0}\Ql\to i_{0,0}^*j_{!*}L[1]
\end{equation*}
which is given by (up to scalar) the class $v^*(\zeta_1-\zeta_0)\in\cohog{1}{i^*_{0,0}j_{!*}L}$. This verifies the stalk of the diagram (\ref{d:LD}) at $(0,0)\in\tcB$ up to scalar.

\appendix


\section{The endoscopic correspondences (following B-C.Ng\^o)}\label{s:endocorr}
This appendix is based on unpublished work of Ng\^o. We construct and study various correspondences between (parabolic) Hitchin moduli stacks for $G$ and its endoscopic group $H$. The results here are used to prove Th. \ref{th:endo} in Sec. \ref{ss:pfendo}.

Throughout the appendix, we fix a rigidified endoscopic datum $(\kappa,\rho)$, hence the endoscopic group scheme $H$.

\subsection{The endoscopic correspondence}

We first construct a correspondence between $\MHit_H$ and $\calA_H\times_{\calA}\MHit$.

\begin{cons}\label{cons:endocorr}
Using the global Kostant sections of $\MHit$ and $\MHit_{H}$ (which we fixed once and for all in Th. \ref{th:endo}), we get identifications
\begin{equation}\label{eq:trivMHit}
\MHitreg\cong\calP;\hspace{1cm}\MHitreg_H\cong\calP_H.
\end{equation}

Recall $\mu_{H}$ is the morphism $\calA_{H}\to\calA$. By \cite[4.15.2]{NgoFL}, there is an exact sequence of Picard stacks over $\calA_{H}$:
\begin{equation}\label{eq:exactcalP}
1\to\calR^G_{H}\to\mu_H^*\calP\xrightarrow{h_{\calP}}\calP_{H}\to1
\end{equation}
where the kernel $\calR^G_{H}$ is a commutative affine group scheme over $\calA_{H}$. Using the identifications (\ref{eq:trivMHit}), $h_{\calP}$ induces a morphism
\begin{equation*}
h_{\calM}:\calA_H\times_{\calA}\MHitreg\cong\mu_H^*\calP\xrightarrow{h_{\calP}}\calP_H\cong\MHitreg_H.
\end{equation*}
The graph $\Gamma(h_{\calM})$ of $h_{\calM}$ is a closed substack of $\MHitreg_H\times_{\calA_H}(\calA_H\times_{\calA}\MHitreg)$. Let $\calC_H$ be its closure in $\MHit_H\times_{\calA_H}(\calA_H\times_{\calA}\MHit)$, viewed as a correspondence
\begin{equation}\label{d:endocorr}
\xymatrix{&\calC_H\ar[dl]_{\overleftarrow{c_H}}\ar[dr]^{\overrightarrow{c_H}} &\\
\MHit_{H}\ar[dr]_{\fHit_{H}} & & \calA_H\times_{\calA}\MHit\ar[dl]^{\id\times\fHit}\\
& \calA_H}.
\end{equation}
\end{cons}

\begin{defn}\label{def:endocorr}
The correspondence $\calC_H$ between $\MHit_H$ and $\calA_H\times_{\calA}\MHit$ over $\calA_H$ is called the {\em endoscopic correspondence} associated to the pair $(G,H)$.
\end{defn}

We can similarly define a parabolic version of the endoscopic correspondence. Let
\begin{eqnarray*}
\tilf_\Theta:\Mpar_\Theta:=\Mpar\times_XX_\Theta\to\tcA_\Theta;\\
\tilf_{H,\Theta}:\Mpar_{H,\Theta}:=\Mpar_{H}\times_XX_\Theta\to\tcA_{H,\Theta}
\end{eqnarray*}
be the base changes of $\tilf$ and $\tilf_H$.

\begin{cons}\label{cons:parcorr} Recall the regular loci $\Mparreg\subset\Mpar$ and $\Mparreg_{H}\subset\Mpar_H$ from \cite[Lem. 3.2.7]{GSI}. Let $\Mparreg_\Theta$ and $\Mparreg_{H,\Theta}$ be their base changes from $X$ to $X_\Theta$. Then (\ref{eq:trivMHit}) gives isomorphisms
\begin{equation}\label{eq:trivMrho}
\Mparreg_\Theta\cong\tcA_{\Theta}\times_{\calA}\calP;\hspace{1cm}\Mparreg_{H,\Theta}\cong\tcA_{H,\Theta}\times_{\calA_H}\calP_H.
\end{equation}
Using these isomorphisms, the morphism $h_{\calP}$ in (\ref{eq:exactcalP}) induces a morphism over $\tcA_{H,\Theta}$:
\begin{equation*}
\tilh_{\Theta}:\tcA_{H,\Theta}\times_{\tcA_\Theta}\Mparreg_{\Theta}\to\Mparreg_{H,\Theta}.
\end{equation*}
Let $\Gamma(\tilh_{\Theta})$ be the graph of $\tilh_{\Theta}$, which is a closed substack of 
$\Mparreg_{H,\Theta}\times_{\tcA_{H,\Theta}}(\tcA_{H,\Theta}\times_{\tcA_\Theta}\Mparreg_{\Theta})=\Mparreg_{H,\Theta}\times_{\tcA_\Theta}\Mparreg_\Theta$. Let $\tcC_{H,\Theta}$ be the closure of $\Gamma(\tilh_{\Theta})$ in $\Mpar_{H,\Theta}\times_{\tcA_\Theta}\Mpar_\Theta$, viewed as a correspondence
\begin{equation*}
\xymatrix{&\tcC_{H,\Theta}\ar[dl]_{\overleftarrow{c_{H,\Theta}}}\ar[dr]^{\overrightarrow{c_{H,\Theta}}} &\\
\Mpar_{H,\Theta}\ar[dr]_{\tilmu_{H,\Theta}\circ\tilf_{H,\Theta}} & & \Mpar_\Theta\ar[dl]^{\tilf_{\Theta}}\\
& \tcA_{\Theta}}
\end{equation*}
\end{cons}

\begin{defn}\label{def:parendo} The correspondence $\tcC_{H,\Theta}$ between $\Mpar_{H,\Theta}$ and $\Mpar_\Theta$ over $\tcA_\Theta$ is called the {\em parabolic endoscopic correspondence} associated to $(G,H)$. In fact, $\tcC_{H,\Theta}$ lies over $\tcA_{\kappa,\Theta}\subset\tcA_{\Theta}$.
\end{defn}

From the above constructions, it is easy to verify:
\begin{lemma}\label{l:tcCUisC}
There is an isomorphism of correspondences between $\Mparrs_{H,\Theta}$ and $\Mparrs_{\Theta}$ over $\tcArs_{H,\Theta}$:
\begin{equation*}
\tcArs_{H,\Theta}\times_{\tcA_\Theta}\tcC_{H,\Theta}\cong\tcArs_{H,\Theta}\times_{\calA_H}\calC_H.
\end{equation*}
\end{lemma}

\subsection{The modular endoscopic correspondence}
In his unpublished work, Ng\^o also suggested another correspondence between $\MHit_{H}$ and $\MHit$ over $\calA$, which has a modular interpretation. Let us recall his construction. In the following discussion, we first work without the presence of the curve $X$. Therefore, we view $G,H$ as group schemes over $\BB\Theta$, with the $G$ the constant group and $H=H^{\textup{sp}}/\Theta$ given by the action of $\Theta$ on the split group $H^{\textup{sp}}$ via $\Theta\xrightarrow{\rho}\pi_0(\kappa)\to\Out(H^{\textup{sp}})$. All the stacks in the following Construction are over $\BB\Theta$.

\begin{cons}\label{cons:modcorr} Using the Kostant sections $\frc\to\frg^{\reg}$ and $\frc_H\to\frh^{\reg}$, we get isomorphisms
\begin{equation}\label{eq:Kos}
[\frg^{\reg}/G]\cong[\frc/J];\hspace{1cm}[\frh^{\reg}/H]\cong[\frc_H/J_H]
\end{equation}
where $J\to\frc$ and $J_H\to\frc_H$ are the regular centralizer group schemes for $G$ and $H$. Let $\mu_{\frc}:\frc_H\to\frc$ be the natural morphism. According to \cite[Prop. 2.5.1]{NgoFL} we have a homomorphism of group schemes over $\frc_H$
\begin{equation*}
h_J:\mu_{\frc}^*J\to J_H,
\end{equation*}
and hence a morphism
\begin{equation*}
\BB h_J:[\frc_H/\mu_{\frc}^*J]\to[\frc_H/J_H]. 
\end{equation*}
Using the identifications (\ref{eq:Kos}), we get a morphism
\begin{equation}\label{eq:hstack}
h:\frc_H\times_{\frc}[\frg^{\reg}/G]\to[\frh^{\reg}/H].
\end{equation}
Let
\begin{equation*}
\Gamma(h)=\frc_H\times_{\frc}[\frg^{\reg}/G]\xrightarrow{(h,\id)}[\frh^{\reg}/H]\times_{\frc_H}(\frc_H\times_{\frc}[\frg^{\reg}/G])=[\frh^{\reg}/H]\times_{\frc}[\frg^{\reg}/G]
\end{equation*}
be the graph of the morphism $h$. We can write $\Gamma(h)\cong[\frr^{\reg}/H\times G]$ where $\frr^{\reg}$ is quasi-affine over $\BB\Theta$ with an $H\times G$-action and a natural $H\times G$-equivariant morphism $\frr^{\reg}\to\frh^{\reg}\times_{\frc}\frg^{\reg}$. Let $\frr$ be the normalization of the affine closure of $\frr^{\reg}$. Therefore $\frr$ is an affine scheme over $\BB\Theta$ with a natural $H\times G$-equivariant morphism $\frr\to\frh\times_{\frc}\frg$. Hence, we can view $[\frr/H\times G]$ as a correspondence:
\begin{equation}\label{eq:corrfrr}
\xymatrix{&[\frr/H\times G]\ar[dl]_{\overleftarrow{r}}\ar[dr]^{\overrightarrow{r}} &\\
[\frh/H]\ar[dr]_{\chi_H} & & \frc_H\times_{\frc}[\frg/G]\ar[dl]^{\id\times\chi}\\
& \frc_H}
\end{equation}
\end{cons}

We can keep track of the $\GG_m$-action on $\frg$ and $\frh$ by homotheties in the above construction, so that $\frr$ also admits a natural $\GG_m$-action.

\begin{defn}\label{def:modendo}
The {\em modular endoscopic correspondence} $\Cmod_H$ is the stack which classifies pairs $(\psi,\iota)$ where $\psi:X\to[\frr/H\times G\times\GG_m]$ is a morphism and $\iota$ is a 2-isomorphism making the following diagram commutative:
\begin{equation*}
\xymatrix{& \ar@{}[dr]_{\Downarrow\iota} & [\frr/H\times G\times\GG_m]\ar[d]\\
X\ar[urr]^{\psi}\ar[rr]_{(X_\Theta,\rho_D)} & & \BB\Theta\times\BB\GG_m}
\end{equation*}
\end{defn}

According to \cite[Sec. 2]{NgoFib}, we also have moduli interpretations of $\MHit$ and $\MHit_H$ in the same style as Def. \ref{def:modendo}. For example, $\MHit_H$ classifies pairs $(\psi_H,\iota_H)$ where $\psi_H:X\to[\frh/H\times\GG_m]$ and $\iota_H$ is a 2-isomorphism between the morphism $X\xrightarrow{\psi_H}[\frh/H\times\GG_m]\to\BB\Theta\times\GG_m$ and the classifying morphism of the $\Theta\times\GG_m$-torsor $X_\Theta\times_X\rho_D$.

Using these moduli interpretations and diagram (\ref{eq:corrfrr}), $\Cmod_H$ can be viewed as a correspondence:
\begin{equation*}
\xymatrix{&\Cmod_H\ar[dl]_{\overleftarrow{m}}\ar[dr]^{\overrightarrow{m}} &\\
\MHit_{H}\ar[dr]_{\fHit_{H}} & & \calA_H\times_{\calA}\MHit\ar[dl]^{\id\times\fHit}\\
& \calA_H}
\end{equation*} 

\begin{lemma}[B-C.Ng\^o, unpublished]\label{l:Cmodproper}
The stack $\Cmod_H$ is an algebraic stack; the morphism $f^{\Hit}_H\circ\overleftarrow{m}=(\id\times f^{\Hit})\circ\overrightarrow{m}:\Cmod_H\to\calA_H$ satisfies the existence part of the valuative criterion, up to a finite extension. More precisely, for any complete discrete valuation ring $R$ containing $k$ with field of fractions $K$, and any commutative diagram
\begin{equation*}
\xymatrix{\Spec K\ar[d]\ar[r]^{c} & \Cmod_H\ar[d]\\
\Spec R\ar[r]^{a} & \calA_H}
\end{equation*}
there exists a finite separable extension $K'$ of $K$, with valuation ring $R'$, and a dotted arrow making the following diagram commutative
\begin{equation*}
\xymatrix{\Spec K'\ar[d]\ar[r] & \Spec K\ar[d]\ar[r]^{c} & \Cmod_H\ar[d]\\
\Spec R'\ar[r]\ar@{-->}[urr] & \Spec R\ar[r]^{a} & \calA_H}
\end{equation*}
\end{lemma}
\begin{proof}
The proof is similar to the existence part of the valuative criterion for the Hitchin fibration $f^{\Hit}:\MHit\to\calA$. The argument only uses the existence of a section $\frc_H\to[\frr/H\times G]$, the affineness and the normality of $\frr$.
\end{proof}

\subsection{The generic locus of the endoscopic correspondence}
We need an explicit description, also due to Ng\^o, of the endoscopic correspondence over the {\em generic} point of $\calA_H$. 

Fix a geometric generic point $a_H\in\calA_H$ with values in some algebraically closed field $k(a_H)$. We base change the diagram (\ref{d:endocorr}) to $a_H$. Let $a\in\calA(k(a_H))$ be the image of $a_H$ in $\calA$. Recall that we have a resultant divisor $\frR^G_H\subset\frc_H$ (cf. \cite[Lemme 1.10.2]{NgoFL}). Let $\frR(a_H)$ be the pull-back of the divisor $\frR^G_H$ to $X\otimes_kk(a_H)$ via $a_H$. Since $a_H$ is a generic point of $\calA_H$, $\frR(a_H)$ is a multiplicity-free divisor of degree $r=r_\kappa$. The restriction of the exact sequence (\ref{eq:exactcalP}) to $a_H$ becomes
\begin{equation}\label{eq:excalP}
1\to\prod_{v\in\frR(a_H)}\calR_{a_H,v}\to\calP_a\to\calP_{H,a_H}\to1
\end{equation}
where each $\calR_{a_H,v}$ is (non-canonically) isomorphic to $\GG_m$. Fix such an isomorphism for each $v\in\frR(a_H)$. In particular, (\ref{eq:excalP}) induces an isomorphism $\pi_0(\calP_a)\isom\pi_0(\calP_{a_H})$.

Fix a point $\tilx_H$ in the cameral curve $X^{\rs}_{a_H,\Theta}$, which determines a surjective homomorphism
\begin{equation}\label{eq:latticesurjpi0}
\xcoch(T)\twoheadrightarrow\pi_0(\calP_{a_H})\cong\pi_0(\calP_{a}).
\end{equation}
Fix a Kostant section for $\fHit_H$, which gives an isomorphism $\MHit_{H,a_H}\cong\calP_{H,a_H}\cong\calP_a/\prod_{v\in\frR(a_H)}\GG_m$.

\begin{lemma}[B-C.Ng\^o]\label{l:simpleC}
\begin{enumerate}
\item []
\item The correspondence $\calC_{H,a_H}$ is isomorphic to
\begin{equation}\label{d:CaH}
\xymatrix{& \calP_a\twtimes{\prod_v\GG_m}(\prod_{v\in\frR(a_H)}\PP^1)
\ar[dr]^{\overrightarrow{c}}\ar[dl]_{\overleftarrow{c}}\\
\calP_{H,a_H}\cong\calP_a/\prod_{v}\GG_m & & \MHit_{a}} 
\end{equation}
Here, for each $v\in\frR(a_H)$, the factor $\calR_{a_H,v}\cong\GG_m$ acts on the corresponding factor $\PP^1$ by homotheties on $\AA^1\subset\PP^1$.

\item Fix an ordering of the points in $\frR(a_H)$, so that we can write $\frR(a_H)=\{v_1,\cdots,v_{r}\}$. The correspondence $\calC_{H,a_H}$ is isomorphic to the composition:
\begin{equation}\label{d:compC}
\xymatrix{&\calC_1\ar[dl]_{\overleftarrow{c_1}} &\cdots & \calC_i\ar[dl]_{\overleftarrow{c_i}}\ar[dr]^{\overrightarrow{c_i}} & \cdots& \calC_r\ar[dr]^{\overrightarrow{c_{r}}}\\
\calM_0 & \cdots &\calM_{i-1} & &\calM_i &\cdots &\calM_{r}}
\end{equation}
where $\calM_0=\MHit_{H,a_H}$ and $\calM_{r}=\MHit_a$. The Picard stack $\calP_a$ acts on the above diagram, compatible with its action on the diagram (\ref{d:CaH}).

\item For $i=1,\cdots,r$, $\calC_i$ is a $\PP^1$-bundle over $\calM_{i-1}$, equipped with two sections $s^0_i,s^\infty_i:\calM_{i-1}\to\calC_i$ whose images $\calC^0_i$ and $\calC^\infty_i$ are disjoint. The morphism $\overrightarrow{c_i}$ is birational, mapping $\calC_i-\calC^0_i\cup\calC^\infty_i$ isomorphically onto an open subset of $\calM_i$, and identifies $\calC^0_i$ and $\calC^\infty_i$ via an isomorphism
\begin{equation*}
\gamma_i:\calC^0_i\isom\calC^\infty_i.
\end{equation*}

\item The diagram (\ref{d:compC}) induces bijections on the set of irreducible components:
\begin{equation*}
\pi_0(\calP_{H,a_H})\cong\Irr(\calM_0)\cong\Irr(\calM_1)\cong\cdots\cong\Irr(\calM_{r})\cong\pi_0(\calP_a).
\end{equation*}
Using the map (\ref{eq:latticesurjpi0}), all the above component groups are naturally quotients of $\xcoch(T)$. Then there exists a coroot $\beta^\vee_i\in\Phi^\vee-\Phi_H^\vee\subset\xcoch(T)$ (i.e., $\kappa(\beta^\vee_i)\neq1$), such that the following diagram is commutative:
\begin{equation*}
\xymatrix{\Irr(\calC^0_i)\ar[r]^{\Irr(\gamma_i)}_{\sim} & \Irr(\calC^\infty_i)\\ 
\Irr(\calM_{i-1})\ar[u]^{\Irr(s^0_i)}_{\wr}\ar[r]^{+\beta^\vee_i}_{\sim} & \Irr(\calM_{i-1})\ar[u]^{\Irr(s^\infty_i)}_{\wr}}
\end{equation*}
where the lower arrow labeled by $+\beta^\vee_i$ means the translation by the image of $\beta^\vee_i$ under the map $\xcoch(T)\twoheadrightarrow\Irr(\calM_{i-1})$.

\end{enumerate}
\end{lemma}


\begin{thebibliography}{99}

\bibitem[AL]{AL}
D.Alvis, G.Lusztig. On Springer’s correspondence for simple groups of type $E_n$ ($n=6,7,8$). Math. Proc. Camb. Phil. Soc. 92, pp. 65–72, 1982. 


\bibitem[B]{Bezr}
R.Bezrukavnikov, The dimension of the fixed point set on affine flag manifolds, Math. Res. Lett. 3(1996), no.2, 185--189.


%

\bibitem[BBD]{BBD}
A.Beilinson, J.Bernstein, P.Deligne, Faisceaux pervers, {\em Analysis and topology on singular spaces}, I (Luminy, 1981),  5--171, Ast\'erisque, 100, Soc. Math. France, Paris, 1982. 


\bibitem[DG]{DG}
R.Donagi, D.Gaitsgory, The gerbe of Higgs bundles. Transform. Groups 7 (2002), no. 2, 109--153. 


\bibitem[Ka]{Kaz}
D.Kazhdan, Proof of Springer's hypothesis. Israel J. Math.  28  (1977), no. 4, 272--286.


\bibitem[Ko]{Kot}
R.Kottwitz, Endoscopy for Hecke Algebras, Notes from the conference on Representation Theory of Real and
$p$-adic Reductive Groups, Seattle, 1997, available at \texttt{http://www.math.ust.hk/\~{}amoy/seattle97/index.html}.



\bibitem[KV]{KV}
D.Kazhdan, Y.Varshavsky, On endoscopic transfer of Deligne-Lusztig functions. arXiv:0902.3426


\bibitem[KW]{KW}
A.Kapustin, E.Witten,  Electric-magnetic duality and the geometric Langlands program. Commun. Number Theory Phys. 1 (2007), no. 1, 1--236. 




\bibitem[LM]{LM}
G.Laumon, L.Moret-Bailly, {\em Champs alg\'ebriques}. Ergebnisse der Mathematik und ihrer Grenzgebiete. 3. Folge. 39. Springer-Verlag, Berlin, 2000.


\bibitem[LN]{LauN}
G.Laumon, B-C. Ng\^{o}, Le lemme fondamental pour les groupes unitaires,  Ann. of Math. (2)  168  (2008),  no. 2, 477--573. 


\bibitem[N06]{NgoFib}
B-C.Ng\^{o}, Fibration de Hitchin et endoscopie. Invent. Math. 164 (2006), no.2, 399--453.


\bibitem[N08]{NgoFL}
B-C.Ng\^{o}, Le lemme Fondamental pour les alg\`{e}bres de Lie, arXiv:0801.0446.


\bibitem[S]{Sp}
T.A.Springer, Trigonometric sums, Green functions of finite groups and representations of Weyl groups. Invent. Math. 36 (1976),173--207.


\bibitem[T]{Tr}
D.Treumann, A topological approach to induction theorems in Springer theory, Represent. Theory 13 (2009), 8-18. 

\bibitem[YunI]{GSI}
Z.Yun, Towards a global Springer theory I: the affine Weyl group action, arXiv:0810.2146.


\bibitem[YunII]{GSII}
Z.Yun, Towards a global Springer theory II: the double affine action. Preprint.


\bibitem[Y]{Subreg}
Z.Yun, Subregular affine Springer fibers and subregular Hitchin fibers, in preparation.

\end{thebibliography}
\end{document}